\documentclass{article}
\usepackage[utf8]{inputenc}
\usepackage{fullpage}
\usepackage{amsthm,mathtools,scalerel}
\usepackage{color,soul,latexsym,amsmath,amssymb,amsfonts,amsthm,dsfont,enumitem,xcolor,bbm,threeparttable,graphicx,float,subfigure}
\usepackage{authblk}
\usepackage[round]{natbib}
\RequirePackage[colorlinks=true,allcolors=blue,hypertexnames=false]{hyperref}%
\usepackage{graphicx,todonotes}
\usepackage{outlines}
\usepackage{array}
\usepackage{selectp}

\usepackage[ruled,linesnumbered, vlined, noend]{algorithm2e}
\mathtoolsset{showonlyrefs}
\usepackage{titlesec}

\newtheorem{theorem}{Theorem}[section]
\newtheorem{proposition}{Proposition}[section]
\newtheorem{lemma}{Lemma}[section]
\newtheorem{corollary}{Corollary}[section]

\DeclareMathOperator*{\argmin}{argmin}

\newcommand{\E}{\mathbb{E}}

\newcommand{\bU}{\boldsymbol{U}}
\newcommand{\bX}{\boldsymbol{X}}
\newcommand{\bW}{\boldsymbol{W}}
\newcommand{\bY}{\boldsymbol{Y}}
\newcommand{\bS}{\boldsymbol{S}}

\def\E{\mathbb{E}}

\allowdisplaybreaks
\title{Maximal Inequalities for Independent Random Vectors}

\author[1]{Supratik Basu}

\author[2]{Arun Kumar Kuchibhotla}
\affil[1]{\texttt{supratik.basu@duke.edu}}
\affil[2]{{\tt arunku@cmu.edu}}
\affil[1]{Department of Statistical Science, Duke University}
\affil[2]{Department of Statistics \& Data Science, Carnegie Mellon University}

\begin{document}
\maketitle



\maketitle

\begin{abstract}
    Maximal inequalities refer to bounds on expected values of the supremum of averages of random variables over a collection. They play a crucial role in the study of non-parametric and high-dimensional estimators, and especially in the study of empirical risk minimizers. Although the expected supremum over an infinite collection appears more often in these applications, the expected supremum over a finite collection is a basic building block. This follows from the generic chaining argument. For the case of finite maximum, most existing bounds stem from the Bonferroni inequality (or the union bound). The optimality of such bounds is not obvious, especially in the context of heavy-tailed random vectors.
    In this article, we consider the problem of finding sharp upper and lower bounds for the expected $L_{\infty}$ norm of the mean of finite-dimensional random vectors under marginal variance bounds and an integrable envelope condition. 
\end{abstract}

\section{Introduction and Motivation}
Suppose $W_1, \ldots, W_n$ are independent random variables in some measurable space, taking values in $\mathcal{W}$. Let $\mathcal{F}$ be a collection of real-valued functions $f:\mathcal{W}\to\mathbb{R}$. The quantity
\[
\mathcal{E}_n(\mathcal{F}) ~:=~ \mathbb{E}\left[\sup_{f\in\mathcal{F}}\left|\frac{1}{n}\sum_{i=1}^n \{f(W_i) - \mathbb{E}[f(W_i)]\}\right|\right],
\]
plays a crucial role in the study of empirical risk minimization~\citep{VANG,vanderVaart98,MR4628026,MR4381414,bartl2025uniform}. To mention one specific example, suppose we have observations $(X_i, Y_i)$ from the non-parametric regression model $Y_i = \mu_0(X_i) + \xi_i$ with $X_i$ independent of $\xi_i$ and $\mu_0\in\mathcal{M}$ for some function class $\mathcal{M}$ equipped with the $L_2$-norm $\|\mu_1 - \mu_2\| = (\mathbb{E}[|\mu_1(X) - \mu_2(X)|^2])^{1/2}$. Consider the least squares estimator (LSE) on the function class $\mathcal{M}$:
\[
\hat{\mu}_n = \argmin_{\mu\in\mathcal{M}}\, \frac{1}{n}\sum_{i=1}^n (Y_i - \mu(X_i))^2.
\]
The results of~\cite{han2018robustness} and~\cite{han2017sharp} characterize the rates of convergence of the LSE using the behavior of
\[
\phi_n(\delta) = \mathbb{E}\left[\sup_{\mu\in\mathcal{M}:\,\|\mu - \mu_0\| \le \delta}\, \left|\frac{1}{n}\sum_{i=1}^n \xi_i(\mu - \mu_0)(X_i)\right|\right].
\]
Hence, optimal (upper and lower) bounds on $\phi_n(\delta)$ allow one to understand the precise rate of convergence of the LSE. Without the independence between $\xi_i$ and $X_i$, $\phi_n(\delta)$ can be difficult to control, as discussed in~\cite{MR4382017}, especially if $\xi_i$'s only have a finite $q$-th moment.

Even though the supremums mentioned above are over an infinite collection, they can be reduced to a finite collection using chaining arguments. A detailed description can be found in~\cite{MR4381414,bartl2025uniform,Dirksen}. The argument goes as follows: suppose $\mathcal{F}_0 \subset \mathcal{F}_1 \subset \cdots \mathcal{F}_k \subset \cdots \mathcal{F}$ is a collection of finite sets approaching $\mathcal{F}$ (as $k\to\infty$), with $|\mathcal{F}_0| = 1$ (i.e., a singleton). Assume, for simplicity, that $\mathbb{E}[f(W_i)] = 0$ for all $f\in\mathcal{F}$ and $1\le i\le n$. Also, for any $f\in\mathcal{F}$, let $\pi_k(f)$ denote the element in $\mathcal{F}_k$ that is ``closest'' to $f$; one can use any generic notion of closeness. With this notation, we can write
\begin{align*}
\mathcal{E}_n(\mathcal{F}) ~\le~ \sum_{j=0}^{\infty} \mathbb{E}\left[\sup_{f\in\mathcal{F}}\,\left|\frac{1}{n}\sum_{i=1}^n (\pi_{j+1}(f) - \pi_{j}(f))(W_i)\right|\right].
\end{align*}
The expected value on the right hand side is that of a finite supremum, because there are at most $|\mathcal{F}_j||\mathcal{F}_{j+1}|$ pairs $(\pi_{j+1}(f), \pi_j(f))$ as $f$ ranges over $\mathcal{F}$. On the other hand, note that 
\[
\mathcal{E}_n(\mathcal{F}) ~\ge~ \max_{j \ge 0}\,\mathbb{E}\left[\sup_{f\in\mathcal{F}_j}\left|\frac{1}{n}\sum_{i=1}^n f(W_i)\right|\right],
\]
which is also an expected value of a finite maximum of averages of random variables. Hence, the bounding $\mathcal{E}_n(\mathcal{F})$ can be solved by considering the expected values of a finite maximum of averages of random variables.
This is the main focus of the manuscript. Note that if $\mathcal{F}$ is a singleton, then $\mathcal{E}_n(\mathcal{F}) = O(1/\sqrt{n})$ whenever $\mbox{Var}(f(W)) < \infty$. Hence, for a scaling of $n^{-1/2}$ (in terms of sample size) for $\mathcal{E}_n(\mathcal{F})$, we need $\max_{f\in\mathcal{F}}\mbox{Var}(f(W)) < \infty$. With only a marginal variance condition, there exists a function class $\mathcal{F}$ such that $\mathcal{E}_n(\mathcal{F}) \ge c\max_{f\in\mathcal{F}}\sqrt{\mbox{Var}(f(W))}\sqrt{|\mathcal{F}|/n}$, for some absolute constant $c$. From the point of view of high-dimensional or nonparametric applications, we want bounds for $\mathcal{E}_n(\mathcal{F})$ that have a logarithmic dependence on the cardinality of $\mathcal{F}$. For this, we consider assumptions on the envelope function 
\[
F(w) := \sup_{f\in\mathcal{F}}\,|f(w)|.
\]
Hence, we consider the problem of bounding $\mathcal{E}_n(\mathcal{F})$ under a boundedness assumption of the marginal variable and integrability of the envelope function. We do not explicitly restrict $|\mathcal{F}| < \infty$, but our upper and lower bounds depend on $|\mathcal{F}|$ and show that the worst case bounds are infinity if $|\mathcal{F}| = \infty$.

The remaining article is organized as follows. In Section~\ref{sec:prob-description}, we formally describe the problem of obtaining sharp upper and lower bounds for the expected maximum and provide some simple reductions. In Section~\ref{sec:bounded-random-vectors}, we consider the special case of bounded random vectors. This is the most commonly considered case in the literature where one applies Hoeffding or Bernstein inequalities to obtain the maximal inequalities. Our results in Section~\ref{sec:bounded-random-vectors} improve upon these existing results and we also prove a lower bound. In Section~\ref{sec:bounds-q-moment}, we consider the general case of random vectors with a finite $q$-th moment on the envelope. As part of our analysis, we improve the Fuk-Nagaev inequalities refined in~\cite{MR3652041}, which maybe of independent interest in the study of heavy-tailed random variables.\footnote{We refer to random variables with infinite moment generating function as heavy-tailed random variables.} Finally, we conclude the article with some remarks in Section~\ref{sec:conclusions}.   
\section{Problem Description}\label{sec:prob-description}
For $q > 0$, we consider the problem of characterizing (up to universal constants)
\begin{equation}\label{eq:def_E_q_A_B}
\mathcal{E}_q(\sigma, B) := \sup_{P^n}\left\{\mathbb{E}_{P^n}\left[\max_{1\le j\le p}\left|\frac{1}{n}\sum_{i=1}^n X_i(j)\right|\right]:\,\mathcal{V}(P^n) \le \sigma^2\;\mbox{and}\; \mathcal{D}_q(P^n) \le B\right\},
\end{equation}
where the supremum is taken over all joint probability distributions $P^n$ of independent random vectors $\bX_1, \bX_2, \ldots, \bX_n$ satisfying $\mathbb{E}_{P^n}[\bX_i] = 0, 1\le i\le n$,
\begin{equation}\label{def:V_and_D_q}
\mathcal{V}(P^n) := \max_{1\le j\le p}\,\frac{1}{n}\sum_{i=1}^n \mathbb{E}[X_i^2(j)],\quad\mbox{and}\quad \mathcal{D}_q(P^n) := \left(\frac{1}{n}\sum_{i=1}^n \mathbb{E}\left[\max_{1\le j\le p}|X_i(j)|^q\right]\right)^{1/q}.
\end{equation}
For $q = \infty$, the constraint $\mathcal{D}_q(P^n) \le B$ should be understood as $\mathbb{P}_{P^n}(\max_{1\le i\le n}\|\bX_i\|_{\infty} > B) = 0$. For notational convenience, throughout, we use $\mathcal{P}_0^n$ to denote the collection of all joint probability distributions $P^n$ of independent random vectors $\bX_1, \ldots, \bX_n$ satisfying $\mathbb{E}_{P^n}[X_i] = 0, 1\le i\le n$. Note that $P^n\in\mathcal{P}_0^n$ are of the form $P^n = \prod_{i=1}^n P_i$ for some probability distributions $P_i$ on $\mathbb{R}^p$.
By characterize (up to universal constants), we mean to find a function $(n, p, q, \sigma, B)\mapsto L(n, p, q, \sigma, B)$ such that for some universal constants $\underline{C}, \overline{C}$,
\[
\underline{C}L(n, p, q, \sigma, B) \le \mathcal{E}_q(\sigma, B) \le \overline{C}L(n, p, q, \sigma, B)\quad\mbox{for all}\quad n, p \ge 1, q > 0, \sigma, B \ge 0.
\]

To characterize $\mathcal{E}_q(\sigma, B)$, we show that it suffices to consider the supremum over independent and identically distributed random vectors $\bX_1, \ldots, \bX_n$. Then, we show that the study of $\mathcal{E}_q(\sigma, B)$ is intricately related to that of bounded random vectors. For the first claim, we use Proposition 1 of~\cite{pruss1997comparisons} and for the second claim, we use Theorem 1.4.4 of~\cite{MR1666908}.

Define the sub-class of mean zero probability distributions $P$ on $\mathbb{R}^d$ as
\[
\mathcal{P}_q(\sigma, B) := \left\{P\mbox{ a probability distribution}:\, \mathbb{E}_P[\bX] = 0,\,\max_{1\le j\le p}\mathbb{E}_P[X^2(j)] \le \sigma^2,\,(\mathbb{E}_P[\|\bX\|_{\infty}^q])^{1/q} \le B\right\},
\]
and define the expectation as
\[
\mathcal{E}_q^*(\sigma, B) := \sup_{P}\left\{\mathbb{E}_P\left[\left\|\frac{1}{n}\sum_{i=1}^n \bX_i\right\|_{\infty}\right]:\,\bX_1, \bX_2, \ldots, \bX_n\overset{iid}{\sim} P,\,P\in\mathcal{P}_q(\sigma, B)\right\}.
\]
\vspace{0.2cm}

\begin{theorem}\label{thm:non-iid-to-iid-reduction}
    For any $n, p \ge 1, 0 \le \sigma \le B$,
    \[
    \mathcal{E}_q^*(\sigma, B) ~\le~ \mathcal{E}_q(\sigma, B) ~\le~ 16\mathcal{E}_q^*(\sigma, B).
    \]
\end{theorem}

The proof for this result can be found in Section \ref{pf:non-iid-to-iid-reduction} of the Supplementary material.

For the second claim proving that the study of $\mathcal{E}_q(\sigma, B)$ is related to that of bounded random vectors, define
\[
\mathcal{P}_{q,\infty}(\sigma, B_q, B_{\infty}) := \left\{P\in\mathcal{P}_q(\sigma, B_q):\, \mathbb{P}_P(\|\bX\|_{\infty} > B_{\infty}) = 0\right\}, 
\]
and
\[
\mathcal{E}_{q,\infty}^*(\sigma, B_q, B_{\infty}) := \sup\left\{\mathbb{E}_{P}\left[\left\|\frac{1}{n}\sum_{i=1}^n \bX_i\right\|_{\infty}\right]:\,\bX_1, \ldots, \bX_n\overset{iid}{\sim} P,\,P\in\mathcal{P}_{q,\infty}(\sigma, B_q, B_{\infty})\right\},
\]
For any distribution $P$ on $\mathbb{R}^d$, define
\[
T(P; n) := \inf\left\{t\in[0,\infty):\, \mathbb{P}_{P}(\|\bX\|_{\infty} > t) \le \frac{1}{8n}\right\}.
\]
Define
\begin{align*}
\tau_q(\sigma, B) &:= \sup\left\{T(P; n):\,P\in\mathcal{P}_{q}(\sigma, B)\right\},\\
\mathcal{M}_q(\sigma, B) &:= \sup\left\{\mathbb{E}_{P}\left[\max_{1\le i\le n}\|\bX_i\|_{\infty}\right]:\, \bX_1, \ldots, \bX_n\overset{iid}{\sim} P,\, P\in\mathcal{P}_{q}(\sigma, B)\right\},\\
\mathcal{M}_q'(\sigma, B) &:= \sup\left\{\mathbb{E}_{P}\left[\max_{1\le i\le n}\|\bX_i\|_{\infty}\boldsymbol{1}\{\|\bX_i\|_{\infty} > \tau_q(\sigma, B)\}\right]:\, \bX_1, \ldots, \bX_n\overset{iid}{\sim} P,\, P\in\mathcal{P}_{q}(\sigma, B)\right\}.
\end{align*}
\vspace{0.2cm}
\begin{theorem}\label{thm:reduction-to-bounded}
    There exists universal constants $\underline{C}\le \overline{C}$ such that for any $q \ge 1, n, d\ge1$,
    \begin{equation}\label{eq:Unbounded-to-bounded}
    \underline{C}\left[\mathcal{E}_{q,\infty}(\sigma, B, \tau_q(\sigma, B)) + \frac{1}{n}\mathcal{M}_q(\sigma, B)\right] \le \mathcal{E}_{q}(\sigma, B) \le \overline{C}\left[\mathcal{E}_{q,\infty}(\sigma, B, \tau_q(\sigma, B)) + \frac{1}{n}\mathcal{M}_q(\sigma, B)\right].
    \end{equation}
    One can take $\overline{C} = 128$ and $\underline{C} = 1/4$. Moreover, inequalities~\eqref{eq:Unbounded-to-bounded} continues to hold when $\mathcal{M}_q(\sigma, B)$ is replaced with $\mathcal{M}_q'(\sigma, B)$. 
\end{theorem}
The details of the proof are discussed in Section \ref{pf:reduction-to-bounded} of the Supplementary material.
The following proposition provides some inequalities relating $\mathcal{M}_q(\sigma, B)$ and $\tau_q(\sigma, B)$, a rigorous proof of which is provided in Section \ref{pf:M-q-sig-B} of the Supplementary Material.
\begin{lemma}\label{lem:M-q-sig-B}
    For any $q > 0$ and $\sigma, B > 0$, 
    \[
    \tau_q(\sigma, B) ~\le~ \left(1 - (1 - 1/(8n))^n\right)^{-1}\mathcal{M}_q(\sigma, B) ~\le~ 9\mathcal{M}_q(\sigma, B),
    \]
    and
    \[
    \frac{n^{1/q}}{(\log n)^{2/q}}\left(B\boldsymbol{1}_{\mathcal{S}}+\frac{p\sigma^2}{B}\boldsymbol{1}_{\mathcal{S}^c}\right)\lesssim\mathcal{M}_q(\sigma, B) ~\le~ n^{1/q}B
    \]
    where $\mathcal{S}=\{2p - 1 \ge 5(B/\sigma)^2/(1 + 2q)^{2/q}\}.$ 
\end{lemma}

\section{Optimal Bounds under Bounded Envelope}\label{sec:bounded-random-vectors}
In this section, we obtain matching upper and lower bounds for $\mathcal{E}_{\infty}(\sigma,B)$. To do so, we compute an upper bound for the expectation using Bennett's inequality~\citep{bennett1962probability,MR3707425}. It turns out that the bound so obtained also serves as a lower bound for the same when appropriately multiplied by a constant for all but one case where the value of $\sigma$ is too small as compared to the value of $B$. We will proceed by first analyzing the nature of this "Bennett bound" following which we will investigate the behaviour and provide sharp upper and lower bounds for $\mathcal{E}'_{\infty}(\sigma,B)$ where $\mathcal{E}'$ is just $\mathcal{E}$ restricted to appropriately shifted and scaled Bernoulli random variables. Keeping these results in mind, we will establish inequalities for the general case and then compare our results with the ones in ~\cite{blanchard2024tight}. 

\subsection{Analysis of the ``Bennett Bound"}
Observe that for random vectors $X_1, \ldots, X_n$ IID from a distribution $P\in\mathcal{P}_{\infty}(\sigma, B)$
\begin{align*}
\mathbb{E}\left[\max_{1\le j\le p}\,\left|\frac{1}{n}\sum_{i=1}^n X_i(j)\right|\right] &= \int_0^{B} \mathbb{P}\left(\max_{1\le j\le p}\,\left|\frac{1}{n}\sum_{i=1}^n X_i(j)\right| > t\right)dt\\
&\le \int_0^B \min\left\{1, \,\sum_{j=1}^p \mathbb{P}\left(\left|\frac{1}{n}\sum_{i=1}^n X_i(j)\right| > t\right)\right\}dt.
\end{align*}
Here $X_i(j), 1\le i\le n$ are IID real-valued mean zero random variables with variance bounded by $\sigma^2$ and the random variables are almost surely bounded by $B$. Hence, one can use any of the univariate concentration inequalities for bounded random variables to obtain an upper bound for $\mathcal{E}_{\infty}(\sigma, B)$. We use the Bennett bound in the following and show its optimality. First, we characterize the upper bounds obtained using Bennett's inequality. For this, we introduce some notation. Define
\[\Psi(x)=(1+x)\ln(1+x)-x \quad\text{ for }\quad x\ge 0.\] In addition, let $q^{\text{Benn}}(.)$ denote the upper bound provided to the tail probability of the random variable $\|\bar{\bX}\|_{\infty}$ obtained using an application of the Bennett's Inequality. We therefore have 
\[
q^{\text{Benn}}(t)=2p\exp\left(-\frac{n\sigma^2}{B^2}\Psi\left(\frac{tB}{\sigma^2}\right)\right) \text{ for }t\ge 0
\] 
and further define $p^{\text{Benn}}(t)=\min\{1,q^{\text{Benn}}(t)\}$ for any $t\ge 0$. See, for example, Proposition 3.1 of~\cite{MR3707425}.

We begin by presenting the following result on the nature of the Bennett bound.

\begin{theorem}\label{thm:benn-nature}
    For $n\in \mathbb{N}$, and $0<\sigma\le B$, define 
    \[
    A=\frac{\sigma^2}{B}\Psi^{-1}\left(\frac{B^2\ln (2p)}{n\sigma^2}\right)\text{ and }f(t)=\frac{q^{\mathrm{Benn}}(t)}{\ln(1+tB/\sigma^2)}\text{ , }t\ge 0.
    \]
    If $A \le B$, then
    \begin{equation}\label{eq:benn-nature-res}
        A+\frac{(\ln 2)^2}{4\sqrt{2}}\frac{B}{n}\left(f(A)-f(B)\right)\le \int_{0}^{B}p^{\mathrm{Benn}}(t)\,dt \le A+\frac{B}{n}(f(A)-f(B)).
    \end{equation} 
    If $A>B$, we have \[\int_{0}^{B}p^{\mathrm{Benn}}(t)\,dt=B.\]
    {In general, we have
    \[
    (A\wedge B) + \frac{(\ln 2)^2}{4\sqrt{2}}\frac{B}{n}(f(A\wedge B) - f(B)) ~\le~ \int_0^B p^{\mathrm{Benn}}(t)dt ~\le~ (A\wedge B) + \frac{B}{n}(f(A\wedge B) - f(B)).
    \]}
\end{theorem}

Theorem~\ref{thm:benn-nature} characterizes the upper bound that can be obtained from the Bennett bound and Bonferroni inequality. Theorem~\ref{thm:benn-nature} is split into two cases depending on the ordering of $A, B$. The following proposition provides some insight into these cases.
{
\begin{proposition}\label{prop:A>B}
    $A>B$ implies $n\sigma^2/(\sigma^2+B^2)\ge 1/(2p)$.
\end{proposition}
\begin{proof}
    $A > B$ is equivalent to
    \[
    \Psi^{-1}\left(\frac{B^2\ln(2p)}{n\sigma^2}\right) \ge \frac{B^2}{\sigma^2} \quad\Leftrightarrow\quad \frac{B^2\ln(2p)}{n\sigma^2} \ge \Psi\left(\frac{B^2}{\sigma^2}\right).
    \]
    This in turn implies that
    \[
    \frac{B^2}{\sigma^2}\left(1 + \frac{\ln(2p)}{n}\right) \ge (1 + B^2/\sigma^2)\log(1 + B^2/\sigma^2)
    \ge (B^2/\sigma^2)\log(B^2/\sigma^2).\]
    Note that the last quantity is non-negative as $B^2\ge \sigma^2$. This further implies that \[\left(1+\frac{\ln (2p)}{n}\right)\ge \log(B^2/\sigma^2)\quad\Leftrightarrow\quad B^2/\sigma^2\le e(2p)^{1/n}.\]
    If possible, let $n\sigma^2/(\sigma^2+B^2)< 1/(2p)$. This means that \[B^2/\sigma^2>2np-1.\] Combining this with the fact that $B^2/\sigma^2\le e(2p)^{1/n}$, we find that \[2np-1<e(2p)^{1/n}\Rightarrow 2np-1-e(2p)^{1/n}<0.\]It is easy enough to verify that $2np-1-e(2p)^{1/n}$ is increasing in $n$ and since $n\ge 2$, we must have $4p-1-e\sqrt{2p}<0$. Again this quantity is seen to be increasing in $p$, and is positive for $p\ge 2$, a contradiction. Hence for $p\ge 2$, we must have $n\sigma^2/(\sigma^2+B^2)\ge 1/(2p)$. When $p=1$, we have $\sigma^2\le B^2\le e\sqrt{2}\sigma^2$. This proves the proposition.
\end{proof}}

We now show that this upper bound is the best possible by constructing the random vectors $\bX_1,\bX_2,\ldots,\bX_n$ that have independent (centered) Bernoulli components such that the corresponding expected maximum is lower bounded by the Bennett upper bound up to some constants. The fact that Bennett's inequality upper bound is matched by Bernoulli random vectors is not surprising, given the results of~\cite{bentkus2004hoeffding}.

\subsection{Analysis for Random Vectors with  Bernoulli components}
Consider independent $p$-variate random vectors $\{\bW_i\}_{1\le i\le n}$ such that the components $\{W_{i,j}\}_{1\le j\le p}$ are independent for each $i\in [n]$. The random variables $W_{i,j}$ are such that 
\[
\mathbb{P}(W_{i,j}=w) = 
\begin{cases}
    {B^2}/(\sigma^2+B^2), & \text{ if }w=-{\sigma^2}/{B},\\
    {\sigma^2}/{(\sigma^2+B^2)}, & \text{ if }w=B\\
    0, & \text{ otherwise.}
\end{cases}
\]
Define $\overline{\bW}:=n^{-1}\sum_{i=1}^n \bW_i$ and analogously define $\overline{W}_j:=n^{-1}\sum_{i=1}^{n}W_{i}(j)$. We are using $\bW$ for these two point random variables to indicate that these are in some sense worst-case random variables. The following result provides matching upper and lower bounds for the expected $L_{\infty}$ norm of the mean of these worst case random vectors. 
\begin{theorem}\label{thm:ind-bern}
    For any $n, p\ge 1$ and $A\le B$, 
    \begin{equation}\label{eq:ind-bern-upper} \mathbb{E}\left[\|\overline{\bW}\|_{\infty}\right]=\mathbb{E}\left[\max_{1\le j\le p}\overline{W}_j\right]\le A+\frac{B}{n}(f(A)-f(B))
    \end{equation}
    Additionally,
    \begin{itemize}
        \item [(a)] When $n\sigma^2/(\sigma^2+B^2)\ge 1/(2p)$,
        \begin{equation}\label{eq:ind-bern-lower-small-B}
        \E\left[\left\|\overline{\bW}\right\|_{\infty}\right]\gtrsim A+\frac{B}{n}(f(A)-f(B)).
        \end{equation}
        \item [(b)] When $n\sigma^2/(\sigma^2+B^2)< 1/(2p)$,
        \begin{equation}\label{eq:ind-bern-large-B}
            \E\left[\left\|\overline{\bW}\right\|_{\infty}\right]\asymp \frac{p\sigma^2}{B}.
        \end{equation}
    \end{itemize}   
    When $A>B$, we have $\E\left[\left\|\overline{\bW}\right\|_{\infty}\right]\asymp B$.
\end{theorem}
The proof that heavily exploits the results of \cite{MR3196787} and can be found in Section \ref{pf:ind-bern} of the Supplement.
Let $\mathcal{E}_{\bW,\infty}(\sigma,B)$ be the quantity $\mathcal{E}_{\infty}(\sigma,B)$ where the components of the random vectors have a marginal distribution same as that of the $W_i(j)$'s. We do not however, assume anything about the correlation structure exhibited by the components of these vectors. We can still argue that the upper and lower bounds that we obtained for the $\bW_i$'s hold even when the components are not necessarily independent. To that end the following result will prove instrumental. 

\begin{theorem}\label{thm:extension-indep}
    Consider independent mean zero random vectors $\bX_1,\bX_2,\cdots, \bX_n$ such that for all $i=1,2,\cdots,n$
    \begin{equation*}
        \max_{1\le j\le p}|X_i(j)|\le B \text{ and } \max_{1\le j\le p}\mathcal{V}\left[X_i(j)\right]\le \sigma^2.
    \end{equation*}
    Then there exist independent random vectors $\bY_1, \bY_2,\cdots \bY_n$ such that $Y_i(1), Y_i(2), \cdots , Y_i(p)$ are jointly independent for all $i=1,2,\cdots, n$ such that 
    \begin{equation*}
        \max_{1\le j\le p}|Y_i(j)|\le B \text{  and } \max_{1\le j\le p}\mathcal{V}\left[Y_i(j)\right]\le \sigma^2
    \end{equation*}
    and \[\E\left[\|\overline{\bX}\|_{\infty}\right]\le 2\E\left[\|\overline{\bY}\|_{\infty}\right].\]
\end{theorem}

The following corollary is therefore an immediate consequence of Theorem~\ref{thm:ind-bern} and Theorem~\ref{thm:extension-indep} (proved in Section \ref{pf:extension-indep} of the Appendix).

\begin{corollary}\label{cor:bern-dep}
    For any $n,p\ge 1$ and $A\le B$, the following hold:
    \begin{itemize}
        \item [(a)] When $n\sigma^2/(\sigma^2+B^2)\ge 1/(2p)$,
        \[\mathcal{E}_{\bW,\infty}(\sigma,B)\asymp \min\left\{B,A+\frac{B}{n}(f(A)-f(B))\right\}.\]
        \item [(b)] When $n\sigma^2/(\sigma^2+B^2)< 1/(2p)$
        \[\mathcal{E}_{\bW,\infty}(\sigma,B)\asymp \frac{p\sigma^2}{B}.\]
    \end{itemize}
    In addition, when $A>B$, we have $\E\left[\left\|\overline{\bW}\right\|_{\infty}\right]\asymp B$.
\end{corollary}

\subsection{Optimality of the bounds on $\mathcal{E}_{\infty}(\sigma,B)$ for the general case}
The discussion in this section focuses on extracting the bounds for $\mathcal{E}_{\infty}(\sigma,B)$ when there are no restrictions on the distribution functions of the random variables $\bX_1,\bX_2,\ldots,\bX_n$ except those mentioned in ~\eqref{def:V_and_D_q}. As before, we shall divide our analysis into two parts: one where $n\sigma^2/(\sigma^2+B^2)\ge 1/(2p)$ and the other where $n\sigma^2/(\sigma^2+B^2)< 1/(2p)$. 

In the former case, by virtue of the proof of the first part of Theorem~\ref{thm:ind-bern}, we have \[A+\frac{B}{n}(f(A)-f(B))\lesssim \E\left[\left\|\bar{\bW}\right\|_{\infty}\right]\le \mathcal{E}_{\infty}(\sigma,B)\le A+\frac{B}{n}(f(A)-f(B)).\]
As for the latter case, we are unable to obtain such a sharp bound and have to therefore, settle for \[\frac{p\sigma^2}{B}\lesssim \mathcal{E}_{\infty}(\sigma,B)\le A+\frac{B}{n}(f(A)-f(B)).\]
This gives us the following result regarding the characterization of $\mathcal{E}_{\infty}(\sigma,B)$:
\begin{theorem}\label{thm:general-inf}
    (a) When $n\sigma^2/(\sigma^2+B^2)\ge 1/(2p)$, we have 
    \[\mathcal{E}_{\infty}(\sigma,B)\asymp \min\left\{B,A+\frac{B}{n}(f(A)-f(B))\right\}.\]
    (b) When $n\sigma^2/(\sigma^2+B^2)< 1/(2p)$,
    \[\frac{p\sigma^2}{B}\lesssim \mathcal{E}_{\infty}(\sigma,B)\lesssim \min\left\{B,A+\frac{B}{n}(f(A)-f(B))\right\}.\]
\end{theorem}

Note that $\mathcal{E}_{\infty}(\sigma,B)\le B$ follows by the Triangle Inequality and the fact that the components are all bounded by $B$. The result above reflects our findings for the Bernoulli variables except for the case where the ratio of the upper bound on the standard deviation of the variables and the bound on the variables shrinks to zero at a rate faster than $1/\sqrt{np}$ i.e. ${\sigma}/{B}=o\left((np)^{-1/2}\right).$

\subsection{A comparitive study with \cite{blanchard2024tight}}
In this section, we compare Theorem~\ref{thm:ind-bern} with the results of \cite{blanchard2024tight}, who considered the problem of bounding the expected $L_{\infty}$ norm of the mean of $n$ i.i.d. centred random vectors $\bX_1,\ldots,\bX_n$ with independent Bernoulli components. They have assumed that $X_i(j)\sim \text{Ber}(p_j)$ and have attempted to furnish bounds for the following quantity:
\[\Delta_n(\boldsymbol{p})=\E\|\boldsymbol{p_n}-\boldsymbol{p}\|_{\infty}\]
where $\boldsymbol{p_n}=\sum_{i=1}^{n}\bX_i/n$. This setting is similar in spirit to the one in Theorem~\ref{thm:ind-bern}. In light of this, it may be fruitful to observe that if we plug in $\boldsymbol{p}=\sigma^2/(\sigma^2+B^2)\boldsymbol{1}_{p}$, we obtain a setting identical to that of Theorem~\ref{thm:ind-bern}. The only difference here would be that the $\bX_i$'s are a scaled and shifted version of the random vectors $\bW_i$'s that we worked on. In particular, $\bX_i=(B\bW_i+\sigma^2)/(B^2+\sigma^2)$. The result that they have obtained is as follows:
\begin{itemize}
    \item [(1)] When $\boldsymbol{p}(j)\le 1/(2nj)$ for all $j$, then $\Delta_n(\boldsymbol{p})\asymp 1/n \wedge \sum_{j\ge 1}p(j)$. In our case, this translates to $np\sigma^2/(\sigma^2+B^2)\le 1/2$, in which case, this result takes the form $\Delta_n(\boldsymbol{p})\asymp 1/n \wedge p\sigma^2/(\sigma^2+B^2)$ which is exactly consistent with our result in part $(b)$ of Theorem~\ref{thm:ind-bern}.
    \item [(2)] For all other cases, \[\Delta_n(\boldsymbol{p})\asymp 1\wedge \sup_{j\ge 1}\left(\sqrt{\frac{p(j)\ln(j+1)}{n}}\vee \frac{\ln(j+1)}{n\ln(2+\frac{\ln(j+1)}{np(j)})}\right).\] 
\end{itemize}
Plugging in $p(j)=\sigma^2/(\sigma^2+B^2)$ for $1\le j\le p$, this result is equivalent to Theorem~\ref{thm:ind-bern} up to a universal constant.
Corollary 3 of \cite{blanchard2024tight} which pertains to an analysis of dependent Bernoulli random variables corresponds to Corollary~\ref{cor:bern-dep}. 

While their analysis is restricted to Bernoulli random variables, our analysis implies that Bernoulli random variables, in some sense, result in the worst case bound. In the following section, we also consider the case of unbounded random vectors. 


\section{Bounds under Integrable Envelope}\label{sec:bounds-q-moment}
This section focuses on acquiring upper bounds for $\mathcal{E}_{q}(\sigma,B)$ for $q \ge 2$. The bound we obtain is consistent with the bounds in the previous section in the sense that the upper bound obtained in this section converges to the one from Theorem~\ref{thm:general-inf} if $q = \infty$. An encouraging fact is that this convergence is monotonic up to constants. We use Theorem~\ref{thm:reduction-to-bounded} to obtain bounds for $\mathcal{E}_q(\sigma, B)$. For this purpose, we need to obtain concentration inequalities for averages of real-valued random variables that have a finite variance, finite $q$-th moment, and are bounded. 

Before we proceed further, we introduce some notation and definitions taken from Section 2 of ~\cite{MR3652041}.
We define the transformation $\mathcal{T}$ on the class $\Psi_c$ of convex functions $\psi: [0, \infty) \to [0, \infty]$ such that $\psi(0) = 0$ by
\[
\mathcal{T} \psi(x) = \inf \left\{ t^{-1}(\psi(t) + x) : t \in (0, \infty) \right\} \quad \text{for any } x \geq 0.
\]

We also define the integrated quantile function $\widetilde{Q}_X$ of the real-valued and integrable random variable $X$ as follows:
\[
\widetilde{Q}_X(u) = u^{-1} \int_0^u Q_X(s) \, ds
\]
where the function $Q_X$ is the cadlag inverse of the tail probability function of the random variable $X$. We now present a result that will be instrumental in obtaining an upper bound for $\mathcal{E}_q(\sigma,B)$. The following result is a refinement of the argument for Theorem 3.1(a) of~\cite{MR3652041}. 

\begin{theorem}\label{thm:prob-ell-q}
    Consider independent mean zero real-valued random variables $X_1,X_2,\ldots,X_n$ such that for some $q\ge 2$,
    \[\frac{1}{n}\sum_{i=1}^n\mathcal{V}[X_i]\le \sigma^2,\quad\frac{1}{n}\sum_{j=1}^n \E[|X_j|^q]\le B^q,\quad\mbox{and}\quad \max_{1\le j\le n}|X_j| \le K.\]
    Then for any $z>1$, 
    \[\mathbb{P}\left(\frac{1}{n}\sum_{i=1}^n X_i\ge B\left(\frac{\sigma^2}{B^2}\right)^{(q-1)/(q-2)}\Psi^{-1}\left[\left(\frac{B^2}{\sigma^2}\right)^{q/(q-2)}\frac{\ln z}{n}\right]+\frac{B^q}{K^{q-1}}\Psi^{-1}\left[\left(\frac{K}{B}\right)^{{q}}\frac{\ln z}{n}\right]\right)\le \frac{1}{z}.\]
\end{theorem}
Note that if $B^q = \sigma^2 K^{q-2}$, then Theorem~\ref{thm:prob-ell-q} reduces to
\[
\mathbb{P}\left(\frac{1}{n}\sum_{i=1}^n X_i\ge 2\frac{\sigma^2}{K}\Psi^{-1}\left(\frac{K^2\ln z}{n\sigma^2}\right)\right) \le \frac{1}{z},
\]
which is precisely the Bennett bound (except for the constant factor of $2$). Observe that the condition $B^q = \sigma^2K^{q-2}$ means that we do not have any information on the random variables other than mean zero, variance bounded by $\sigma^2$, and an almost sure bound of $K$. Also, if $q = 2$, then the condition becomes the same as Bennett's inequality, and Theorem~\ref{thm:prob-ell-q} becomes the Bennett's inequality.
The proof is inspired by the proof of Theorem 3.1 $(a)$ of \cite{MR3652041} and is detailed in Section \ref{pf:prob-ell-q} of the Supplement. 

The following is an alternative to Theorem~\ref{thm:prob-ell-q} that exhibits more desirable properties in connection to Bennett's inequality. 
\begin{theorem}\label{thm:prob-ell-q-version-2}
Consider independent mean zero real-valued random variables $X_1,X_2,\ldots,X_n$ such that for some $q\ge 2$,
    \[
    \frac{1}{n}\sum_{i=1}^n\mathcal{V}[X_i]\le \sigma^2,\quad\frac{1}{n}\sum_{j=1}^n \E[|X_j|^q]\le B^q,\quad\mbox{and}\quad \max_{1\le j\le n}|X_j| \le K.
    \]
    Then for any $z>1$, with probability at least $1 - 1/z$,
    \[
    \frac{1}{n}\sum_{i=1}^n X_i < B\left(\frac{\sigma^2}{B^2}\right)^{(q-1)/(q-2)}\Psi^{-1}\left(\left(\frac{B^2}{\sigma^2}\right)^{q/(q-2)}\frac{\ln z}{n}\right) + \frac{2K\ln z}{nq}\left(W\left(\frac{(K/B)^{q/([q]+1)}(\ln z)^{1/([q]+1)}}{10(nR)^{1/([q]+1)}}\right)\right)^{-1},
    \]
    where $R = (1 - (B^{q}/(\sigma^2K^{q-2}))^{1/(q-2)})_+$ and $W(\cdot)$ represents the Lambert function.
\end{theorem}
See Section \ref{pf:prob-ell-q-version-2} of the supplement for a proof of Theorem~\ref{thm:prob-ell-q-version-2}. Note that $R = 0$ if $B^q \ge \sigma^2K^{q-2}$ and in this case, the second term in the bound in Theorem~\ref{thm:prob-ell-q-version-2} becomes zero because the Lambert function is increasing and unbounded. Moreover, as $q \to \infty$, the argument of the Lambert function becomes $(K/B)/10$ and therefore the second term again becomes zero, because of the factor of $1/q$ for the second term. Hence, as $q\to\infty$ or $q = 2$, Theorem~\ref{thm:prob-ell-q-version-2} becomes the Bennett bound. 
To elaborate on the behavior of the second term, note that from Theorem 2.3 of~\cite{hoorfar2008inequalities}, we have
\[
W(x) \ge \ln(x) - \ln\left(\ln\left(\frac{x(1+1/e)}{\ln x}\right)\right),\quad\mbox{for all}\quad x > 0.
\]
This follows by taking $y = x/e$ in Theorem 2.3 of~\cite{hoorfar2008inequalities} and using the fact that $W(x) = x/e^{W(x)}$.

Theorems~\ref{thm:prob-ell-q} and~\ref{thm:prob-ell-q-version-2} can be used to obtain the bounds for $\mathcal{E}_q(\sigma, B)$ using Theorem~\ref{thm:reduction-to-bounded}. For simplicity, we use Theorem~\ref{thm:prob-ell-q} for the following result. See Section \ref{pf:e_q(s,b)-upper} of the supplement for a proof.
\begin{proposition}\label{thm:e_q(s,b)-upper}
    Fix any $\sigma,B>0$, and any $q\ge 2$. Then $\mathcal{E}_q(\sigma,B)\le \min\{B,U(q,\sigma,B,p,n)\}$, where
    \[
    U(q,\sigma,B,p,n)= \begin{dcases*}
        10 \sigma \sqrt{\frac{\ln (2p)}{n}} & if $(B^2/\sigma^2)^{{q}/({q}-2)}(\ln (2p)/n)\le e,$ \\
        \frac{5B(B^2/\sigma^2)^{1/({q}-2)}(\ln (2p)/n)}{\ln(1+(B^2/\sigma^2)^{{q}/({q}-2)}(\ln (2p)/n))} & if $(B^2/\sigma^2)^{{q}/({q}-2)}(\ln (2p)/n)> e$.
    \end{dcases*}
    \]
\end{proposition}


The bound provided in Proposition~\ref{thm:e_q(s,b)-upper} can be further improved for the case $(B^2/\sigma^2)^{{q}/({q}-2)}(\ln (2p)/n)> e$ as shown in the following result. This improvement stems from directly applying Bennett's inequality to the truncated random vectors (ignoring the $q$-th moment bound) and optimizing the truncation parameter. Although counterintuitive, this leads to a better bound than Theorem~\ref{thm:prob-ell-q}, which accounts for the $q$-th moment condition. 

\begin{proposition}\label{thm:lq-benn-upper}
    For $q\ge 2$, $n\ge 1$, $p\ge 1$ and $\sigma,B \ge 0$ such that $(B^2/\sigma^2)^{{q}/({q}-2)}(\ln (2p)/n)> e$ we have
    \begin{equation}\label{eq:better-bound-E_q-second-case}
    \mathcal{E}_q(\sigma,B)  \lesssim B\left(\frac{\ln (2p)}{n}\right)^{1-1/q}\left[\ln\left(\frac{B^2}{\sigma^2}\left(\frac{\ln (2p)}{n}\right)^{1-2/q}\right)\right]^{1/q-1}
\end{equation}
\end{proposition}
The proof is available in Section \ref{pf:lq-benn-upper} of the Supplementary material. Note that Proposition \ref{thm:lq-benn-upper} holds only for the case $(B^2/\sigma^2)^{{q}/({q}-2)}(\ln (2p)/n)> e$. 
For this reason, we combine it with Proposition \ref{thm:e_q(s,b)-upper} to obtain a very good upper bound encompassing all the possible cases which is stated in the following Corollary.

\begin{theorem}\label{cor:final-bd}
    For $\sigma,B>0$, and $q\ge 2$, we have $\mathcal{E}_q(\sigma,B)\lesssim \min\{B,U^*(q,\sigma,B,p,n)\}$
    \[U^*(q,\sigma,B,p,n)= \begin{dcases*}
         \sigma \sqrt{\frac{\ln (2p)}{n}} & if $(B^2/\sigma^2)^{{q}/({q}-2)}(\ln (2p)/n)\le e,$ \\
        B\left(\frac{\ln (2p)}{n}\right)^{1-1/q}\left[\ln\left(\frac{B^2}{\sigma^2}\left(\frac{\ln (2p)}{n}\right)^{1-2/q}\right)\right]^{1/q-1} &  if $(B^2/\sigma^2)^{{q}/({q}-2)}(\ln (2p)/n)> e.$
    \end{dcases*}\]
\end{theorem}

In this setting, the bound presented in Proposition B.1 of \citep{MR4382017} is of the order \[\sigma\sqrt{\frac{\ln(2p)}{n}}+B\left(\frac{\ln(2p)}{n}\right)^{1-1/q}.\] Note that this is of the order $\sigma\sqrt{\ln(2p)/n}$ when $(B^2/\sigma^2)^{{q}/({q}-2)}(\ln (2p)/n)\le e$ and the order $B(\ln(2p)/n)^{1-1/q}$ when $(B^2/\sigma^2)^{{q}/({q}-2)}(\ln (2p)/n)> e$. In the former case, our bound performs comparably with that of \cite{MR4382017}, while in the latter case, our bound performs better by the ploy-log factor in~\eqref{eq:better-bound-E_q-second-case}. This difference is more pronounced as $q$ increases, i.e., our bound always performs better and more so when $q$ is larger.

\subsection{Properties of the bound}
Having studied the upper bounds for $\mathcal{E}_q(\sigma,B)$, there are certain natural questions that arise namely how would the bound behave as $q\to \infty$. Is this bound optimal in some sense? In this section, we discuss such properties. It is observed that upper bound exhibits certain desirable properties. These are:
\begin{itemize}
    \item [(1)] The upper bound obtained is a continuous function of $q$ for $q\ge 2$.
    \item [(2)] The bound as a function of $q$ is decreasing in $q$ for $q\ge 2$ if $\ln(2p)/n$ does not diverge.
    \item [(3)] As $q\to \infty$, the bound converges to a quantity which has the same order as $\mathcal{E}_{\infty}(\sigma,B)$.
\end{itemize}

Property $(1)$ is easy to verify from the results we have obtained so far. To vouch for Property $(2)$, we present the following result.

\begin{theorem}\label{thm:mon-ell-q}
    If $\ln (2p)/n\le 1$, the upper bound on $\mathcal{E}_{q}(\sigma,B)$, obtained in Theorem \ref{cor:final-bd} is decreasing in $q$ (upto constants).
\end{theorem}

\begin{proof}
    Note that it suffices to establish that for $q,q'> 2$ such that $q<q'$, $U(q',\sigma,B,p,n)\le U(q,\sigma,B,p,n)$. Decreasingness at $q=2$ is then just a consequence of continuity. Towards proving decreasingness for $q>2$, we first define the collection $\mathcal{S}_q$ of sets of the tuples $(\sigma,B, p, n, q)$ such that \[\mathcal{S}_q=\left\{(\sigma,B,p,n,q): (B^2/\sigma^2)^{\frac{q}{q-2}}(\ln (2p)/n)> e\right\}\]
    It is clear that the collection is decreasing in $q$ i.e. $\mathcal{S}_{q'}\subset \mathcal{S}_{q}$. Hence the corresponding sets $\Omega \setminus \mathcal{S}_{q}$ are increasing in $q$. Further, observe that for any $b\ge a\ge 1/3$, we have \[\left(1-\frac{1}{q}\right)\ln(b/a)<\frac{1}{6}(b/a-1)+1\le 1+\frac{1}{2}\left(1-\frac{2}{q}\right)(b/a-1).\] Now, take \[a=({B^2}/{\sigma^2})\left({\ln (2p)}/{n}\right)^{1-2/q'}~~\text{and}~~b=({B^2}/{\sigma^2})\left({\ln (2p)}/{n}\right)^{1-2/q}.\]Thus, 
    \begin{align*}
        &\left(1-\frac{1}{q}\right)\ln(\ln(({B^2}/{\sigma^2})\left({\ln (2p)}/{n}\right)^{1-2/q}))-\left(1-\frac{1}{q'}\right)\ln(\ln(({B^2}/{\sigma^2})\left({\ln (2p)}/{n}\right)^{1-2/q'}))\\
        \le & \left(1-\frac{1}{q}\right)\ln(\ln(({B^2}/{\sigma^2})\left({\ln (2p)}/{n}\right)^{1-2/q}))-\left(1-\frac{1}{q}\right)\ln(\ln(({B^2}/{\sigma^2})\left({\ln (2p)}/{n}\right)^{1-2/q'}))+\frac{1}{3}\ln 3\\
        \le & 1+\frac{1}{2}\left(1-\frac{2}{q}\right)(\ln(({B^2}/{\sigma^2})\left({\ln (2p)}/{n}\right)^{1-2/q'})/\ln(({B^2}/{\sigma^2})\left({\ln (2p)}/{n}\right)^{1-2/q}))-1)+\frac{1}{3}\ln 3\\
        \le & 1+\frac{1}{2}\left(\ln(({B^2}/{\sigma^2})\left({\ln (2p)}/{n}\right)^{1-2/q})-\ln(({B^2}/{\sigma^2})\left({\ln (2p)}/{n}\right)^{1-2/q'}))\right)+\frac{1}{3}\ln 3\\
        \le & 1+\frac{1}{3}\ln 3-\left(\frac{1}{q}-\frac{1}{q'}\right)\ln(\ln(2p)/n)
    \end{align*}
    From this, we can say that the ratio of the upper bound on $\mathcal{E}_{q'}(\sigma,B)$ to the upper bound on $\mathcal{E}_{q}(\sigma,B)$ is bounded above by $4$. 
    This shows that on $S_{q'}$, the bound is decreasing. Further on $\Omega \setminus \mathcal{S}_{q}$, the bounds are of the same order. All we are left with is $(\Omega \setminus \mathcal{S}_{q'} )\cap \mathcal{S}_{q}$. For this case, the bound on $\mathcal{E}_{q'}(\sigma,B)$ is of the order $\sigma \sqrt{\ln (2p)/n}$ while that on $\mathcal{E}_{q}(\sigma,B)$ is of the order \[B\left({\ln (2p)}/{n}\right)^{1-1/q}\left[\ln\left(({B^2}/{\sigma^2})\left({\ln (2p)}/{n}\right)^{1-2/q}\right)\right]^{1/q-1}.\] On the set $(\Omega \setminus \mathcal{S}_{q'} )\cap \mathcal{S}_{q}$, these two are of the same order which shows that the upper bound on $\mathcal{E}_{q}(\sigma,B)$ is larger (upto constants and possibly even in order) than the upper bound on $\mathcal{E}_{q}(\sigma,B)$.
\end{proof}

As for Property $(3)$, we may write the bound $U_q$ on $\mathcal{E}_q(\sigma,B)$ as follows:
\[U_q= \sigma \sqrt{\frac{\ln (2p)}{n}}\boldsymbol{1}\left\{\mathcal{S}_q^c\right\}+B\left({\ln (2p)}/{n}\right)^{1-1/q}\left[\ln\left(({B^2}/{\sigma^2})\left({\ln (2p)}/{n}\right)^{1-2/q}\right)\right]^{1/q-1}\boldsymbol{1}\{\mathcal{S}_q\}\]
where 
\[\mathcal{S}_q=\{(B^2/\sigma^2)^{{q}/({q}-2)}(\ln (2p)/n)> e\}.\]
As we saw in the proof of Theorem~\ref{thm:mon-ell-q}, the $\mathcal{S}_q$'s are decreasing in $q$ and they decrease to \[\mathcal{S}_{\infty}=\{(B^2/\sigma^2)(\ln (2p)/n)> e\}\] while $\mathcal{S}_q^c\uparrow \mathcal{S}_{\infty}^c.$ Further, as $q\to \infty$, we must have \[B\left({\ln (2p)}/{n}\right)^{1-1/q}\left[\ln\left(({B^2}/{\sigma^2})\left({\ln (2p)}/{n}\right)^{1-2/q}\right)\right]^{1/q-1}\to \frac{B\ln (2p)/n}{\ln(B^2\ln (2p)/(n\sigma^2))}.\] It is easy to verify from here that the limiting bound that we obtain is of the same order as the bound we obtained on $\mathcal{E}_{\infty}(\sigma, B)$. In this sense, the upper bound we obtained for $\mathcal{E}_q(\sigma, B)$ may be considered optimal.

\section{Conclusions and Future Work}\label{sec:conclusions}
In this article, we studied the problem of characterizing (up to universal constants) the expected $L_{\infty}$ norm of the mean of independent random vectors by providing matching upper and lower bounds. Two conditions were imposed on the random vectors: the componentwise variances are finite (bounded above by $\sigma^2$) and the envelope has a finite $q$-th moment bounded by $B^q$. Such concentration inequalities play a fundamental role in empirical process theory. Our results demonstrate that it suffices to analyze the case of independent and identically distributed (i.i.d.) random vectors, and also to study bounded random vectors.

For bounded random vectors, it is common in the literature to use Bernstein's inequality. Our results imply a closed-form upper bound from Bennett's inequality and that it is optimal up to universal constants. It improved upon Bernstein's inequality by a log-factor, as can be expected. Except for the case where the variance of the random variables is too small compared to the upper bound on the random variables, our result is final. We also study the case where the envelope is not bounded but only has a finite $q$-th moment. Our results in this case are limited to upper bounds, which we show exhibit several desirable properties. We leave the question of lower bounds for future research.

Finally, our results can be applied to study or refine maximal inequalities for the supremum of empirical processes. For instance, some results of~\cite{MR4382017} can be improved by replacing their Proposition B.1 with our Theorem~\ref{cor:final-bd}. We leave these important applications to future research.

\bibliography{references} 

\begin{thebibliography}{}

\bibitem[Bartl and Mendelson, 2025]{bartl2025uniform}
Bartl, D. and Mendelson, S. (2025).
\newblock Uniform mean estimation via generic chaining.
\newblock {\em arXiv preprint arXiv:2502.15116}.

\bibitem[Batir, 2008]{batir2008sharp}
Batir, N. (2008).
\newblock Sharp inequalities for factorial $n$.
\newblock {\em Proyecciones (Antofagasta)}, 27(1):97--102.

\bibitem[Bennett, 1962]{bennett1962probability}
Bennett, G. (1962).
\newblock Probability inequalities for the sum of independent random variables.
\newblock {\em Journal of the American Statistical Association},
  57(297):33--45.

\bibitem[Bentkus, 2004]{bentkus2004hoeffding}
Bentkus, V. (2004).
\newblock On {H}oeffding’s inequalities.
\newblock {\em Annals of Probability}, 32(2):1650--1673.

\bibitem[Blanchard and Voracek, 2024]{blanchard2024tight}
Blanchard, G. and Voracek, V. (2024).
\newblock Tight bounds for local {G}livenko-{C}antelli.
\newblock In {\em International Conference on Algorithmic Learning Theory},
  pages 179--220. PMLR.

\bibitem[Csisz{\'a}r and Talata, 2006]{csiszar2006context}
Csisz{\'a}r, I. and Talata, Z. (2006).
\newblock Context tree estimation for not necessarily finite memory processes,
  via bic and mdl.
\newblock {\em IEEE Transactions on Information theory}, 52(3):1007--1016.

\bibitem[de~la Pe\~{n}a and Gin\'{e}, 1999]{MR1666908}
de~la Pe\~{n}a, V.~H. and Gin\'{e}, E. (1999).
\newblock {\em Decoupling}.
\newblock Probability and its Applications (New York). Springer-Verlag, New
  York.
\newblock From dependence to independence, Randomly stopped processes.
  $U$-statistics and processes. Martingales and beyond.

\bibitem[Dirksen, 2015]{Dirksen}
Dirksen, S. (2015).
\newblock Tail bounds via generic chaining.
\newblock {\em Electron. J. Probab.}, 20:no. 53, 1--29.

\bibitem[Gin{\'e} and Nickl, 2021]{gine2021mathematical}
Gin{\'e}, E. and Nickl, R. (2021).
\newblock {\em Mathematical foundations of infinite-dimensional statistical
  models}.
\newblock Cambridge Series in Statistical and Probabilistic Mathematics.
  Cambridge university press.

\bibitem[{Han} and {Wellner}, 2018]{han2018robustness}
{Han}, Q. and {Wellner}, J.~A. (2018).
\newblock {Robustness of shape-restricted regression estimators: an envelope
  perspective}.
\newblock {\em ArXiv preprints arxiv:1805.02542}.

\bibitem[{Han} and {Wellner}, 2019]{han2017sharp}
{Han}, Q. and {Wellner}, J.~A. (2019).
\newblock {Convergence rates of least squares regression estimators with
  heavy-tailed errors}.
\newblock {\em Ann. Statist.}, 47:2286 -- 2319.

\bibitem[Hoorfar and Hassani, 2008]{hoorfar2008inequalities}
Hoorfar, A. and Hassani, M. (2008).
\newblock Inequalities on the lambert w function and hyperpower function.
\newblock {\em J. Inequal. Pure and Appl. Math}, 9(2):5--9.

\bibitem[Kock and Preinerstorfer, 2024]{MR4747018}
Kock, A.~B. and Preinerstorfer, D. (2024).
\newblock A remark on moment-dependent phase transitions in high-dimensional
  {G}aussian approximations.
\newblock {\em Statist. Probab. Lett.}, 211:Paper No. 110149, 6.

\bibitem[Kuchibhotla and Patra, 2022]{MR4382017}
Kuchibhotla, A.~K. and Patra, R.~K. (2022).
\newblock On least squares estimation under heteroscedastic and heavy-tailed
  errors.
\newblock {\em Ann. Statist.}, 50(1):277--302.

\bibitem[Ledoux and Talagrand, 2011]{LED91}
Ledoux, M. and Talagrand, M. (2011).
\newblock {\em Probability in {B}anach spaces}.
\newblock Classics in Mathematics. Springer-Verlag, Berlin.
\newblock Isoperimetry and processes, Reprint of the 1991 edition.

\bibitem[Montgomery-Smith and Pruss, 2001]{montgomery2001comparison}
Montgomery-Smith, S.~J. and Pruss, A.~R. (2001).
\newblock A comparison inequality for sums of independent random variables.
\newblock {\em Journal of mathematical analysis and applications},
  254(1):35--42.

\bibitem[Pruss, 1997]{pruss1997comparisons}
Pruss, A.~R. (1997).
\newblock Comparisons between tail probabilities of sums of independent
  symmetric random variables.
\newblock In {\em Annales de l'Institut Henri Poincare (B) Probability and
  Statistics}, volume~33, pages 651--671. Elsevier.

\bibitem[Rio, 2017]{MR3652041}
Rio, E. (2017).
\newblock About the constants in the {F}uk-{N}agaev inequalities.
\newblock {\em Electron. Commun. Probab.}, 22:Paper No. 28, 12.

\bibitem[Talagrand, 2021]{MR4381414}
Talagrand, M. (2021).
\newblock {\em Upper and lower bounds for stochastic processes---decomposition
  theorems}, volume~60 of {\em Ergebnisse der Mathematik und ihrer
  Grenzgebiete. 3. Folge. A Series of Modern Surveys in Mathematics [Results in
  Mathematics and Related Areas. 3rd Series. A Series of Modern Surveys in
  Mathematics]}.
\newblock Springer, Cham, second edition.

\bibitem[van~de Geer, 2000]{VANG}
van~de Geer, S.~A. (2000).
\newblock {\em Applications of empirical process theory}, volume~6 of {\em
  Cambridge Series in Statistical and Probabilistic Mathematics}.
\newblock Cambridge University Press, Cambridge.

\bibitem[Van~der Vaart, 1998]{vanderVaart98}
Van~der Vaart, A.~W. (1998).
\newblock {\em Asymptotic statistics}, volume~3 of {\em Cambridge Series in
  Statistical and Probabilistic Mathematics}.
\newblock Cambridge University Press, Cambridge.

\bibitem[van~der Vaart and Wellner, 2023]{MR4628026}
van~der Vaart, A.~W. and Wellner, J.~A. (2023).
\newblock {\em Weak convergence and empirical processes---with applications to
  statistics}.
\newblock Springer Series in Statistics. Springer, Cham, second edition.

\bibitem[Wellner, 2017]{MR3707425}
Wellner, J.~A. (2017).
\newblock The {B}ennett-{O}rlicz norm.
\newblock {\em Sankhya A}, 79(2):355--383.

\bibitem[Zubkov and Serov, 2013]{MR3196787}
Zubkov, A.~M. and Serov, A.~A. (2013).
\newblock A complete proof of universal inequalities for the distribution
  function of the binomial law.
\newblock {\em Theory Probab. Appl.}, 57(3):539--544.

\end{thebibliography}
\bibliographystyle{apalike}
\newpage
\setcounter{section}{0}
\setcounter{equation}{0}
\setcounter{figure}{0}
\renewcommand{\thesection}{S.\arabic{section}}
\renewcommand{\theequation}{E.\arabic{equation}}
\renewcommand{\thefigure}{A.\arabic{figure}}
\renewcommand{\theHsection}{S.\arabic{section}}
\renewcommand{\theHequation}{E.\arabic{equation}}
\renewcommand{\theHfigure}{A.\arabic{figure}}
  \begin{center}
  \Large {\bf Supplement to ``Maximal Inequalities for Independent Random Vectors''}
  \end{center}
       
\begin{abstract}
This supplement contains the proofs of all the main results in the paper and some supporting lemmas. 
\end{abstract}


\section{Proofs of Results in Section 1}
\subsection{Proof of Theorem~\ref{thm:non-iid-to-iid-reduction}}\label{pf:non-iid-to-iid-reduction}
The lower bound of $\mathcal{E}_q^*(\sigma, B) \le \mathcal{E}_q(\sigma, B)$ is trivially true, because $P^{\otimes n} = \prod_{i=1}^n P$ for $P\in\mathcal{P}_q(\sigma, B)$ is a distribution in $\mathcal{P}_0^n$ satisfying $\mathcal{V}(P^{\otimes n}) \le \sigma^2$ and $\mathcal{D}_q(P^{\otimes n}) \le B$. To prove the upper bound, we show that for any $P^n\in\mathcal{P}_0^n$ satisfying $\mathcal{V}(P^n) \le \sigma^2$ and $\mathcal{D}_q(P^n) \le B$, there exists a distribution $P\in\mathcal{P}_q(\sigma, B)$ such that 
    \[
    \mathbb{E}_{P^n}\left[\left\|\frac{1}{n}\sum_{i=1}^n \bX_i\right\|_{\infty}\right] ~\le~ c\mathbb{E}_P\left[\left\|\frac{1}{n}\sum_{i=1}^n \bX_i\right\|_{\infty}\right].
    \]
    By symmetrization~\citep[Theorem 3.1.21]{gine2021mathematical}, we have
    \begin{equation}\label{eq:symmetrization}
    \mathbb{E}_{P^n}\left[\left\|\frac{1}{n}\sum_{i=1}^n \bX_i\right\|_{\infty}\right] \le 2\mathbb{E}_{P^n}\left[\left\|\frac{1}{n}\sum_{i=1}^n \varepsilon_i\bX_i\right\|_{\infty}\right],
    \end{equation}
    where $\varepsilon_i, 1\le i\le n$ are independent Rademacher random variables (i.e., $\mathbb{P}(\varepsilon_i = -1) = \mathbb{P}(\varepsilon_i = 1) = 1/2$). We can write the right hand side of~\eqref{eq:symmetrization} as $\mathbb{E}_{Q^n}[\|n^{-1}\sum_{i=1}^n Y_i\|_{\infty}]$ where $Y_i \overset{d}{=} \varepsilon_i\bX_i$ (and $Q_i$ is the distribution of $Y_i$). Proposition 1 of~\cite{pruss1997comparisons} implies that for any $Q^n = \prod_{i=1}^n Q_i$, there exists a distribution $\widetilde{Q}$ supported on $\mathbb{R}^p$ such that
    \begin{equation}\label{eq:symmetric-non-iid-to-iid}
    \mathbb{P}_{Q^n}\left(\left\|\sum_{i=1}^n Y_i\right\|_{\infty} > \lambda\right) \le 8\mathbb{P}_{\widetilde{Q}}\left(\left\|\sum_{i=1}^n \tilde{Y}_i\right\|_{\infty} > \lambda/2\right),
    \end{equation}
    where $\tilde{Y}_1, \ldots, \tilde{Y}_n\overset{iid}{\sim}\widetilde{Q}$. (Proposition 1 of~\cite{pruss1997comparisons} is only stated for real-valued random variables but the proof holds for random vectors as stated in~\cite{montgomery2001comparison}.) This implies
    \[
    \mathbb{E}_{Q^n}\left[\left\|\frac{1}{n}\sum_{i=1}^n X_i\right\|_{\infty}\right] ~\le~ 16\mathbb{E}_{\widetilde{Q}}\left[\left\|\frac{1}{n}\sum_{i=1}^n \widetilde{Y}_i\right\|_{\infty}\right].
    \]
    The distribution $\widetilde{Q}$ that satisfies~\eqref{eq:symmetric-non-iid-to-iid} is defined via
    \[
    \mathbb{E}_{\widetilde{Q}}[g(\widetilde{Y})] = \frac{1}{n}\sum_{i=1}^n \mathbb{E}_{Q^n}[g(Y_i)],\quad\mbox{for all bounded Borel measurable functions }g:\mathbb{R}^p\to\mathbb{R}.
    \]
    Because $\mathcal{V}(Q^n) \le \sigma^2$ and $\mathcal{D}_q(Q^n) \le B$, this implies that $\mathcal{V}(\widetilde{Q}^{\otimes n}) \le \sigma^2$ and $\mathcal{D}_q(\widetilde{Q}^{\otimes n}) \le B$. Therefore, $\widetilde{Q}\in\mathcal{P}_{q}(\sigma, B)$ and implies $\mathcal{E}_q(\sigma, B) \le 16\mathcal{E}_q^*(\sigma, B)$. \qed

\subsection{Proof of Theorem~\ref{thm:reduction-to-bounded}}\label{pf:reduction-to-bounded}
We prove the result by proving the equivalence for $\mathcal{E}_q^*(\sigma, B)$ and by Theorem~\ref{thm:non-iid-to-iid-reduction} the equivalence also holds for $\mathcal{E}_q(\sigma, B)$. 

    By triangle inequality, for any $t$ and distribution $P$, we have
    \[
    \mathbb{E}_P\left[\left\|\frac{1}{n}\sum_{i=1}^n X_i\right\|_{\infty}\right] \le \mathbb{E}_P\left[\left\|\frac{1}{n}\sum_{i=1}^n \bX_i\boldsymbol{1}\{\|\bX_i\|_{\infty} \le t\}\right\|_{\infty}\right] + \mathbb{E}_P\left[\left\|\frac{1}{n}\sum_{i=1}^n \bX_i\boldsymbol{1}\{\|\bX_i\|_{\infty} > t\}\right\|_{\infty}\right].
    \]
    Taking $t = \tau_q(\sigma, B)$, we get that for any $P\in\mathcal{P}_{q}(\sigma, B)$, the first term is bounded by $\mathcal{E}_{q,\infty}(\sigma, B, \tau_q(\sigma, B))$. For the second term, note that for any $P\in\mathcal{P}_q(\sigma, B)$,
    \begin{align*}
    \mathbb{P}_P\left(\left\|\frac{1}{n}\sum_{i=1}^n \bX_i\boldsymbol{1}\{\|\bX_i\|_{\infty} > \tau_q(\sigma, B)\}\right\|_{\infty} > 0\right) &\le \sum_{i=1}^n \mathbb{P}_P(\|\bX_i\|_{\infty} > \tau_q(\sigma, B)) \\
    &= n\mathbb{P}_P(\|\bX_i\|_{\infty} > \tau_q(\sigma, B)) \le 1/8.
    \end{align*}
    Hence, by Proposition 6.8 of~\cite{LED91}, for any $P\in\mathcal{P}_q(\sigma, B)$
    \begin{equation}\label{eq:proposition-6.8-application}
    \mathbb{E}_P\left[\left\|\frac{1}{n}\sum_{i=1}^n \bX_i\boldsymbol{1}\{\|\bX_i\|_{\infty} > \tau_q(\sigma, B)\}\right\|_{\infty}\right] ~\le~ \frac{8}{n}\mathbb{E}_P\left[\max_{1\le i\le n}\|\bX_i\|_{\infty}\right] \le \frac{8}{n}\mathcal{M}_q(\sigma, B).
    \end{equation}
    Therefore, by Theorem~\ref{thm:non-iid-to-iid-reduction},
    \[
    \mathcal{E}_q(\sigma, B) \le 16\mathcal{E}_q^*(\sigma, B) \le 16\left[\mathcal{E}^*_{q,\infty}(\sigma, B, \tau_q(\sigma, B)) + \frac{8}{n}\mathcal{M}_q(\sigma,B)\right].
    \]
    This completes the proof of the upper bound.

    To prove the lower bound, note that $\mathcal{P}_{q}(\sigma, B) \supset \mathcal{P}_{q,\infty}(\sigma, B, \tau_q(\sigma, B))$ and hence,
    \[
    \mathcal{E}_q^*(\sigma, B) \ge \mathcal{E}_{q,\infty}^*(\sigma, B, \tau_q(\sigma, B)).
    \]
    Moreover, for any $P\in\mathcal{P}_{q}(\sigma, B)$, by Theorem 1.1.1 of~\cite{MR1666908}
    \[
    \mathbb{E}_P\left[\max_{1\le i\le n}\|\bX_i\|_{\infty}\right] \le 2\mathbb{E}_P\left[\left\|\sum_{i=1}^n \bX_i\right\|_{\infty}\right].
    \]
    Taking the supremum over $P\in\mathcal{P}_{q}(\sigma, B)$ yields $n^{-1}\mathcal{M}_q(\sigma, B) \le 2\mathcal{E}_{q}^*(\sigma, B)$. Therefore, 
    \begin{align*}
    \mathcal{E}_q^*(\sigma, B) &\ge \max\left\{\mathcal{E}_{q,\infty}^*(\sigma, B, \tau_q(\sigma, B)),\, \frac{1}{2n}\mathcal{M}_q(\sigma, B)\right\}\\ 
    &\ge \frac{1}{2}\mathcal{E}_{q,\infty}^*(\sigma, B, \tau_q(\sigma, B)) + \frac{1}{4n}\mathcal{M}_q(\sigma, B).
    \end{align*}
    This completes the proof of lower bound.

    To prove that the final statement that~\eqref{eq:Unbounded-to-bounded} holds even when $\mathcal{M}_q(\sigma, B)$ is replaced with $\mathcal{M}'_q(\sigma, B)$, first note that $\mathcal{M}_q'(\sigma, B) \le \mathcal{M}_q(\sigma, B)$. This implies that the lower bound with $\mathcal{M}_q'(\sigma, B)$ readily holds from the lower bound in~\eqref{eq:Unbounded-to-bounded}. For the upper bound with $\mathcal{M}_q'(\sigma, B)$, we revisit the application of Proposition 6.8 of~\cite{LED91} for~\eqref{eq:proposition-6.8-application}. By Proposition 6.8 of~\cite{LED91}, we in fact obtain
    \begin{align*}
    \mathbb{E}_P\left[\left\|\frac{1}{n}\sum_{i=1}^n \bX_i\boldsymbol{1}\{\|\bX_i\|_{\infty} > \tau_q(\sigma, B)\}\right\|_{\infty}\right] ~&\le~ \frac{8}{n}\mathbb{E}_P\left[\max_{1\le i\le n}\|\bX_i\|_{\infty}\boldsymbol{1}\{\|\bX_i\|_{\infty} > \tau_q(\sigma, B)\}\right]\\ 
    ~&\le~ \frac{8}{n}\mathcal{M}_q'(\sigma, B).
    \end{align*}
    This completes the proof.\qed
\subsection{Proof of Lemma~\ref{lem:M-q-sig-B}}\label{pf:M-q-sig-B}
Observe that for any $P\in\mathcal{P}_q(\sigma, B)$ and $X_1, \ldots, X_n\overset{iid}{\sim} P$,
    \begin{align*}
    1 - \left(1 - \mathbb{P}_{P}(\|X_i\|_{\infty} > t\right)^n &= \mathbb{P}_{P}\left(\max_{1\le i\le n}\|X_i\|_{\infty} > t\right)\\
    &\le \frac{\mathbb{E}_{P}[\max_{1\le i\le n}\|X_i\|_{\infty}\boldsymbol{1}\{\|X_i\|_{\infty} > t\}]}{t} \le \frac{\mathcal{M}_q(\sigma, B)}{t}.
    \end{align*}
    Taking $t = (1 - (1 - 1/(8n))^n)^{-1}\mathcal{M}_q(\sigma, B)$, we get
    \[
    \mathbb{P}_P(\|X_i\|_{\infty} > t) \le \frac{1}{8n}\quad\mbox{for all}\quad P\in\mathcal{P}_{q}(\sigma, B).
    \]
    Therefore, $\tau_q(\sigma, B) \le (1 - (1 - 1/(8n))^n)^{-1}\mathcal{M}_q(\sigma, B)$ for all $n\ge1$. Because $(1 - (1 - 1/(8n))^n)^{-1} \le 8.52$ for all $n\ge1$, the stated inequality for $\tau_q(\sigma, B)$ follows.

    By an application of Jensen's inequality, we get
    \[
    \mathbb{E}_P\left[\max_{1\le i\le n}\|X_i\|_{\infty}\right] \le \left(n\mathbb{E}_P\left[\|X_i\|_{\infty}^q\right]\right)^{1/q} \le n^{1/q}B.
    \]
For any $q \ge 2$, define a random variable $G = G(q)$ via the distribution function as  
\begin{equation}\label{eq:heavy-tailed-CDF}
\mathbb{P}(G \le x) ~:=~ 
\begin{cases}
(1/2)|x|^{-q}(\max\{\ln(|x|), 1\})^{-2}, &\mbox{if }x \le -1,\\
1/2, &\mbox{if }x\in(-1, 1),\\
1 - (1/2)x^{-q}(\max\{\ln(x),\, 1\})^{-2}, &\mbox{if }x \ge 1.
\end{cases}
\end{equation}
Then $\mathbb{E}[G] = 0$ and $\mathbb{E}[|G|^q] = 1 + 2q$; see Eq. (9) of~\cite{MR4747018} for details. It is easy to see that $\mathbb{E}[|G|^{q + \delta}] = \infty$ for any $\delta > 0$. Also, for $0 \le \tau \le K$, define a two-point random variable $W = W(\tau^2, K)$ as
\[
\mathbb{P}\left(W = -\frac{\tau^2}{K}\right) ~=~ \frac{K^2}{K^2 + \tau^2},\quad\mbox{and}\quad \mathbb{P}(W = K) ~=~ \frac{\tau^2}{K^2 + \tau^2}.
\]
It is easy to verify that $\mathbb{E}[W] = 0$, $\mathbb{P}(|W| \le K) = 1$, and
\[
\mbox{Var}(W) = \frac{\tau^4}{K^2}\frac{K^2}{K^2 + \tau^2} + \frac{K^2\tau^2}{(K^2 + \tau^2)} = \tau^2.
\]
Finally, consider the random vector $H\in\mathbb{R}^d$ where $H(j) = G(q)W_j$ where $W_1, \ldots, W_p$ are independent identically distributed copies of $W(\tau^2, K)$ and independent of $G(q)$. Then for any $1\le j\le p$,
\[
\mathbb{E}[H(j)] = \mathbb{E}[G(q)]\mathbb{E}[W(\tau^2, K)] = 0,\quad
\mathbb{E}[H^2(j)] = \mathbb{E}[G^2(q)]\mathbb{E}[W^2(\tau^2, K)] = \mathbb{E}[G^2(q)]\tau^2,
\]
and
\begin{equation}\label{eq:q-moment-bound}
\begin{split}
(\mathbb{E}[\|H\|_{\infty}^q])^{1/q} &= (\mathbb{E}[|G(q)|^q])^{1/q}\left(\mathbb{E}\left[\max_{1\le j\le p}|W_j|^q\right]\right)^{1/q}\\ 
&= (1 + 2q)^{1/q}\left[\left(\frac{\tau^2}{K}\right)^q\left(\frac{K^2}{K^2 + \tau^2}\right)^p + K^q\left(1 - \left(\frac{K^2}{K^2 + \tau^2}\right)^p\right)\right]^{1/q}.
\end{split}
\end{equation}
This follows by noting that 
\[
\mathbb{P}\left(\max_{1\le j\le p}|W_j| = \frac{\tau^2}{K}\right) ~=~ 1 - \mathbb{P}\left(\max_{1\le j\le p}|W_j| = K\right) ~=~ \left(\frac{K^2}{K^2 + \tau^2}\right)^p,
\]
which in turn yields
\[
\mathbb{E}\left[\max_{1\le j\le p}|W_j|^q\right] = \left(\frac{\tau^2}{K}\right)^q\left(\frac{K^2}{K^2 + \tau^2}\right)^p + K^q\left(1 - \left(\frac{K^2}{K^2 + \tau^2}\right)^p\right).
\]
Observe that
\begin{align*}
\mathbb{E}[G^2(q)] &= \int_0^{\infty} \mathbb{P}(|G(q)| \ge x)dx = 1 + \int_1^{\infty} \mathbb{P}(|G(q)| \ge x^{1/2})dx \ge 1.
\end{align*}
Moreover, by Jensen's inequality (using $q \ge 2$)
\[
\mathbb{E}[G^2(q)] \le (\mathbb{E}[|G(q)|^q])^{2/q} \le (1 + 2q)^{2/q} \le 5.
\]
Taking $\tau^2 = \sigma^2/5$, we get $\mathbb{E}[H^2(j)] \le 5\tau^2 = \sigma^2.$ Moreover, take $K = K(\sigma; q, p)$ to be the unique solution to the equation
\begin{equation}\label{eq:K-solution}
(1 + 2q)\left[\left(\frac{\sigma^2}{5K}\right)^q\left(\frac{5K^2}{5K^2 + \sigma^2}\right)^p + K^q\left(1 - \left(\frac{5K^2}{5K^2 + \sigma^2}\right)^p\right)\right] ~=~ B^q.
\end{equation}
By Lemma \ref{lem:unique-soln}, the above equation has a unique solution in $K$. Having found the unique solution, we may write the following:
\[\E\left[\max_{1\le i\le n}\|H_i\|_{\infty}\right]\ge \E\left[\max_{1\le i\le n}|G_{i}(q)|\right]\max_{1\le i\le n}\E\left[\max_{1\le j\le p}|W_{i,j}|\right]=\E\left[\max_{1\le i\le n}|G_{i}(q)|\right]\E\left[\max_{1\le j\le p}|W_{j}|\right]\]
Now, 
\begin{align*}
    \E\left[\max_{1\le j\le p}|W_j|\right]&=\left(\frac{\tau^2}{K}\right)\left(\frac{K^2}{K^2+\tau^2}\right)^p+K\left(1-\left(\frac{K^2}{K^2+\tau^2}\right)^p\right)\\
    &= \left(\frac{\tau^2}{K}\right)+\left(K-\frac{\tau^2}{K}\right)\left(1-\left(\frac{K^2}{K^2+\tau^2}\right)^p\right)
\end{align*}
When $p\tau^2/(\tau^2+K^2)\ge 1/2$, then we have \[\frac{\tau^2}{\tau^2+K^2}\ge \frac{1}{2p}\implies 1-\left(\frac{K^2}{\tau^2+K^2}\right)^p\ge 1-\left(1-\frac{1}{2p}\right)^p\ge 1-\frac{1}{\sqrt{e}}.\]
When $p\tau^2/(\tau^2+K^2)< 1/2$, we have 
\[1-\left(\frac{K^2}{\tau^2+K^2}\right)^p\ge \frac{p\tau^2}{\tau^2+K^2}-\frac{1}{2}\left(\frac{p\tau^2}{\tau^2+K^2}\right)\ge \frac{1}{2}\left(\frac{p\tau^2}{\tau^2+K^2}\right).\]
From here it is easy to see that \[\E\left[\max_{1\le j\le p}|W_j|\right]\gtrsim \begin{cases}
    K, &\text{ if }p\tau^2/(\tau^2+K^2)\ge 1/2\\
    p\tau^2/K &\text{ if }p\tau^2/(\tau^2+K^2)< 1/2
\end{cases}\]
Now, in the equation \eqref{eq:K-solution}, it is easy to see that taking the left side as function $g$ of $K$, that $g$ is continuous and $g(B/(1+2q)^{1/q})\le B^q$ which means that the solution $K\ge B/(1+2q)^{1/q}$. For the choice $K^*=B/(1+2q)^{1/q}$, both $\operatorname{Var}\left(H_i\right)\le \sigma^2$ and $\E\left[\|H_i\|_{\infty}^q\right]\le B^q$ and hence, we may plug in $K=B/(1+2q)^{1/q}$ to obtain a lower bound. Thus, we have \[\E\left[\max_{1\le j\le d}|W_j|\right]\gtrsim \begin{cases}
    B, &\text{ if }2p - 1 \ge 5(B/\sigma)^2/(1 + 2q)^{2/q}\\
    p\sigma^2/B &\text{ if }2p - 1 < 5(B/\sigma)^2/(1 + 2q)^{2/q}
\end{cases}\]
We now compute a lower bound for the expected maximum of the $G_i$'s. Note that if $(2qn(\log n)^{-2})^{1/q}>e$,  
\begin{align*}
    \E\left[\max_{1\le j\le n}|G_i(q)|\right] &= \int_{1}^{e} 1-(1-x^{-q})^n\,dx+\int_e^{\infty} 1-(1-x^{-q}(\ln x)^{-2})^n\,dx\\
    &\ge \int_{e}^{(2qn(\log n)^{-2})^{1/q}}1-(1-x^{-q}(\ln x)^{-2})^n\,dx \\
    &\ge ((2qn(\log n)^{-2})^{1/q}-e)\left(1-\exp\left(-\frac{q}{2}\frac{(\log n)^2}{(\log n+\log (2q))^2}\right)\right)\\
    &\gtrsim (n(\log n)^{-2})^{1/q}.
\end{align*}
The penultimate inequality is a simple application of the fact that $1-x\le e^{-x}$ for $x\ge 0$ which leads to the fact that $1-(1-x)^n\ge 1-e^{-nx}$. The last inequality holds because the multiplier of $((2qn(\log n)^{-2})^{1/q}-e)$ is increasing in each of $n$ and $q$, so plugging in, $n=q=2$, we get $(1-e^{-1/9})$. When $(2qn(\log n)^{-2})^{1/q}\le e$, we need only show that the expected maximum is larger than a constant. This is done as follows. Consider the solution $z$ to the equation $1-e^{-x}/2=x$. Then \[1-e^{-nz}/2=1-(2(1-z))^n\le 1-(1-z)^n.\] This means that for $t=z^{-1/q}\le \sqrt{2}$. Thus we may write 
\begin{align*}
    \E\left[\max_{1\le j\le n}|G_i(q)|\right] & \ge \int_{\sqrt{2}}^{e} 1-(1-x^{-q})^n\,dx\\
    &\ge \int_{\sqrt{2}}^{e} 1-e^{-nx^{-q}}/2\,dx\\
    &\ge (e-\sqrt{2})(1-e^{-ne^{-q}}/2)\ge (e-\sqrt{2})/2\gtrsim (n(\log n)^{-2})^{1/q}.
\end{align*}
Taking $\mathcal{S}=\{2p - 1 \ge 5(B/\sigma)^2/(1 + 2q)^{2/q}\}$, we get \[\mathcal{M}_q(\sigma,B)\ge\E\left[\max_{1\le j\le n}\|H_i\|_{\infty}\right]\gtrsim (n(\log n)^{-2})^{1/q}\left(B\boldsymbol{1}_{\mathcal{S}}+\frac{p\sigma^2}{B}\boldsymbol{1}_{\mathcal{S}^c}\right).\]

\section{Proofs of Results in Section 2}
\subsection{Proof of Theorem~\ref{thm:benn-nature}}
Recall that we have \[p^{\text{Benn}}(t)= \min\left\{1, 2p \exp\left(-\frac{n\sigma^2}{B^2}\Psi\left(\frac{tB}{\sigma^2}\right)\right)\right\}.\]
In addition to this, an observation that $\Psi''(x)=1/(1+x)>0$ for all $x\ge 0$ enables us to apply Lemma~\ref{lem:mills-general} which gives us \[\int_{0}^Bp^{\text{Benn}}(t)\,dt-A\le \int_{A}^B 2p \exp\left(-\frac{n\sigma^2}{B^2}\Psi\left(\frac{tB}{\sigma^2}\right)\right)\,dt\le \frac{B}{n}(f(A)-f(B)).\]

    Before we proceed further, recall that $p^{\text{Benn}}(t) = \boldsymbol{1}\{t\le A\}+q^{\text{Benn}}(t)\boldsymbol{1}\{t>A\}$. Define the finite sequence $\ell_1, \ell_2, \cdots, \ell_{K+1}$ such that $\ell_1=A$ and in general, for $k\ge 2$,
    \begin{equation}\label{def:ell_k}
        \ell_k = \min\left\{t> A: 2p\exp\left(-\frac{n\sigma^2}{B^2}\Psi\left(\frac{tB}{\sigma^2}\right)\right)\le \frac{1}{2^{k-1}}\right\}=\frac{\sigma^2}{B}\Psi^{-1}\left(\frac{B^2\ln(2^kp)}{n\sigma^2}\right)
    \end{equation}
    Here $K$ is chosen such that $K = \max\{k\in \mathbb{N}:\ell_{k+1}\le B\}$. This allows us to rewrite $p^{\text{Benn}}(t)$ as
    \begin{align}\label{eq:p-Benn-char}
        p^{\text{Benn}}(t) = \boldsymbol{1}\{t\le A\}+\sum_{k=1}^{K}q^{\text{Benn}}(t)\boldsymbol{1}\{\ell_k<t\le \ell_{k+1}\}+q^{\text{Benn}}(t)\boldsymbol{1}\{\ell_{K+1}<t\le B\}.
    \end{align}
    From the definition of the $\ell_k$'s, it is evident that 
    \begin{equation}\label{eq:ellk-ineq}
        \frac{1}{2^{k}}(\ell_{k+1}-\ell_k)\le \int_{\ell_k}^{\ell_{k+1}}q^{\text{Benn}}(t)\,dt\le \frac{1}{2^{k-1}}(\ell_{k+1} - \ell_k)
    \end{equation}
    Appealing to the Lagrange's mean value theorem, for every $k$, there exists $\tau_k\in[0,1]$ such that \begin{align}
        \ell_{k+1} - \ell_{k} &= \frac{\sigma^2}{B}\left(\frac{B^2\ln(2^{k+1}p)}{n\sigma^2} - \frac{B^2\ln(2^{k}p)}{n\sigma^2}\right)\Big/{\Psi'\left(\Psi^{-1}\left(\frac{B^2\ln(2^{k+\tau_k}p)}{n\sigma^2}\right)\right)}\nonumber\\
        & = \frac{B\ln(2)}{n}\Big/{\Psi'\left(\Psi^{-1}\left(\frac{B^2\ln(2^{k+\tau_k}p)}{n\sigma^2}\right)\right)}\nonumber
    \end{align}
    which implies that 
    \begin{equation}\label{eq:ellk-bound}
        \frac{B\ln(2)}{n}\Big/{\Psi'\left(\Psi^{-1}\left(\frac{B^2\ln(2^{k+1}p)}{n\sigma^2}\right)\right)}\le \ell_{k+1}-\ell_k\le \frac{B\ln(2)}{n} \Big/{\Psi'\left(\Psi^{-1}\left(\frac{B^2\ln(2^{k}p)}{n\sigma^2}\right)\right)}
    \end{equation}

    Lemma \ref{lem:psi-bd-small-x} implies that for $k$ such that ${B^2\ln(2^{k+1}p)}/(n\sigma^2)\le 1$, we have
    \begin{equation}\label{eq:smaller-than-1-ratio}
\frac{\Psi'\left(\Psi^{-1}\left(\frac{B^2\ln(2^{k+1}p)}{n\sigma^2}\right)\right)}{\Psi'\left(\Psi^{-1}\left(\frac{B^2\ln(2^{k}p)}{n\sigma^2}\right)\right)}\le \frac{3}{2\sqrt{2}}\frac{\ln\left(1+\sqrt{\left(\frac{B^2\ln(2^{k+1}p)}{n\sigma^2}\right)}\right)}{\ln\left(1+\sqrt{\left(\frac{B^2\ln(2^{k}p)}{n\sigma^2}\right)}\right)}\le \frac{3}{2\sqrt{2}}\sqrt{\frac{\ln(2^{k+1}p)}{\ln(2^kp)}}\le \frac{3}{2}
    \end{equation}
    This follows from the fact that since $x\mapsto\ln(1+x)/x$ is a decreasing function, hence for $0<x<y$, 
    \begin{equation}\label{eq:log-dec}
        \frac{\ln(1+y)}{\ln(1+x)}\le \frac{y}{x}.
    \end{equation}
    Now, using \eqref{eq:log-dec} and appealing to Lemma \ref{lem:psi-bd-large-x}, we get that for $k$ such that ${B^2\ln(2^{k+1}p)}/{(n\sigma^2)}> 1$,
    \begin{equation}\label{eq:larger-than-1-ratio}
    \frac{\Psi'\left(\Psi^{-1}\left(\frac{B^2\ln(2^{k+1}p)}{n\sigma^2}\right)\right)}{\Psi'\left(\Psi^{-1}\left(\frac{B^2\ln(2^{k}p)}{n\sigma^2}\right)\right)}\le \frac{\sqrt{2}}{\ln 2}\frac{\ln\left(1+\left(\frac{B^2\ln(2^{k+1}p)}{n\sigma^2}\right)\right)}{\ln\left(1+\left(\frac{B^2\ln(2^{k}p)}{n\sigma^2}\right)\right)}\le \frac{\sqrt{2}}{\ln 2}\frac{\ln(2^{k+1}p)}{\ln(2^kp)}\le \frac{2\sqrt{2}}{\ln 2}
    \end{equation}
    Combining~\eqref{eq:smaller-than-1-ratio} and~\eqref{eq:larger-than-1-ratio}, we conclude
    \[
    {\Psi'\left(\Psi^{-1}\left(\frac{B^2\ln(2^{k+1}p)}{n\sigma^2}\right)\right)}\Big/{\Psi'\left(\Psi^{-1}\left(\frac{B^2\ln(2^{k}p)}{n\sigma^2}\right)\right)}\le \frac{2\sqrt{2}}{\ln 2}\quad\mbox{for all}\quad k\ge1.
    \]
    Taking $a={2\sqrt{2}}/{\ln 2}$, we have that for any $1\le k\le K$,
    \begin{equation}\label{eq:telescope-psi}
        \Psi'\left(\Psi^{-1}\left(\frac{B^2\ln(2^{k+1}p)}{n\sigma^2}\right)\right)\le a\Psi'\left(\Psi^{-1}\left(\frac{B^2\ln(2^{k}p)}{n\sigma^2}\right)\right)\le \cdots\le a^k \Psi'\left(\Psi^{-1}\left(\frac{B^2\ln(2p)}{n\sigma^2}\right)\right)
    \end{equation}
    For further analysis, we need to consider two cases separately.
    \paragraph{Case 1: $\ell_2\le B$.} In this case, 
    \begin{align}
        \int_A^B p^{\text{Benn}}(t)\,dt &= \sum_{k=1}^K \int_{\ell_k}^{\ell_{k+1}}q^{\text{Benn}}(t)\,dt+\int_{\ell_{K+1}}^B q^{\text{Benn}}(t)\,dt\nonumber\\
        & \ge \sum_{k=1}^K \int_{\ell_k}^{\ell_{k+1}}q^{\text{Benn}}(t)\,dt\nonumber\\
        & \ge \sum_{k=1}^K \frac{1}{2^k}(\ell_{k+1}-\ell_k)\nonumber\\
        & \ge \sum_{k=1}^K \frac{1}{2^{k}}\frac{B\ln 2}{n}\Big/{\Psi'\left(\Psi^{-1}\left(\frac{B^2\ln(2^{k+1}p)}{n\sigma^2}\right)\right)}\nonumber\\
        & \ge \sum_{k=1}^K \frac{1}{(2a)^k}\frac{B\ln 2}{n}\frac{1}{\ln(1+AB/\sigma^2)}\nonumber\\
        & \ge \frac{1}{2a}\frac{B\ln 2}{n}\frac{1}{\ln(1+AB/\sigma^2)}\nonumber\\
        & = \frac{1}{2a}\frac{B\ln 2}{n}f(A)\nonumber\\
        & \ge \frac{1}{2a}\frac{B\ln 2}{n}(f(A)-f(B)).\label{eq:small-ell_2-bd}
    \end{align}
    \paragraph{Case 2: $\ell_2 > B$.} In this case, we have 
    \[
    \int_A^B p^{\text{Benn}}(t)\,dt \ge \int_A^B \frac{1}{2}\,dt = \left(\frac{B-A}{2}\right).
    \]
    
    Further, note that 
    \begin{align}
        f(A)-f(B) &= \frac{1}{\ln(1+AB/\sigma^2)} - \frac{p^{\text{Benn}}(B)}{\ln(1+B^2/\sigma^2)}\nonumber\\
        & \le \frac{1}{\ln(1+AB/\sigma^2)}-\frac{1}{2\ln(1+B^2/\sigma^2)}\nonumber\\
        & \le \frac{1}{\ln(1+AB/\sigma^2)}-\frac{1}{2\ln(1+\ell_2B/\sigma^2)}\nonumber\\
        & \le \frac{1}{\ln(1+AB/\sigma^2)}-\frac{1}{2a\ln(1+AB/\sigma^2)}\nonumber\\
        & = \left(1-\frac{1}{2a}\right)\frac{1}{\ln(1+AB/\sigma^2)}.\label{eq:upper-bd-f-large-ell_2}
    \end{align}
    
    Moreover, by continuity of $\Psi^{-1}$, for some $\alpha \in (0,1)$,
    \begin{align}
        B-A & = \frac{\sigma^2}{B}\left(\Psi^{-1}\left(\frac{B^2\ln(2^{1+\alpha}p)}{n\sigma^2}\right) - \Psi^{-1}\left(\frac{B^2\ln(2p)}{n\sigma^2}\right)\right)\nonumber\\
        & = \frac{B\ln 2}{n}\Big/{\Psi'\left(\Psi^{-1}\left(\frac{B^2\ln(2^{1+\beta}p)}{n\sigma^2}\right)\right)} \label{eq:lmvt-psi-large-ell_2}\\
        & \ge \frac{B\ln 2}{n}\Big/{\Psi'\left(\Psi^{-1}\left(\frac{B^2\ln(2^{2}p)}{n\sigma^2}\right)\right)}\nonumber\\
        & = \frac{B\ln 2}{n}\frac{1}{\ln(1+\ell_2B/\sigma^2)}\nonumber\\
        & \ge \frac{B\ln 2}{na}\frac{1}{\ln(1+AB/\sigma^2)}.\label{eq:eq:lower-B-large-ell_2}
    \end{align}
    Note that $\exists$ $\beta\in(0,\alpha)$ satisfying \eqref{eq:lmvt-psi-large-ell_2} by Lagrange's Mean Value Theorem. Also, from \eqref{eq:upper-bd-f-large-ell_2} and \eqref{eq:eq:lower-B-large-ell_2}, we have
    \begin{equation}
        \frac{B}{n}(f(A)-f(B))\le \frac{2a-1}{2\ln 2}(B-A). \nonumber
    \end{equation}
    Consequently, we must have for $B<\ell_2$, 
    \begin{equation}\label{eq:large-ell_2-bd}
        \int_A^B p^{\text{Benn}}(t)\,dt \ge \frac{B-A}{2}\ge \frac{\ln 2}{2a-1}\frac{B}{n}(f(A)-f(B)).
    \end{equation}
    Combining the results \eqref{eq:small-ell_2-bd} and \eqref{eq:large-ell_2-bd} for the cases $B\ge \ell_2$ and $B<\ell_2$ respectively, we find that 
    \begin{equation}\label{eq:ell_2-bd}
        \int_A^B p^{\text{Benn}}(t)\,dt\ge \frac{\ln 2}{2a}\frac{B}{n}(f(A)-f(B)).
    \end{equation}
    Our analysis thus culminates in the following result
    \begin{equation}\label{eq:benn-bd-final}
        A+\frac{\ln2}{2a}\frac{B}{n}(f(A)-f(B))\le\int_0^B p^{\text{Benn}}(t)\,dt\le A+\frac{B}{n}(f(A)-f(B))
    \end{equation}
    or in other words
    \begin{equation}\label{eq:benn-order}
        \int_0^B p^{\text{Benn}}(t)\,dt - A \asymp \frac{B}{n}(f(A)-f(B)).   
    \end{equation}\qed

\subsection{Proof of Theorem~\ref{thm:ind-bern}}\label{pf:ind-bern}
By Bennett's Inequality \citep{bennett1962probability}, for independent random variables $X_1, X_2,\cdots, X_n$ with $|X_i|\le M_i$ and $\operatorname{Var}(X_i)=\sigma_i^2$, we have
     \begin{equation}\label{eq:bennet-1}
         \mathbb{P}(S_n\ge t)\le \exp\left(-\frac{\sigma^2}{M^2}\Psi\left(\frac{tM}{\sigma^2}\right)\right)
    \end{equation}
    where $S_n = \sum_{i=1}^n X_i$, $M=\max_{1\le i\le n}M_i$ and $\sigma^2=\sum_{i=1}^n \sigma_i^2$. Replacing the $X_i$'s by $-X_i$'s, we get the following inequality
    \begin{equation}\label{eq:bennet-2}
        \mathbb{P}(S_n\le -t)\le \exp\left(-\frac{\sigma^2}{M^2}\Psi\left(\frac{tM}{\sigma^2}\right)\right).
    \end{equation}
    Combining \eqref{eq:bennet-1} and \eqref{eq:bennet-2}, we have 
    \begin{equation}\label{eq:bennet-3}
        \mathbb{P}(|S_n|\ge t)\le 2\exp\left(-\frac{\sigma^2}{M^2}\Psi\left(\frac{tM}{\sigma^2}\right)\right).
    \end{equation}
    Towards obtaining an upper bound on the tail probability of the $L_{\infty}$ norm of $\overline{\bX}=\frac{1}{n}\sum_{i=1}^n \bX_i$ where $\bX_i$'s are independent $p$-variate random vectors $\{\bX_i\}_{1\le i\le n}$ such that for every $i\in \{1,2,\cdots, n\}$, $|X_{i, j}|\le B$, $j=1,2,\cdots,p$ where $\bX_i=(X_{i,1}, X_{i,2},\cdots, X_{i,p})$ and $\operatorname{Var}(X_{i,j})\le \sigma^2$, we first define $\bS_n = \sum_{i=1}^n \bX_i$ and $S_{n,j} = \sum_{i=1}^n X_{i,j}$ and observe the following:
    \begin{equation}
        \mathbb{P}(\|\overline{\bX}\|_{\infty}\ge t) = \mathbb{P}(\|\bS_n\|_{\infty}\ge nt)=\mathbb{P}\left(\bigcup_{j=1}^p S_{n,j}\ge nt\right).
    \end{equation}
    Now, applying the union bound and using \eqref{eq:bennet-3}
    \begin{equation}\label{eq:tail-bd-xbar}
        \mathbb{P}(\|\overline{\bX}\|_{\infty}\ge t)\le \sum_{j=1}^p \mathbb{P}(S_{n,j}\ge nt)= p\mathbb{P}(S_{n,1}\ge nt)\le 2p \exp\left(-\frac{n\sigma^2}{B^2}\Psi\left(\frac{tB}{\sigma^2}\right)\right) = q^{\text{Benn}}(t).
    \end{equation}
    It is however, immediate that the bound in \eqref{eq:tail-bd-xbar} can be improved to $p^{\text{Benn}}(t)$, and, as we saw in the proof of Theorem \ref{thm:benn-nature} that \[\int_0^B p^{\text{Benn}}(t)\,dt\le A+\frac{B}{n}(f(A)-f(B)).\]
    Consequently,
    \[\mathbb{E}\left[\|\overline{\bW}\|_{\infty}\right]=\int_0^B\mathbb{P}\left(\|\overline{\bW}\|_{\infty}\ge t\right)\,dt\le \int_0^B p^{\text{Benn}}(t)\,dt\le A+\frac{B}{n}(f(A)-f(B)).\]

    This proves the first part of the theroem.\\

    \textbf{(a)} In this part, we are required to elicit lower bounds on the expected $L_{\infty}$ norm when the following condition holds:
    \[n\sigma^2/(\sigma^2+B^2)\ge 1/(2p).\] We split our analysis into several cases, the first of which is the following:\\
    $B^2\ln (2p)\le n\sigma^2$. To handle this case, we make use of the following bound on the inverse of $\Psi(x)$ obtained by invoking Lemma \ref{lem:benn-psi-inv-bd}
    \begin{equation*}
        \sqrt{2x}\le \Psi^{-1}(x)\le 2\sqrt{2x}\text{ for }0\le x\le e.
    \end{equation*}
    In our case, we have ${B^2\ln (2p)}/{(n\sigma^2)}\le e$ and appealing to Lemma \ref{lem:benn-psi-inv-bd}, we have 
    \begin{equation}\label{eq:small-b-ineq}
        \sigma \sqrt{\frac{2\ln (2p)}{n}}\le \frac{\sigma^2}{B}\Psi^{-1}\left(\frac{B^2\ln (2p)}{n\sigma^2}\right)\le 2\sigma \sqrt{\frac{2\ln (2p)}{n}}
        \Rightarrow \sigma \sqrt{\frac{2\ln (2p)}{n}}\le A\le 2\sigma \sqrt{\frac{2\ln (2p)}{n}}.
    \end{equation}
    Further, 
    \begin{equation}\label{eq:small-b-ineq-2}
        \frac{B}{n}(f(A)-f(B))\le \frac{B}{n}f(A) = \frac{B}{n}\frac{1}{\ln(1+AB/\sigma^2)}\le \frac{B}{n}\frac{1}{\ln(1+B\sqrt{\ln (2p)}/(\sigma\sqrt{n}))}.
    \end{equation}
    The last inequality follows by Lemma \ref{lem:psi-der-lower}. Combining \eqref{eq:small-b-ineq-2} with Lemma \ref{lem:log-lower-bd}, we get 
    \begin{equation}\label{eq:small-b-ineq-3}
        \frac{B}{n}(f(A)-f(B))\le \frac{3B}{n}\frac{\sigma\sqrt{n}}{B\sqrt{\ln (2p)}}=\frac{3\sigma}{\sqrt{n}}\frac{1}{\sqrt{\ln (2p)}}
    \end{equation}
    and \eqref{eq:small-b-ineq}, Lemma \ref{lem:aux-log} and \eqref{eq:small-b-ineq-3} together yield 
    \begin{align}
        A+\frac{B}{n}(f(A)-f(B))&\le \frac{\sigma}{\sqrt{n}}\sqrt{2\ln (2p)}\left(2+\frac{3}{\sqrt{\ln (2p)}}\right)\nonumber\\
        &\le \frac{b\sigma}{\sqrt{n}}\sqrt{2\ln (2p)}\nonumber\\
        &\le \frac{2b\sigma}{\sqrt{n}}\sqrt{2\ln\left(\frac{2p}{\sqrt{\ln (2p)}}\right)}\label{small-b-ineq-4}
    \end{align}
    where $b = (2+3/\sqrt{\ln 2})$. We now choose 
    $t = {\sigma\sqrt{\ln(2p/\sqrt{\ln (2p)})}}/{\sqrt{en}}$
    and by Lemma \ref{lem:bin-bd-norm}, we observe that 
    \begin{align}
        \mathbb{P}\left(\left|X-\frac{n\sigma^2}{\sigma^2+B^2}\right|\ge \frac{nBt}{\sigma^2+B^2}\right) & \ge \mathbb{P}\left(X-\frac{n\sigma^2}{\sigma^2+B^2}\ge \frac{nBt}{\sigma^2+B^2}\right) \ge \frac{\tau}{\gamma}\Phi\left(-\frac{t\sqrt{2n}}{\sigma}\right). \label{thm-2.2-eq-1}
    \end{align}
    Observing that \[\frac{\ln(2p/\sqrt{\ln (2p)})}{2e}\le \frac{\ln (2p)}{e}\le \ln(2p/\sqrt{\ln (2p)})\] and from Eq. \eqref{almost-mills}, it follows that 
    \begin{align}
        \Phi\left(-\frac{t\sqrt{2n}}{\sigma}\right) &= \Phi\left(-\sqrt{2\ln(2p/\sqrt{\ln (2p)})}\right)\nonumber \\
        & \ge \frac{1}{1+\sqrt{2\ln(2p/\sqrt{\ln (2p)})}}\frac{1}{\sqrt{2\pi}}\frac{\sqrt{\ln (2p)}}{2p}\nonumber \\
        & \ge \frac{1}{2\sqrt{2\ln(2p/\sqrt{\ln (2p)})}}\frac{1}{\sqrt{2\pi}}\frac{\sqrt{\ln (2p)}}{2p}\nonumber \\
        & = \frac{1}{\sqrt{2\pi}}\frac{1}{\sqrt{2-\ln(\ln (2p))/\ln (2p)}}\frac{1}{4p} \ge \frac{1}{24p}\nonumber
    \end{align}
    Setting $x = \tau/(\tau + 24\gamma)$ and using Lemma \ref{lem:bin-th-gen}, 
    \[\left(\frac{24\gamma}{\tau+24\gamma}\right)^{1/p}\ge 1-\frac{\tau}{24p\gamma}\]
    Consequently
    \begin{align}
        \mathbb{P}\left(\|W\|_{\infty}\ge t\right) = \mathbb{P}\left(\max_{1\le j\le p}|W_j|\ge t\right)\nonumber  = 1-\left(1-\mathbb{P}\left(|W_j|\ge t\right)\right)^{p}\nonumber\ge \frac{\tau}{\tau+24\gamma}.\nonumber
    \end{align}
    Finally, by Markov's Inequality,
    \begin{align}
        \mathbb{E}\left[\|W\|_{\infty}\right] & \ge t\mathbb{P}\left(\|W\|_{\infty}\ge t\right)\nonumber \\
        &= \frac{t\tau}{\tau+24\gamma}\left(A+\frac{B}{n}(f(A)-f(B))\right)\nonumber\ge \frac{\tau}{4b\sqrt{e}(\tau+24\gamma)}\left(A+\frac{B}{n}(f(A)-f(B))\right).\nonumber
    \end{align}
    This settles our analysis for the case that $B^2\ln(2p)\le en\sigma^2$. The other case, i.e. $B^2\ln(2p)> en\sigma^2$ is considerably more involved which is why, we divide this broad case into multiple parts. \\
    
    In the first part, alongside $B^2\ln (2p)>ne\sigma^2$, we have $t^*B\ge 4(\sigma^2+B^2)/(5n)$ where 
    \[t^*=\frac{4\sigma^2}{9B}\left(\frac{(1+B^2/\sigma^2)(4\ln (2p)/5n)}{W((1+B^2/\sigma^2)(4\ln (2p)/5n))}-1\right).\]
    Now plugging in $x=B^2/\sigma^2$ and $y=\ln (2p)/n$ in Lemma \ref{lem:aux-log-2}, we get
    \[\Psi^{-1}(B^2\ln (2p)/(n\sigma^2))\le \frac{e}{e-1}\frac{B^2\ln (2p)/(n\sigma^2)-1}{\ln(1+(B^2\ln (2p)/(n\sigma^2)-1)/e)}\le \frac{e^2}{e-1}\frac{(1+B^2/\sigma^2)(\ln (2p)/n)}{\ln((1+B^2/\sigma^2)(\ln (2p)/n))}.\]
    From Lemma \ref{lem:psi-bd-small-x} and Lemma \ref{lem:psi-bd-large-x}, we have 
    \[\Psi'(\Psi^{-1}(x))\ge \frac{1}{\sqrt{2}}\ln(1+x)\]
    from which it follows that
    \[\frac{B^2}{n\sigma^2}(f(A)-f(B))\le \frac{B^2}{n\sigma^2}\frac{1}{\ln(1+AB/\sigma^2)}\le \frac{B^2}{n\sigma^2}\frac{\sqrt{2}}{\ln(1+B^2\ln (2p)/(n\sigma^2))}.\]
    Consider first the case that $\ln (2p) \le n$. In this case 
    \[\frac{B^2}{n\sigma^2}\frac{\sqrt{2}}{\ln(1+B^2\ln (2p)/(n\sigma^2))}\le \frac{B^2\ln (2p)}{n\sigma^2}\frac{\sqrt{2}/\ln 2}{\ln((1+B^2/\sigma^2)(\ln (2p)/n))}\le \frac{\sqrt{2}}{\ln 2}\frac{(1+B^2/\sigma^2)(\ln (2p)/n)}{\ln((1+B^2/\sigma^2)(\ln (2p)/n))}.\]
    As for the other case i.e. $\ln (2p)>n$, we have 
    \begin{equation}\label{eq:psi-AB-sig}
        \frac{B^2}{n\sigma^2}\frac{\sqrt{2}}{\ln(1+B^2\ln (2p)/(n\sigma^2))}\le \frac{B^2\ln (2p)}{n\sigma^2}\frac{\sqrt{2}}{\ln(1+B^2/\sigma^2)\ln (2p)}.
    \end{equation}
    Further
    \begin{equation}\label{eq:psi-AB-sig-2}
        \frac{B^2\ln (2p)}{n\sigma^2}\frac{\sqrt{2}}{\ln(1+B^2/\sigma^2)\ln (2p)}\le \sqrt{2}\frac{(1+B^2/\sigma^2)(\ln (2p)/n)}{\ln((1+B^2/\sigma^2)(\ln (2p)/n))}.
    \end{equation}
    \eqref{eq:psi-AB-sig-2} is true due to the fact that \[\frac{\ln (2p)}{n}\le \ln (2p)\le (2p)^{1-1/n}\] which means that 
    \[\ln(1+B^2/\sigma^2)+\ln(\ln(2p)/n)\le \frac{\ln(2p)}{n}\ln(1+B^2/\sigma^2)+\left(1-\frac{1}{n}\right)\ln (2p)\ln\left(1+B^2/\sigma^2\right)\] with the last inequality holding because \[1= \ln(e)\le \left(1+B^2/\sigma^2\right).\] Combining \eqref{eq:psi-AB-sig} and \eqref{eq:psi-AB-sig-2}, we get
    \begin{align}
        \frac{B}{\sigma^2}\left(A+\frac{B}{n}(f(A)-f(B))\right) &\le \left(\frac{e^2}{e-1}+\frac{\sqrt{2}}{\ln 2}+\frac{1}{e}\right)\frac{(1+B^2/\sigma^2)(\ln (2p)/n)}{\ln((1+B^2/\sigma^2)(\ln (2p)/n))}-1 \nonumber \\
        & \le 9\frac{(1+B^2/\sigma^2)(4\ln (2p)/5n)}{\ln((1+B^2/\sigma^2)(4\ln (2p)/5n))}-1 .\label{eq:large-B-bd-pre}
    \end{align}
    Note that 
    \begin{align}
        \frac{x}{\ln x}\ge e \Longrightarrow \frac{x}{\ln x}\ge \frac{67}{25}  \Longleftrightarrow \frac{25x}{\ln x}\ge 67 \Longleftrightarrow \frac{34x}{\ln x}-68 \ge \frac{9x}{\ln x}-1.\label{eq:aux-log}
    \end{align}
    In addition, for $x\ge 4e/5$
    \[W(x)\le \ln(1+x)\le 2\ln x.\]
    Combining the above with \eqref{eq:large-B-bd-pre} and \eqref{eq:aux-log}, we get
    \begin{equation}\label{eq:large-B-bd-1}
        \frac{B}{\sigma^2}\left(A+\frac{B}{n}(f(A)-f(B))\right)\le 68 \left(\frac{(1+B^2/\sigma^2)(4\ln (2p)/5n)}{W((1+B^2/\sigma^2)(4\ln (2p)/5n))}-1\right).
    \end{equation}
    By \cite{MR3196787}, 
    \begin{equation}\label{eq:zubkov-bd-large-B}
        \mathbb{P}\left(\left |X - \frac{n\sigma^2}{\sigma^2+B^2}\right |\ge \frac{nBt^*}{\sigma^2+B^2}\right)\ge \mathbb{P}\left(X\ge \frac{n(\sigma^2+Bt^*)}{\sigma^2+B^2}\right)\ge \Phi(-b(t^*))
    \end{equation}
    where $X\sim Bin(n,\sigma^2/(\sigma^2+B^2))$ and \[b(t^*)=\sqrt{\frac{2n\sigma^2}{\sigma^2+B^2}\left(1+\frac{t^*B+(\sigma^2+B^2)/n}{\sigma^2}\right)\ln \left(1+\frac{t^*B+(\sigma^2+B^2)/n}{\sigma^2}\right)}.\]
    Since it is given that $t^*B\ge 4(\sigma^2+B^2)/5n$, we must have 
    \[b(t^*)\le \sqrt{\frac{2n\sigma^2}{\sigma^2+B^2}\left(1+\frac{9t^*B}{4\sigma^2}\right)\ln\left(1+\frac{9t^*B}{4\sigma^2}\right)}.\]
    Moreover
    \begin{align}\label{thm-2.3-eq-7}
        &\sqrt{\frac{2n\sigma^2}{\sigma^2+B^2}\left(1+\frac{9t^*B}{4\sigma^2}\right)\ln\left(1+\frac{9t^*B}{4\sigma^2}\right)} \nonumber\\
         =& \sqrt{\frac{2n\sigma^2}{\sigma^2+B^2}\left(1+\frac{B^2}{\sigma^2}\right)\frac{4\ln (2p)}{5n}}\nonumber  = \sqrt{\frac{8\ln (2p)}{5}}\nonumber \le \sqrt{2\ln\left(\frac{2p}{\sqrt{\ln (2p)}}\right)}.\nonumber
    \end{align}
    As shown in part (a), we must have 
    \[\Phi(-b(t^*))\ge\frac{1}{24p}.\]
    Consequently, 
    \[\mathbb{P}\left(\|W\|_{\infty}\ge t^*\right)\ge \frac{1}{25}.\]
    Finally by Markov's inequality, 
    \[\mathbb{E}\left[\|W\|_{\infty}\right]\ge t^* \mathbb{P}\left(\|W\|_{\infty}\ge t^*\right)\ge \frac{t^*}{25}\ge \frac{1}{3825}\left(A+\frac{B}{n}(f(A)-f(B))\right).\]

    For the case where $B^2\ln (2p)>ne\sigma^2$ and $t^*B<4(\sigma^2+B^2)/(5n)$, note first that the two assumptions mentioned, imply that 
    \begin{align}
        &\frac{4}{9}\left(\frac{(1+B^2/\sigma^2)(4\ln2p /5n)}{W((1+B^2/\sigma^2)(4\ln (2p)/5n))}-1\right)<\frac{4(\sigma^2+B^2)}{5n\sigma^2}\nonumber \\
         \implies &\left(1+\frac{B^2}{\sigma^2}\right)\frac{4\ln (2p)}{5n}<\left(1+\frac{9(\sigma^2+B^2)}{5n\sigma^2}\right)\ln \left(1+\frac{9(\sigma^2+B^2)}{5n\sigma^2}\right)\nonumber \\
         \implies &\left(1+\frac{B^2}{\sigma^2}\right)\frac{4\ln (2p)}{5n}< \begin{cases}
            \frac{9(\sigma^2+B^2)}{5n\sigma^2}4\ln 10, & \text{ if }\frac{\sigma^2+B^2}{n\sigma^2}\le 5\\
            \frac{18(\sigma^2+B^2)}{\sigma^2}\ln \left(\frac{18(\sigma^2+B^2)}{\sigma^2}\right), & \text{ if }\frac{\sigma^2+B^2}{n\sigma^2}> 5
        \end{cases} \label{eq:large-B-bd-2}
    \end{align}
    where \eqref{eq:large-B-bd-2} follows from the fact that \[\frac{\sigma^2+B^2}{n\sigma^2}\ge \frac{5}{9}\inf_{x>1}\left(\frac{x}{2\ln x}-1\right)\ge \frac{5}{9}(e/2-1).\] We first consider the case where $(\sigma^2+B^2)/(n\sigma^2)\le 5$. In this case, we have 
    \[\left(1+\frac{B^2}{\sigma^2}\right)\frac{4\ln (2p)}{5n}\le 8\ln 10 \frac{(1+B^2/\sigma^2)}{n}\le 8\left(1+\frac{B^2}{\sigma^2}\right)\frac{\ln 10}{\ln 2}\frac{\ln (2p)}{n}.\]
    This in turn implies that 
    \[\ln 2\le \ln (2p) \le 10\ln 10.\]
    Let $\gamma=10\ln 10/\ln 2$. Then 
    \begin{align}
        \Phi(-b(t^*)) & \ge \Phi \left(-\sqrt{2n\gamma \frac{\sigma^2}{\sigma^2+B^2}\left(1+\frac{B^2}{\sigma^2}\right)\frac{\ln (2p)}{n}}\right)\nonumber \\
        & = \Phi(-\sqrt{2\gamma \ln (2p)})\nonumber \\
        & \ge \frac{1}{1+\sqrt{2\gamma \ln (2p)}}\exp(-\gamma \ln (2p))\nonumber \\
        & \ge \frac{1}{2\sqrt{2\gamma \ln (2p)}}\left(\frac{1}{2p}\right)^{\gamma}\nonumber \\
        & \ge \frac{1}{2}\frac{1}{\gamma\sqrt{2\ln 2}}\left(\frac{1}{10^{10}}\right)^{\gamma-1}\left(\frac{1}{p}\right) = \frac{\varepsilon}{p}.\label{eq:norm-bd-large-B}
    \end{align}
    From \eqref{eq:norm-bd-large-B} and Lemma \ref{lem:benn-psi-inv-bd}, we get 
    \[\mathbb{P}\left(\left\|W\right\|_\infty\ge t^*\right)\ge \frac{\varepsilon}{1+\varepsilon}\]
    Now, as shown in the proof of part (b), $t^*$ is lower bounded by a constant multiple of the Bennett bound and hence, by Markov inequality the result follows.\\

    We are now only left with the case that $n\sigma^2/(\sigma^2+B^2)<1/5$ and $nBt^*/(\sigma^2+B^2)<4/5$. Together they imply that $n(\sigma^2+Bt^*)/(\sigma^2+B^2)<1$. For notational simplicity, for further computations we will denote the quantity $\sigma^2/(\sigma^2+B^2)$ by $p_{\sigma,B}$. Note that 
    \[\mathbb{P}\left(\left |\overline{W_j}\right |\ge t^*\right)=\mathbb{P}\left(X\le \frac{n\sigma^2}{\sigma^2+B^2}-\frac{nBt^*}{\sigma^2+B^2}\right)+\mathbb{P}\left(X\ge \frac{n\sigma^2}{\sigma^2+B^2}+\frac{nBt^*}{\sigma^2+B^2}\right)\]
    where $X\sim Bin(n,p_{\sigma,B})$. When $t^*\le \sigma^2/B$, the above probability becomes $\mathbb{P}(X\le 0)+\mathbb{P}(X\ge 1)=1$. By Markov's inequality
    \begin{align}
        \mathbb{E}\left[\max_{1\le j\le p}\left |\overline{W_j}\right |\right] & \ge t^* \mathbb{P}\left(\max_{1\le j\le p}\left |\overline{W_j}\right |\ge t^*\right) \nonumber \\
        & =t^* \left(1-\left(1-\mathbb{P}\left(\left |\overline{W_j}\right |\ge t^*\right)\right)^p\right)\nonumber =t^*\ge \frac{1}{153}\left(A+\frac{B}{n}(f(A)-f(B))\right)\nonumber 
    \end{align}
    We now turn our attention to the case where $t^*>\sigma^2/B$. We are given that $np_{\sigma,B}\ge 1/(2p)$. When $t^*>\sigma^2/B$, we can immediately say that
    \[\mathbb{P}\left(\left |\overline{W_j}\right |\ge t^*\right)=\mathbb{P}(X\ge 1)=1-\left(1-p_{\sigma,B}\right)^n.\]
    As a consequence,
    \[\mathbb{P}\left(\max_{1\le j\le p}\left |\overline{W_j}\right |\ge t^*\right)=1-\left(1-p_{\sigma,B}\right)^{np}=\ge 1-\left(1-\frac{1}{2np}\right)^{np}\ge 1-e^{-1/2}.\]
     Hence by Markov's inequality, 
     \[\E\left[\|\bW\|_{\infty}\right]\ge t^*\mathbb{P}\left(\max_{1\le j\le p}\left |\overline{W_j}\right |\ge t^*\right)\ge \frac{(1-e^{-1/2})}{153}\left(A+\frac{B}{n}(f(A)-f(B))\right).\]
     
     Throughout this proof, we have implicitly assumed that $A\le B$. If it is the other way around i.e. $A>B$, then $A+(B/n)(f(A)-f(B))\gtrsim B$. This time around, however, the term $(B/n)(f(A)-f(B))$ is negative but is of a lower order as compared to $A$. Further, we have not used the decreasingness of $f$ anywhere in the proof. We have lower bounded our quantity of interest by constant multiples of $A+Bf(A)/n$ and lower bounded that by $A+(B/n)(f(A)-f(B))$. To add to this, courtesy Proposition \ref{prop:A>B}, $A>B$ corresponds to the case $n\sigma^2/(\sigma^2+B^2)\ge 1/(2p)$. Thus, the lower bounds computed in this proof not only apply to the case that $A>B$, but they also imply that \[\E\left[\left\|\overline{\bW}\right\|_{\infty}\right]\asymp B.\]
     
     \textbf{(b)} We now work with the assumption $n\sigma^2/(\sigma^2+B^2)<1/(2p)$. Consequently, observing that 
     \[\overline{W_j}=\frac{\sigma^2+B^2}{nB}W_j-\frac{\sigma^2}{B}\]
     where $W_j\sim Bin(n, p_{\sigma,B})$, and noting that $n\sigma^2/(\sigma^2+B^2)<1/2$ we can easily see that it is possible for $\overline{W_j}$ to take only one negative value i.e. $-\sigma^2/B$ and that value is in fact smaller in absolute value than the smallest positive value ($(\sigma^2+B^2)/nB-\sigma^2/B$) that $\overline{W_j}$ can take. Hence, we may write the following:
     \begin{align}
         \E \left[\max_{1\le j\le p}\left |\overline{W_j}\right |\right] & = \E \left[\left |\max_{1\le j\le p}\overline{W_j}\right |\right]\nonumber \\
         & = \E \left[\left(\max_{1\le j\le p}\overline{W_j}\right)^+\right]+\E \left[\left(\max_{1\le j\le p}\overline{W_j}\right)^-\right]\nonumber \\
         & = \E \left[\left(\max_{1\le j\le p}\overline{W_j}\right)^+\right]-\E \left[\left(\max_{1\le j\le p}\overline{W_j}\right)^-\right]+2\E \left[\left(\max_{1\le j\le p}\overline{W_j}\right)^-\right]\nonumber \\
         & = \E \left[\max_{1\le j\le p}\overline{W_j}\right]+\frac{2\sigma^2}{B}\left(\frac{B^2}{\sigma^2+B^2}\right)^{np}\nonumber \\
         & = \left(\frac{\sigma^2+B^2}{nB}\right)\E\left[\max_{1\le j\le p}W_j\right]-\frac{\sigma^2}{B}+\frac{2\sigma^2}{B}\left(\frac{B^2}{\sigma^2+B^2}\right)^{np}.\label{thm-2.3-exp}
     \end{align}
     Let $U=\max_{1\le j\le p}W_j$ where $W_j$ are iid binomial random variables with parameters $n$ and $p_{\sigma,B}$. Then we know that 
     \begin{equation}\label{eq:very-large-B}
         \E[U]=\sum_{i=1}^n \mathbb{P}(U\ge j)=\sum_{i=1}^n \left[1-\mathbb{P}(W_1\le j-1)^p\right] = \sum_{i=1}^n \left[1-(1-\mathbb{P}(W_1\ge j))^p\right].
     \end{equation}
     Along with
     \begin{align}
         \E[U] & \le \E\left[\sum_{j=1}^p W_j\right] = p\E[W_1]=npp_{\sigma,B}, \nonumber
     \end{align}
     appealing to Lemma \ref{lem:union-bd-inc-bd}, we have
     \begin{align*}
         \E[U] & = \sum_{i=1}^n \left[1-\mathbb{P}(W_1\le j-1)^p\right]  \\
         & \ge 1-(\mathbb{P}(W_1= 0))^p  \\
         & = 1 - (1-p_{\sigma,B})^{np} \\
         & \ge npp_{\sigma,B}-\binom{np}{2}p_{\sigma,B}^2\\
         & = npp_{\sigma,B}\left(1 - \frac{(np-1)p_{\sigma,B}}{2}\right)\\
         & \ge npp_{\sigma,B}\left(1-\frac{npp_{\sigma,B}}{2}\right)\\
         & > npp_{\sigma,B}\left(1-\frac{1}{4}\right)=\frac{3}{4}npp_{\sigma,B}.
     \end{align*}
     Combining both inequalities above, we get
     \[1-\frac{(np-1)p_{\sigma,B}}{2}\le \frac{\E[U]}{npp_{\sigma,B}}\le 1\] and as $npp_{\sigma,B}=o(1)$, this implies that $\E[U] \sim npp_{\sigma,B}$. Note that in this case $f\sim g$ means that the ratio of $f$ and $g$ converges to 1. To be precise, $\le\E\left[\max_{1\le j\le p}\left |\overline{W_j}\right |\right]\le(p+1)\sigma^2/B$ and 
     \begin{align}
         & \E\left[\max_{1\le j\le p}\left |\overline{W_j}\right |\right]\ge\frac{\sigma^2}{B}\left[\frac{3p}{4}-1+2\left(\frac{B^2}{B^2+\sigma^2}\right)^{np}\right] \nonumber\\
          \implies &\E\left[\max_{1\le j\le p}\left |\overline{W_j}\right |\right]\ge\frac{\sigma^2}{B}\left[\frac{3p}{4}-1+2\left(1-\frac{np\sigma^2}{\sigma^2+B^2}\right)\right]\nonumber \\
          \implies &\E\left[\max_{1\le j\le p}\left |\overline{W_j}\right |\right]\ge\frac{3p\sigma^2}{4B}.\nonumber
     \end{align}
     These findings immediately lead us to the desired conclusion that 
     \[\mathbb{E}\left[\|\overline{\bW}\|_{\infty}\right]\asymp \frac{p\sigma^2}{B}\] when $npp_{\sigma,B}<1/2$.

\subsection{Proof of Theorem~\ref{thm:extension-indep}}\label{pf:extension-indep}
Consider independent random vectors  $\bY_1, \bY_2,\cdots \bY_n$ such that $Y_i(1), Y_i(2), \cdots , Y_i(p)$ are jointly independent for all $i=1,2,\cdots, n$ and
    \begin{equation*}
        \max_{1\le j\le p}|Y_i(j)|\le B \text{ and } \max_{1\le j\le p}\mathcal{V}\left[Y_i(j)\right]\le \sigma^2.
    \end{equation*}
As the random variables $|\overline{Y(j)}|$ are independent, we have 
    \begin{align}
        \mathbb{E}\left[\left\|\overline{\bY}\right\|_{\infty}\right]=\mathbb{E}\left[\max_{1\le j\le p}\left |\overline{Y(j)}\right |\right] &= \int_0^B \mathbb{P}\left(\left\|\overline{\bY}\right\|_{\infty}\ge t\right)\,dt \nonumber\\
        & = \int_0^B \mathbb{P}\left(\max_{1\le j\le p}\left |\overline{Y(j)}\right |\ge t\right)\,dt \nonumber\\
        & = \int_0^B 1-\prod_{j=1}^p\left(1-\mathbb{P}\left(\left |\overline{Y(j)}\right |\ge t\right)\right) \,dt. \label{eq:indep-prob}
    \end{align}
    Further, we have 
    \begin{align}
        \sum_{j=1}^p \mathbb{P}\left(\left |\overline{Y(j)}\right |\ge t\right)>1 
        &\implies \prod_{j=1}^p \exp\left(-\mathbb{P}\left(\left |\overline{Y(j)}\right |\ge t\right)\right)<\frac{1}{e}\nonumber \\
        &\implies \prod_{j=1}^p\left(1-\mathbb{P}\left(\left |\overline{Y(j)}\right |\ge t\right)\right)<\frac{1}{e}\nonumber \\
        &\implies 1-\prod_{j=1}^p\left(1-\mathbb{P}\left(\left |\overline{Y(j)}\right |\ge t\right)\right)>1-\frac{1}{e}.\label{eq:indep-prob-lower}
    \end{align}

    As in the proof of Theorem \ref{thm:benn-nature}, we define the sequence $\{\ell_k\}_{k\ge 1}$ as follows:
    \[\ell_k = \inf \left\{0\le x\le B: \sum_{i=1}^p\mathbb{P}\left(\left |\overline{Y(i)}\right |\ge x\right)\le \frac{1}{2^{k-1}}\right\}.\]
    Combining this with Lemma \ref{lem:inc-exc}, we get 
    \begin{align}
        \int_0^{\ell_1}\left(1-\frac{1}{e}\right)\,dt &\le \int_0^{\ell_1}1-\prod_{j=1}^p\left(1-\mathbb{P}\left(\left |\overline{Y(j)}\right |\ge t\right)\right)\,dt\nonumber \\ &\le \int_0^{\ell_1}\min\left\{1, \sum_{j=1}^p\mathbb{P}\left(\left |\overline{Y(j)}\right |\ge t\right)\right\}\,dt\le \int_0^{\ell_1} \,dt. \label{eq:bd-by-1}
    \end{align}
    We also define 
    \[K = \sup \left\{k\ge 0: \ell_{k+1}\le B\right\}.\]
    Note that $K$ may be finite or infinite. We first take up the case where $K$ is finite. We know that for $t$ between $\ell_k$ and $\ell_{k+1}$, 
    \[\sum_{j=1}^p\mathbb{P}\left(\left |\overline{Y(j)}\right |\ge t\right)\in \left[\frac{1}{2^k},\frac{1}{2^{k-1}}\right].\]
    Further, appealing to Lemma \ref{lem:inc-exc}, we find that 
    \begin{align}
        &1-\prod_{j=1}^p\left(1-\mathbb{P}\left(\left |\overline{Y(j)}\right |\ge t\right)\right)\\
         \ge &\sum_{j=1}^p\mathbb{P}\left(\left |\overline{Y(j)}\right |\ge t\right)-\sum_{i=1}^{p-1}\sum_{j>i}\mathbb{P}\left(\left |\overline{Y(i)}\right |\ge t\right)\mathbb{P}\left(\left |\overline{Y(j)}\right |\ge t\right)\nonumber \\
         =&\sum_{j=1}^p\mathbb{P}\left(\left |\overline{Y(j)}\right |\ge t\right)-\frac{1}{2}\left(\sum_{j=1}^p \mathbb{P}\left(\left |\overline{Y(j)}\right |\ge t\right)\right)^2+\frac{1}{2}\sum_{j=1}^p \mathbb{P}\left(\left |\overline{Y(j)}\right |\ge t\right)^2\nonumber \\
         \ge & \sum_{j=1}^p \mathbb{P}\left(\left |\overline{Y(j)}\right |\ge t\right)\left(1-\frac{1}{2}\sum_{j=1}^p\mathbb{P}\left(\left |\overline{Y(j)}\right |\ge t\right)\right)\nonumber \\
         \ge &\sum_{j=1}^p \mathbb{P}\left(\left |\overline{Y(j)}\right |\ge t\right)\left(1-\frac{1}{2^k}\right)\nonumber \\
         \ge &\frac{1}{2}\sum_{j=1}^p \mathbb{P}\left(\left |\overline{Y(j)}\right |\ge t\right).\label{eq:prob-indep-lower-union}
    \end{align}
    Since,
    \begin{align}
        \mathbb{E}\left[\max_{1\le j\le p}\left |\overline{Y(j)}\right |\right] &= \int_0^B F(t) \,dt =\int_0^{\ell_1}F(t) \,dt + \sum_{k=1}^K \int_{\ell_k}^{\ell_{k+1}}F(t) \,dt+\int_{\ell_{K+1}}^B F(t) \,dt.\nonumber
    \end{align}
    where \[F(t)=1-\prod_{j=1}^p\left(1-\mathbb{P}\left(\left |\overline{Y(j)}\right |\ge t\right)\right). \]
    From \eqref{eq:bd-by-1} and \eqref{eq:prob-indep-lower-union}, it follows that 
    \begin{equation}\label{eq:exp-bd-finite-ellk}
        \frac{1}{2}\mathcal{E}_1\le\mathbb{E}\left[\max_{1\le j\le p}\left |\overline{Y(j)}\right |\right]\le \mathcal{E}_1
    \end{equation}
    where \[\mathcal{E}_1=\int_0^{\ell_1}\,dt+\sum_{k=1}^K \int_{\ell_k}^{\ell_{k+1}}\sum_{j=1}^p\mathbb{P}\left(\left |\overline{Y(j)}\right |\ge t\right)\,dt+\int_{\ell_{K+1}}^B \sum_{j=1}^p\mathbb{P}\left(\left |\overline{Y(j)}\right |\ge t\right)\,dt.
    \]
    On the other hand, when $K$ is infinite, by the exact same arguments, we find that 
    \begin{equation}\label{eq:exp-bd-infinite-ellk}
        \frac{1}{2}\mathcal{E}_2\le\mathbb{E}\left[\max_{1\le j\le p}\left |\overline{Y(j)}\right |\right]\le \mathcal{E}_2
    \end{equation}
    where \[\mathcal{E}_2=\int_0^{\ell_1}\,dt+\sum_{k=1}^{\infty} \int_{\ell_k}^{\ell_{k+1}}\sum_{j=1}^p\mathbb{P}\left(\left |\overline{Y(j)}\right |\ge t\right)\,dt.
    \]
    Therefore 
    \[\frac{1}{2}\mathcal{E}\le\mathbb{E}\left[\max_{1\le j\le p}\left |\overline{Y(j)}\right |\right]\le \mathcal{E}\]
    where 
    \[\mathcal{E}=\int_0^{B}\min\left\{1, \sum_{j=1}^p\mathbb{P}\left(\left |\overline{Y(j)}\right |\ge t\right)\right\}.\]
    All that remains in this proof, is the actual nature of the distribution functions of the variables $Y_i(j)$. For each $i$ and $j$, we choose $Y_i(j)\stackrel{d}{=}X_i(j)$ that is we take the distribution of $Y_i(j)$ to be the marginal distribution of $X_i(j)$. Note that the $Y_i(j)$'s are chosen independent of each other. Now, 
    \begin{align}
        \E\left[\|\overline{\bX}\|_{\infty}\right]=\mathbb{E}\left[\max_{1\le j\le p}\left |\overline{X(j)}\right |\right] &= \int_0^B \mathbb{P}\left(\left\|\overline{\bX}\right\|_{\infty}\ge t\right)\,dt \nonumber\\
        & = \int_0^B \mathbb{P}\left(\max_{1\le j\le p}\left |\overline{X(j)}\right |\ge t\right)\,dt \nonumber\\
        & \le \int_0^B \min\left\{1,\sum_{j=1}^p \mathbb{P}\left(\left |\overline{X(j)}\right |\ge t\right)\right\}\,dt \nonumber \\
        & =\int_0^B \min\left\{1,\sum_{j=1}^p \mathbb{P}\left(\left |\overline{Y(j)}\right |\ge t\right)\right\}\,dt \nonumber \\
        & =\mathcal{E}\le 2 \mathbb{E}\left[\max_{1\le j\le p}\left |\overline{Y(j)}\right |\right] =2\E\left[\|\overline{Y}\|_{\infty}\right].\nonumber
    \end{align}
    This completes the proof. \qed

\section{Proofs of Results in Section 3}
\subsection{Proof of Theorem~\ref{thm:prob-ell-q}}\label{pf:prob-ell-q}
Now, consider the random variables $X_i$ when they are divided by $B$. Then, using Lemma \ref{lem:benn-type}, we can say that $\ln \E\left[e^{tS_n}\right]$ (where $S_n=\sum_{i=1}^n X_i/B$) is bounded above by $n(\ell_0(t)+\ell_1(t)+\ell_2(t))$ where 
\begin{align*}
    \ell_0(t) & = \frac{vt^2}{2}\\
    \ell_1(t) & = \sum_{2<k<{q}}v^{({q}-k)/({q}-2)}\frac{t^k}{k!}\\
    \ell_2(t) & = \sum_{k\ge {q}} \left(\frac{K}{B}\right)^{k-{q}}\frac{t^k}{k!}
\end{align*}
where $v=\sigma^2/B^2$. Further, using Equation 2.4 of \cite{MR3652041}, we get 
\[\widetilde{Q}_{S_n/B}(1/z)\le \mathcal{T}(n\ell^*(x)+n\ell^{**}(x))\le \mathcal{T}(n\ell^*(x))+ \mathcal{T}(n\ell^{**}(x))\]
where $x=\log z$ and \[\ell^*(x)=\sum_{k=2}^{\infty}v^{(q-k)/(q-2)}\frac{t^k}{k!}~~\text{and}~~\ell^{**}(x)=\sum_{k=2}^{\infty}\left(\frac{K}{B}\right)^{k-q}\frac{t^k}{k!}.\] Therefore, we obtain that \[\mathcal{T} n\ell^{**}(x)= \inf_{t>0} \frac{n(K/B)^{-{q}}(e^{Kt/B}-1-Kt/B)}{t}+\frac{x}{t}.\]
This infimum will be attained at the solution to the equation \[-\frac{n(K/B)^{-{q}}(e^{Kt/B}-1-Kt/B)}{t^2}+\frac{n(K/B)^{1-{q}}(e^{Kt/B}-1)}{t}-\frac{x}{t^2}=0.\]
The solution to the equation comes out to be \[t^*=\frac{B}{K}\ln\left[1+\Psi^{-1}\left(\left(\frac{K}{B}\right)^{{q}}\frac{x}{n}\right)\right].\]
Consequently, the quantity $\mathcal{T} n\ell^{**}(x)$ turns out to be 
\begin{equation}\label{eq:prob-ell-q-1}
    \mathcal{T} n\ell^{**}(x)=n\left(\frac{K}{B}\right)^{1-{q}}\Psi^{-1}\left(\left(\frac{K}{B}\right)^{q}\frac{x}{n}\right). 
\end{equation}

We adopt a similar approach towards bounding the quantity $\mathcal{T}(n\ell^{*}(x))$. Note that the quantity $\ell^*(t)$ can be expressed as follows:
\begin{align*}
    \ell^*(t)&=\sum_{k=2}^{{\infty}}v^{({q}-k)/({q}-2)}\frac{t^{k}}{k!}= v^{{q}/({q}-2)}\left(e^{v^{-1/({q}-2)}t}-1-v^{-1/({q}-2)}t\right).
\end{align*}
The technique for computing the value of $\mathcal{T} n\ell^{**}(x)$ is reapplied to obtain the value 
\begin{equation}\label{eq:prob-ell-q-2}
    \mathcal{T}n\ell^*(x)= nv^{({q}-1)/({q}-2)}\Psi^{-1}\left[v^{-{q}/({q}-2)}\frac{\ln z}{n}\right].
\end{equation}
Combining \eqref{eq:prob-ell-q-1} and $\eqref{eq:prob-ell-q-2}$, we obtain the result
\[\widetilde{Q}_{S_n}(1/z)\le \min\left\{nB\left(\frac{\sigma^2}{B^2}\right)^{({q}-1)/({q}-2)}\Psi^{-1}\left[\left(\frac{B^2}{\sigma^2}\right)^{{q}/({q}-2)}\frac{\ln z}{n}\right]+nB\left(\frac{K}{B}\right)^{1-q}\Psi^{-1}\left[\left(\frac{K}{B}\right)^{{q}}\frac{\ln z}{n}\right]\right\}.\]
Hence, by Proposition 2.4 of \citep{MR3652041},  
\[\mathbb{P}\left(\frac{1}{n}\sum_{i=1}^n X_i\ge B\left(\frac{\sigma^2}{B^2}\right)^{(q-1)/(q-2)}\Psi^{-1}\left[\left(\frac{B^2}{\sigma^2}\right)^{q/(q-2)}\frac{\ln z}{n}\right]+\frac{B^q}{K^{q-1}}\Psi^{-1}\left[\left(\frac{K}{B}\right)^{{q}}\frac{\ln z}{n}\right]\right)\le \frac{1}{z}.\qed\]

\subsection{Proof of Theorem~\ref{thm:prob-ell-q-version-2}}\label{pf:prob-ell-q-version-2}
     Let
    \[
    S_n = \sum_{i=1}^n X_i.
    \]
    For any $t \ge 0$, we get
    \begin{align*}
        \log\mathbb{E}[e^{tS_n}] &= \sum_{i=1}^n \log\mathbb{E}[e^{tX_i}]\\
        &= \sum_{i=1}^n \log\left(1 + \sum_{k=2}^{\infty} \frac{t^k\mathbb{E}[X_i^k]}{k!}\right)\\
        &\le \sum_{k=2}^{\infty} \frac{t^k}{k!}\sum_{i=1}^n \mathbb{E}[X_i^k]\\
        &\le n\sum_{2 \le k < q} (\sigma^2)^{(q-k)/(q-2)}(B^q)^{(k-2)/(q-2)}\frac{t^k}{k!} + n\sum_{k \ge q} K^{k-q}B^q\frac{t^k}{k!}\\
        &= n\sum_{k=2}^{\infty} (\sigma^2)^{(q-k)/(q-2)}(B^q)^{(k-2)/(q-2)}\frac{t^k}{k!} + n\sum_{k \ge q} \left(K^{k-q}B^q - (\sigma^2)^{(q-k)/(q-2)}(B^q)^{(k-2)/(q-2)}\right)\frac{t^k}{k!}\\
        &=: \ell_1(t) + \ell_2(t).
    \end{align*}
    Observe that 
    \begin{align*}
    \ell_1(t) &= n(\sigma^2/B^2)^{q/(q-2)}\sum_{k=2}^{\infty} \frac{1}{k!} \left(t\left(\frac{B^q}{\sigma^2}\right)^{1/(q-2)}\right)^k\\ 
    &= n(\sigma^2/B^2)^{q/(q-2)}\left(e^{(B^q/\sigma^2)^{1/(q-2)}t} - 1 - {(B^q/\sigma^2)^{1/(q-2)}t}\right).
    \end{align*}
    Note that
    \begin{align*}
    K^{k-q}B^q - (\sigma^2)^{(q-k)/(q-2)}(B^q)^{(k-2)/(q-2)} &= K^{k-q}B^q\left(1 - \left(\frac{B^q}{\sigma^2K^{q-2}}\right)^{(k-q)/(q-2)}\right)\\
    &\le K^{k-q}B^q(k-q)R.
    \end{align*}
     Observe that if $R < 0$, then $\ell_2(t) \le 0$ for all $t \ge 0$. Hence,
    \begin{align*}
        \ell_2(t) &\le n\sum_{k\ge q}\, K^{k-q}B^q(k-q)R\frac{t^k}{k!}\\ 
        &= nRB^q\sum_{k > q} (k-q)K^{k-q}\frac{t^k}{k!}\\
        &\le nRB^q\sum_{k=[q]+1}^{\infty} (k-q)K^{k-q}\frac{t^{k}}{(k-[q])!}\frac{(k-[q])!}{k!}\\
        &\le nRB^q\sum_{k=[q] + 1}^{\infty} K^{k-q}\frac{t^k}{(k-[q]-1)!}\frac{1}{[q]!\binom{k}{[q]}}\\
        &\le \frac{nRB^qt^{[q]+1}K^{[q]-q+1}}{[q]!}\sum_{k=[q]+1}^{\infty} \frac{(tK)^{k-[q]-1}}{(k-[q]-1)!}\\
        &= \frac{nRB^qt^{[q]+1}K^{[q]-q+1}}{[q]!}e^{tK} =: \bar{\ell}_2(t).
    \end{align*}
    Therefore, 
    \[
    \log\mathbb{E}[e^{tS_n}] \le \bar{\ell}_1(t) + \bar{\ell}_2(t).
    \]
    From Eq. (2.4) of~\cite{MR3652041}, we get that
    \[
    \widetilde{Q}_{S_n}(1/z) \le (\mathcal{T}\bar{\ell}_1)(\log z) + (\mathcal{T}\ell_2)(\log z).
    \]
    We now evaluate the quantities on the right-hand side. By Lemma~\ref{lem:calculation-of-T-part1}, we have 
    \begin{align*}
        (\mathcal{T}\bar\ell_1)(\log z) &= n(\sigma^2/B^2)^{q/(q-2)}(B^q/\sigma^2)^{1/(q-2)}\Psi^{-1}\left(\frac{\log z}{n(\sigma^2/B^2)^{q/(q-2)}}\right)\\ 
        &= nB\left(\frac{\sigma^2}{B^2}\right)^{(q-1)/(q-2)}\Psi^{-1}\left(\frac{\log z}{n(\sigma^2/B^2)^{q/(q-2)}}\right),
    \end{align*}
    and by Lemma~\ref{lem:calculation-of-T-part2}, we have
    \begin{equation}\label{eq:bound-for-part-2}
    \begin{split}
    (\mathcal{T}\bar{\ell}_2)(\log z) &\le \frac{2K\log z}{([q]+1)}\left(W\left(\frac{K(\log z)^{1/([q]+1)}}{2([q]+1)(nRB^qK^{1-q+[q]}/[q]!)^{1/([q]+1)}}\right)\right)^{-1}\\
    &\le \frac{2K\log z}{q}\left(W\left(\frac{(K/B)^{q/([q] +1)}(\log z)^{1/([q]+1)}}{10(nR)^{1/([q]+1)}}\right)\right)^{-1},
    \end{split}
    \end{equation}
    where the last inequality follows because $([q]!)^{1/([q] + 1)}/([q]+1) \ge 0.2$ for all $q \ge 2$; see, for example,~\cite{batir2008sharp}.

    One can also use an alternative, simpler bound for $\ell_2(t)$ as
    \begin{align*}
        \ell_2(t) &\le nB^qt^{[q] + 1}\sum_{k = [q] + 1}^{\infty} K^{k-q}\frac{t^{k-[q] - 1}}{k!}\\
        &\le nB^qt^{[q] + 1}K^{1 - q + [q]}\sum_{k=[q] + 1}^{\infty} \frac{(Kt)^{k-[q]-1}}{(k-[q]-1)!([q]+1)!}\frac{1}{\binom{k}{[q]+1}} \\
        &\le \frac{nB^qK^{1-q+[q]}t^{[q]+1}}{([q]+1)!}e^{tK}.
    \end{align*}
    Therefore, 
    \begin{align*}
    (\mathcal{T}\ell_2)(\log z) &\le \frac{2K\log z}{[q] + 1}\left(W\left(\frac{K(\log z)^{1/([q]+1)}}{2([q]+1)(nB^qK^{1-q+[q]}/([q]+1)!)^{1/([q]+1)}}\right)\right)^{-1}\\
    &\le \frac{2K\log z}{q}\left(W\left(\frac{(K/B)^{q/([q]+1)}(\log z)^{1/([q]+1)}}{2en^{1/([q]+1)}}\right)\right)^{-1}
    \end{align*}
    Because $R \le 1$, this bound is weaker than the bound in~\eqref{eq:bound-for-part-2}.\qed

\subsection{Proof of Proposition \ref{thm:e_q(s,b)-upper}}\label{pf:e_q(s,b)-upper}
By the Triangle Inequality,
    \[\E\left[\left\|\frac{1}{n}\sum_{i=1}^n \bX_i\right\|_{\infty}\right]\le \E\left[\left\|\frac{1}{n}\sum_{i=1}^n \bX_i\boldsymbol{1}\{\|\bX_i\|_{\infty}\le K\}\right\|_{\infty}\right]+\E\left[\left\|\frac{1}{n}\sum_{i=1}^n \bX_i\boldsymbol{1}\{\|\bX_i\|_{\infty}> K\}\right\|_{\infty}\right].\]
    Towards bounding the second term, observe that
    \begin{align}
        \E\left[\left\|\sum_{i=1}^n \bX_i\boldsymbol{1}\{\|\bX_i\|_{\infty}> K\}\right\|_{\infty}\right] & \le \sum_{i=1}^n\E\left[\|\bX_i\boldsymbol{1}\{\|\bX_i\|_{\infty}> K\}\|_{\infty}\right] \nonumber \\
        & = \sum_{i=1}^n\int_{K}^{\infty}\mathbb{P}\left(\|\bX_i\|_{\infty}>t\right)\,dt \nonumber \\
        & = \sum_{i=1}^n\int_{K}^{\infty}\frac{{q} t^{{q}-1}\mathbb{P}\left(\|\bX_i\|_{\infty}>t\right)}{{q} t^{{q}-1}}\,dt \nonumber \\
        & \le \frac{1}{{q} K^{{q}-1}}\sum_{i=1}^n \int_{0}^{\infty}{q} t^{{q}-1}\mathbb{P}\left(\|\bX_i\|_{\infty}>t\right)\,dt \le \frac{nB^{q}}{{q} K^{{q}-1}}. \nonumber 
    \end{align}
    As for the first term, having observed that the random variables $\bU_i=\bX_i\boldsymbol{1}\{\|\bX_i\|_{\infty}\le K\}$ satisfy 
    \[\max_{1\le j\le p}\frac{1}{n}\sum_{i=1}^n \E\left[U_i^2(j)\right]\le \sigma^2, \max_{1\le j\le p}\frac{1}{n}\sum_{i=1}^n \E\left[\left|U_i(j)\right|^{{q}}\right]\le B^{q}, \max_{1\le j\le p}\|\bU_i\|_{\infty}\le K,\] we seek to bound the expected $L_{\infty}$ norm of the mean of the $\bU_{i}$'s. To that end, we utilize the result in Theorem \ref{thm:prob-ell-q}. Choosing 
    \[z^*=B\left(\frac{\sigma^2}{B^2}\right)^{\frac{{q}-1}{{q}-2}}\Psi^{-1}\left[\left(\frac{B^2}{\sigma^2}\right)^{\frac{{q}}{{q}-2}}\frac{\ln (2p)}{n}\right]+\frac{B^{q}}{K^{{q}-1}}\Psi^{-1}\left[\left(\frac{K}{B}\right)^{{q}}\frac{\ln (2p)}{n}\right]\]
    observe that
    \begin{align*}
        \E\left[\left\|\frac{1}{n}\sum_{i=1}^n \bU_i\right\|_{\infty}\right] & = \int_{0}^{\infty} \mathbb{P}\left(\left\|\frac{1}{n}\sum_{i=1}^n \bU_i\right\|_{\infty}\ge x\right)\,dx \\
        & \le \int_{0}^{z^{*}}\,dx+\int_{z^{*}}^{\infty}2p\mathbb{P}\left(\frac{1}{n}\sum_{i=1}^n U_{i}(1)\ge x\right)\,dx.
    \end{align*}
    The value of the first integral equals $z^*$. As for the second integral, we make the variable change
    \[x=B\left(\frac{\sigma^2}{B^2}\right)^{\frac{{q}-1}{{q}-2}}\Psi^{-1}\left[\left(\frac{B^2}{\sigma^2}\right)^{\frac{{q}}{{q}-2}}\frac{\ln z}{n}\right]+\frac{B^{q}}{K^{{q}-1}}\Psi^{-1}\left[\left(\frac{K}{B}\right)^{{q}}\frac{\ln z}{n}\right]\]
    and get the following:
    \begin{align*}
        &\int_{z^*}^{\infty}2p\mathbb{P}\left(\frac{1}{n}\sum_{i=1}^n U_i(1)\ge x\right)\,dx \\
        & \le \int_{2p}^{\infty}\frac{2p}{nz^2}\left[B\left(\frac{\sigma^2}{B^2}\right)^{\frac{{q}-1}{{q}-2}}\frac{(B^2/\sigma^2)^{\frac{{q}}{{q}-2}}}{\Psi'(\Psi^{-1}((B^2/\sigma^2)^{\frac{{q}}{{q}-2}}(\ln z)/n))}+\frac{B^{q}}{K^{{q}-1}}\frac{(K/B)^{q}}{\Psi'(\Psi^{-1}(K/B)^{q}(\ln z)/n)}\right]\,dz\\
        & \le \int_{2p}^{\infty}\frac{2p}{nz^2}\left[B\left(\frac{\sigma^2}{B^2}\right)^{\frac{{q}-1}{{q}-2}}\frac{(B^2/\sigma^2)^{\frac{{q}}{{q}-2}}}{\Psi'(\Psi^{-1}((B^2/\sigma^2)^{\frac{{q}}{{q}-2}}(\ln (2p))/n))}+\frac{B^{q}}{K^{{q}-1}}\frac{(K/B)^{q}}{\Psi'(\Psi^{-1}(K/B)^{q}(\ln (2p))/n)}\right]\,dz\\
        & = \frac{B}{n}\left(\frac{\sigma^2}{B^2}\right)^{\frac{{q}-1}{{q}-2}}\frac{(B^2/\sigma^2)^{\frac{{q}}{{q}-2}}}{\Psi'(\Psi^{-1}((B^2/\sigma^2)^{\frac{{q}}{{q}-2}}(\ln (2p))/n))}+\frac{B^{q}}{nK^{{q}-1}}\frac{(K/B)^{q}}{\Psi'(\Psi^{-1}(K/B)^{q}(\ln (2p))/n)}\\
        & \le B\left(\frac{\sigma^2}{B^2}\right)^{\frac{{q}-1}{{q}-2}}\frac{(B^2/\sigma^2)^{\frac{{q}}{{q}-2}}(\ln (2p) /n)}{\Psi'(\Psi^{-1}((B^2/\sigma^2)^{\frac{{q}}{{q}-2}}(\ln (2p))/n))}+\frac{B^{q}}{K^{{q}-1}}\frac{(K/B)^{q}(\ln (2p)/n)}{\Psi'(\Psi^{-1}(K/B)^{q}(\ln (2p))/n)}\\
        & \le B\left(\frac{\sigma^2}{B^2}\right)^{\frac{{q}-1}{{q}-2}}\Psi^{-1}\left[\left(\frac{B^2}{\sigma^2}\right)^{\frac{{q}}{{q}-2}}\frac{\ln (2p)}{n}\right]+\frac{B^{q}}{K^{{q}-1}}\Psi^{-1}\left[\left(\frac{K}{B}\right)^{{q}}\frac{\ln (2p)}{n}\right]=z^*.
    \end{align*}
    The last inequality holds due to the following chain of inequalities for all $y=\Psi^{-1}(z)$ with $y,z\ge 0$:
    \begin{align*}
        &\ln(1+y)\le y\\
        &\implies \ln(1+y)+y\ln(1+y)-y\le y\ln(1+y)\\
        &\implies \Psi(y)\le y\Psi'(y)\\
        & \implies \Psi(\Psi^{-1}(z))\le \Psi^{-1}(z)\Psi'(\Psi^{-1}(z))\\
        & \implies \frac{z}{\Psi'(\Psi^{-1}(z))}\le \Psi^{-1}(z).
    \end{align*}
    We therefore find that 
    \begin{align}
        &\E\left[\left\|\frac{1}{n}\sum_{i=1}^{n}\bX_i\right\|_{\infty}\right]\nonumber\\
         \le & 2B\left(\frac{\sigma^2}{B^2}\right)^{\frac{{q}-1}{{q}-2}}\Psi^{-1}\left[\left(\frac{B^2}{\sigma^2}\right)^{\frac{{q}}{{q}-2}}\frac{\ln (2p)}{n}\right]+\frac{2B^{q}}{K^{{q}-1}}\Psi^{-1}\left[\left(\frac{K}{B}\right)^{{q}}\frac{\ln (2p)}{n}\right]+\frac{B^{q}}{{q}K^{{q}-1}}. \label{eq:exp-bound-form-ell-q}
    \end{align}
    We partition the rest of our analysis into two cases:\\
    \textbf{\underline{Case 1:}} $(B^2/\sigma^2)^{\frac{{q}}{{q}-2}}(\ln (2p)/n)\le e$. In this situation, by Lemma \ref{lem:benn-psi-inv-bd}, we can say that 
    \begin{align*}
        &\E\left[\left\|\frac{1}{n}\sum_{i=1}^{n}\bX_i\right\|_{\infty}\right]\\
        \le & 4\sqrt{2}B\left(\frac{\sigma^2}{B^2}\right)^{\frac{{q}-1}{{q}-2}}\left(\frac{B}{\sigma}\right)^{\frac{{q}}{{q}-2}}\sqrt{\frac{\ln (2p)}{n}}+\frac{2B^{q}}{K^{{q}-1}}\Psi^{-1}\left[\left(\frac{K}{B}\right)^{{q}}\frac{\ln (2p)}{n}\right]+\frac{B^{q}}{{q}K^{{q}-1}}\\
        =& 4\sqrt{2}\sigma\sqrt{\frac{\ln (2p)}{n}}+\frac{2B^{q}}{K^{{q}-1}}\Psi^{-1}\left[\left(\frac{K}{B}\right)^{{q}}\frac{\ln (2p)}{n}\right]+\frac{B^{q}}{{q}K^{{q}-1}}
    \end{align*}
    Now, setting $K = (\sqrt{e}(B^{q}/\sigma) \sqrt{n/\ln (2p)})^{1/({q}-1)}$, we find that $B^{q}/K^{{q}-1}=(1/\sqrt{e})\sigma\sqrt{\ln (2p)/n}$. Note that the above choice of $K$ is feasible i.e. $K\ge B$ as $\sqrt{n/\ln (2p)}\ge (1/\sqrt{e})(B/\sigma)^{\frac{{q}}{{q}-2}}$ and hence $K/B=(B^2/\sigma^2)^{1/({q}-2)}>1$. Further \[\left(\frac{K}{B}\right)^{q}\frac{\ln (2p)}{n}=(\sqrt{e})^{\frac{q}{{q}-1}}\left(\frac{B}{\sigma}\right)^{\frac{q}{{q}-1}}\left(\sqrt{\frac{\ln (2p)}{n}}\right)^{\frac{{q}-2}{{q}-1}}\le (\sqrt{e})^{\frac{q}{{q}-1}}(\sqrt{e})^{\frac{{q}-2}{{q}-1}}=e.\]
    Therefore, we find that 
    \[\E\left[\left\|\frac{1}{n}\sum_{i=1}^{n}\bX_i\right\|_{\infty}\right]\le 10 \sigma \sqrt{\frac{\ln (2p)}{n}}.\]

    \textbf{\underline{Case 2:}} $(B^2/\sigma^2)^{\frac{{q}}{{q}-2}}(\ln (2p)/n)> e$. In this case, we choose $K$ to be $(B^{q}/\sigma^2)^{1/({q}-2)}$. For this choice of $K$, note that 
    \[\left(\frac{K}{B}\right)^{q}\frac{\ln (2p)}{n}=\left(\frac{B^2}{\sigma^2}\right)^{\frac{{q}}{{q}-2}}\frac{\ln (2p)}{n}>e\]
    Further notice that for $x>e$, \[\frac{2}{{q}}\le 1\le \frac{x}{\ln(1+x)}\le \Psi^{-1}(x)\le \frac{2x}{\ln(1+x)}.\]
    Therefore, 
    \[\E\left[\left\|\frac{1}{n}\sum_{i=1}^{n}\bX_i\right\|_{\infty}\right]\le \frac{9}{2}B\left(\frac{\sigma^2}{B^2}\right)^{\frac{{q}-1}{{q}-2}}\Psi^{-1}\left[\left(\frac{B^2}{\sigma^2}\right)^{\frac{{q}}{{q}-2}}\frac{\ln (2p)}{n}\right]\le \frac{9B(B^2/\sigma^2)^{1/({q}-2)}(\ln (2p)/n)}{\ln(1+(B^2/\sigma^2)^{\frac{{q}}{{q}-2}}(\ln (2p)/n))}.\qed\]

\subsection{Proof of Proposition \ref{thm:lq-benn-upper}}\label{pf:lq-benn-upper}

By the Triangle Inequality,
    \[\E\left[\left\|\frac{1}{n}\sum_{i=1}^n \bX_i\right\|_{\infty}\right]\le \E\left[\left\|\frac{1}{n}\sum_{i=1}^n \bX_i\boldsymbol{1}\{\|\bX_i\|_{\infty}\le K\}\right\|_{\infty}\right]+\E\left[\left\|\frac{1}{n}\sum_{i=1}^n \bX_i\boldsymbol{1}\{\|\bX_i\|_{\infty}> K\}\right\|_{\infty}\right].\]
    Towards bounding the second term, observe that
    \begin{align}
        \E\left[\left\|\sum_{i=1}^n \bX_i\boldsymbol{1}\{\|\bX_i\|_{\infty}> K\}\right\|_{\infty}\right] & \le \sum_{i=1}^n\E\left[\bX_i\boldsymbol{1}\{\|\bX_i\|_{\infty}> K\}\|_{\infty}\right] \nonumber \\
        & = \sum_{i=1}^n\int_{K}^{\infty}\mathbb{P}\left(\|\bX_i\|_{\infty}>t\right)\,dt \nonumber \\
        & = \sum_{i=1}^n\int_{K}^{\infty}\frac{qt^{q-1}\mathbb{P}\left(\|\bX_i\|_{\infty}>t\right)}{qt^{q-1}}\,dt \nonumber \\
        & \le \frac{1}{qK^{q-1}}\sum_{i=1}^n \int_{0}^{\infty}qt^{q-1}\mathbb{P}\left(\|\bX_i\|_{\infty}>t\right)\,dt \le \frac{nB^q}{qK^{q-1}}. \nonumber 
    \end{align}
    As for the first term, having observed that the random variables $\bU_i=\bX_i\boldsymbol{1}\{\|\bX_i\|_{\infty}\le K\}$ satisfy 
    \[\max_{1\le j\le p}\frac{1}{n}\sum_{i=1}^n \E\left[U_i^2(j)\right]\le \sigma^2, \max_{1\le j\le p}\frac{1}{n}\sum_{i=1}^n \E\left[\left|U_i(j)\right|^q\right]\le B^q, \max_{1\le j\le p}\|\bU_i\|_{\infty}\le K,\] we bound it by making use of the Bennett's Inequality as in the proof of Theorem \ref{thm:ind-bern}. to obtain
    \begin{equation}\label{eq:lq-benn-1}
        \E\left[\left\|\sum_{i=1}^n \bX_i\boldsymbol{1}\{\|\bX_i\|_{\infty}\le K\}\right\|_{\infty}\right] \le nA_0+\frac{K}{\ln(1+A_0K/\sigma^2)}
    \end{equation}
    where $A_0=(\sigma^2/K)\Psi^{-1}\left(K^2\ln (2p)/(n\sigma^2)\right)$. Consequently, we have 
    \begin{equation}\label{eq:lq-benn-2}
        \E\left[\left\|\sum_{i=1}^n \bX_i\right\|_{\infty}\right]\le nA_0+\frac{K}{\ln(1+A_0K/\sigma^2)}+\frac{nB^q}{qK^{q-1}}.
    \end{equation}
    Appealing to Lemma \ref{lem:psi-bd-small-x} and \ref{lem:psi-bd-large-x}, therefore we obtain \[\frac{1}{\sqrt{2}}\ln\left(1+K^2\ln (2p)/(n\sigma^2)\right)\le\ln\left(1+A_0K/\sigma^2\right)\]
    Further, using Lemma \ref{lem:benn-psi-inv-bd} and the increasingness of the function $x/\ln(1+x)$, we have 
    \begin{equation}\label{eq:lq-benn-3}
        \E\left[\left\|\sum_{i=1}^n \bX_i\right\|_{\infty}\right]\le \frac{(e\ln (2p) + \sqrt{2})K}{\ln(1+K^2\ln (2p)/(n\sigma^2))}+\frac{nB^q}{qK^{q-1}}\le \frac{(e\ln (2p) + \sqrt{2})K}{\ln(1+K\sqrt{\ln (2p)/(n\sigma^2)})}+\frac{nB^q}{qK^{q-1}}
    \end{equation}
    Since $K$ is a free parameter, we can improve the upper bound to 
    \[\E\left[\left\|\sum_{i=1}^n \bX_i\right\|_{\infty}\right]\le \inf_{K>0}\frac{(e\ln (2p) + \sqrt{2})K}{\ln(1+K\sqrt{\ln (2p)/(n\sigma^2)})}+\frac{nB^q}{qK^{q-1}}\]
Clearly, the right hand side of this inequality is differentiable and we can find the infimum by solving the equation
\begin{equation}\label{eq:lq-benn-4}
    \frac{\Psi(K\sqrt{\ln (2p)/(n\sigma^2)})}{(1+K\sqrt{\ln (2p)/(n\sigma^2)})\ln^2(1+K\sqrt{\ln (2p)/(n\sigma^2)})}=\frac{1}{(e\ln (2p)+\sqrt{2})}\frac{n(q-1)B^q}{qK^q}
\end{equation}
However, it is not possible to solve this equation analytically, and hence we resort to an approximation of the function on the left hand side. By Lemma \ref{lem:c-1}, when $x>\sqrt{e}$,
\begin{equation}\label{eq:lq-benn-5}
    \frac{1}{6\ln x}\le\frac{1}{3\ln (1+x)}\le \frac{\Psi(x)}{(1+x)\ln^2(1+x)}\le \frac{1}{\ln (1+x)}\le \frac{1}{\ln x}.
\end{equation}
It therefore makes more sense for us to solve an equation of the form 
\begin{equation}\label{eq:lq-benn-6}
    \frac{1}{\ln(K\sqrt{\ln (2p)/(n\sigma^2)})}=\frac{n(q-1)B^q}{(e\ln (2p) +\sqrt{2})qK^q}.
\end{equation}
This can be written as an equation of the form 
\begin{equation}\label{eq:lq-benn-7}
    \frac{x^q}{q\ln(x)}=c
\end{equation}
Here $K\sqrt{\ln (2p)/(n\sigma^2)}$ plays the role of $x$. There may be two types of solutions to this equation. One is simply \[K\sqrt{\ln (2p)/(n\sigma^2)} = \left[-DW(-1/D)\right]^{1/q}\] where $W$ is the Lambert function. However, by Lemma \ref{lem:c-3}, this solution is bounded between $e^{1/q}$ and $1$. If $q>2$, then the maximum value that this solution can attain is $\le \sqrt{e}$ whereas the minimum value of $K\sqrt{\ln (2p)/(n\sigma^2)}$ is $\sqrt{e}$. Therefore, we discard this solution due to its limited scope of use. By Lemma \ref{lem:c-2},
\[K\sqrt{\ln 2p/(n\sigma^2)}=\left[D\ln D\ln D\ln D\ldots\right]^{1/q},\]
where 
\[D = \frac{(q-1)n\left(B\sqrt{\ln (2p)/(n\sigma^2)}\right)^q}{(e\ln (2p)+\sqrt{2})q^2}.\]

Note that here we are implicitly assuming that $K\sqrt{\ln (2p)}/(n\sigma^2)>e$, which allows us to argue the optimality of the bound (up to constant multiples) when $[D\ln D]^{1/q}\ge e$. Nevertheless, we can always plug in this value of $K$ to obtain an upper bound, which may or may not be optimal in all situations. However, it is to be noted that the bound is meaningful when $K^2\ln(2p)/n\sigma^2>1$ which in turn implies that \[[D\ln D\ln D\ldots]^{1/q}\ge 1\] which follows if we make sure that $[D\ln D]^{1/q}\ge 1$. Note that 
\begin{equation}\label{eq:optim-choice-K}
    [D\ln D]^{1/q}\asymp 2(B/\sigma)(\ln(2p)/n)^{1/2-1/q}\left[\ln\left((B/\sigma)(\ln(2p)/n)^{1/2-1/q}\right)\right]^{1/q}\ge 1
\end{equation} when $(B^2/\sigma^2)^{q/(q-2)}(\ln (2p)/n)>e$.

Plugging this value of $K$ into \eqref{eq:lq-benn-3}, we have 
\begin{align}
    \E\left[\left\|\frac{1}{n}\sum_{i=1}^n \bX_i\right\|_{\infty}\right] & \le \frac{(e\ln (2p)+\sqrt{2})K}{n\ln(K\sqrt{\ln (2p)/ (n\sigma^2)})}+\frac{B^q}{qK^{q-1}}\nonumber \\
    & = K\left[\frac{(e\ln (2p)+\sqrt{2})}{n\ln(K\sqrt{\ln (2p)/(n\sigma^2)})}+\frac{B^q}{qK^q}\right]\nonumber \\
    & = K\left[\frac{(q-1)B^q}{qK^q}+\frac{B^q}{qK^q}\right]= \frac{B^q}{K^{q-1}}\nonumber \\
    & \le B\left[\frac{n(q-1)}{(e\ln (2p)+\sqrt{2})q^2}\ln\left(\frac{(q-1)n\left(B\sqrt{\ln (2p)/(n\sigma^2)}\right)^q}{(e\ln (2p)+\sqrt{2})q^2}\right)\right]^{1/q-1}\label{eq:lq-benn-8}\\
    &\lesssim B\left(\frac{\ln (2p)}{n}\right)^{1-1/q}\left[\ln\left(\frac{B^2}{\sigma^2}\left(\frac{\ln (2p)}{n}\right)^{1-2/q}\right)\right]^{1/q-1}.\nonumber
\end{align}
Even when $[D\ln D]^{1/q}<1$, we may argue that it is only off by a constant, whereby, we may directly choose $(K/\sigma)\sqrt{\ln (2p)/n}$ to be the quantity on the right of \eqref{eq:optim-choice-K} and we will obtain the same bound as follows. 
\[K=2B(\ln(2p)/n)^{-1/q}\left[\ln\left((B/\sigma)(\ln(2p)/n)^{1/2-1/q}\right)\right]^{1/q}.\] It is thus easy to see that both \[\frac{B^q}{K^{q-1}}\lesssim B\left(\frac{\ln (2p)}{n}\right)^{1-1/q}\left[\ln\left(\frac{B^2}{\sigma^2}\left(\frac{\ln (2p)}{n}\right)^{1-2/q}\right)\right]^{1/q-1}\]and \[\frac{K(e\ln (2p)+\sqrt{2})}{n\ln(K\sqrt{\ln (2p)/(n\sigma^2)})}\lesssim B\left(\frac{\ln (2p)}{n}\right)^{1-1/q}\left[\ln\left(\frac{B^2}{\sigma^2}\left(\frac{\ln (2p)}{n}\right)^{1-2/q}\right)\right]^{1/q-1}\] and the bound holds either way.\qed

\section{Auxiliary Lemmas and their proofs}
\begin{lemma}\label{lem:unique-soln}
    Equation \eqref{eq:K-solution} has a unique solution in $K$.
\end{lemma}

\begin{proof}
    Recall that Eq.\eqref{eq:K-solution} is as follows: \[(1 + 2q)\left[\left(\frac{\sigma^2}{5K}\right)^q\left(\frac{5K^2}{5K^2 + \sigma^2}\right)^d + K^q\left(1 - \left(\frac{5K^2}{5K^2 + \sigma^2}\right)^d\right)\right] ~=~ B^q\] where $B\ge \sigma$ and $\sqrt{5}K\ge \sigma$ by the definition of the random variables $W_j$ in the proof of Lemma \ref{lem:M-q-sig-B}. Observe further that 
    \begin{align}
        &(1 + 2q)\left[\left(\frac{\sigma^2}{5K}\right)^q\left(\frac{5K^2}{5K^2 + \sigma^2}\right)^d + K^q\left(1 - \left(\frac{5K^2}{5K^2 + \sigma^2}\right)^d\right)\right] ~=~ B^q\nonumber\\
        \implies & \left(\frac{\sigma^2}{5K}\right)^q\left(\frac{5K^2}{5K^2 + \sigma^2}\right)^d + K^q\left(1 - \left(\frac{5K^2}{5K^2 + \sigma^2}\right)^d\right)~=~\frac{B^q}{1+2q}\nonumber\\
        \implies & \left(\frac{\sigma}{\sqrt{5}K}\right)^q\left(\frac{5K^2}{5K^2 + \sigma^2}\right)^d + \left(\frac{\sqrt{5}K}{\sigma}\right)^q\left(1 - \left(\frac{5K^2}{5K^2 + \sigma^2}\right)^d\right)~=~\frac{1}{(1+2q)}\left(\frac{B}{\sqrt{5}\sigma}\right)^q\label{eq:lem-unique-soln}
    \end{align}
    Taking $y=\sqrt{5}K/\sigma$, we will show that the function \[F(y)=y^{-q}\left(\frac{y^2}{1+y^2}\right)^d+y^q\left(1-\left(\frac{y^2}{1+y^2}\right)^d\right)\] is strictly increasing in $y$ whereby, Eq~\eqref{eq:lem-unique-soln} has a unique solution in $\sqrt{5}K/\sigma$ and consequently, in $K$. To show that the function $F$ is strictly increasing in $y$ for all $y\ge 1$, we take the derivative of $F$ with respect to $y$ which comes out to be \[F'(y)=(1+y^2)^{-(d+1)}(qy^{q-1}(1+y^2)^{d+1}-qy^{2d-q-1}(1+y^2)+2dy^{2d-q-1}-qy^{2d
+q-1}(1+y^2)-2dy^{2d+q-1}).\]
To show that this is positive, we only need to show that \[qy^{q-1}(1+y^2)^{d+1}>(q-2d)y^{2d-q-1}+qy^{2d-q+1}-(q+2d)y^{2d+q-1}+qy^{2d+q+1}.\]
We first consider the case where $d\ge 2$. Observe that the coefficients of 4 largest powers of $y$ in the expansion of $qy^{q-1}(1+y^2)^{d+1}$ are $q, q(d+1), qd(d+1)/2, qd(d^2-1)/6$. It is easy to see therefore, that the first term in the expansion cancels out with $qy^{2d+q+1}$. As for the second term, note that for $d\ge 2$, we have $q\ge 2\ge d/(d-1)$ which implies that $qd+d\ge q+2d$ and hence $q(d+1)y^{2d+q-1}>(q+2d)y^{2d+q-1}$. as for the third and the fourth terms, not only are the coefficients in the expansion larger, the powers of $y$ are higher and given that $y\ge 1$, the inequality is proved. For the case $d=1$, we are required to show that \[qy^{q-1}(1+2y^2+y^4)>qy^{q+3}+(q+2)y^{q+1}+qy^{3-q}+(q-2)y^{1-q}.\]
Note that 
\begin{align*}
    &qy^{q-1}(1+2y^2+y^4)-(qy^{q+3}+(q+2)y^{q+1}+qy^{3-q}+(q-2)y^{1-q})\\
    =& (q-2)y^{q+1}-(q-2)y^{1-q}+qy^{q-1}-q^{3-q}>0
\end{align*}
This proves the result.
\end{proof}

\begin{lemma}\label{lem:mills-general}
    If $\psi(x)$ is a twice differentiable and 
 strictly increasing, non-negative convex function on $[a,b)$ such that $\psi'(a)>0$, then \[\int_{a}^{b}\exp(-\psi(x))\,dx \le \frac{1}{\psi'(a)}\exp(-\psi(a))-\frac{1}{\psi'(b)}\exp(-\psi(b)).\]
\end{lemma}

\begin{proof}
    The first observation we make is that \[\int_{a}^{b}\exp(-\psi(x))\,dx = \int_{a}^b \frac{1}{\psi'(x)}\psi'(x)\exp(-\psi(x))\,dx.\]
    Now, applying integration by parts, we get 
    \begin{align*}
        & \int_{a}^b \frac{1}{\psi'(x)}\psi'(x)\exp(-\psi(x))\,dx\\
        & = -\frac{1}{\psi'(x)}\exp(-\psi(x))\bigg|_a^b - \int_a^b \frac{\psi''(x)}{\psi'(x)^2}\exp(-\psi(x))\,dx\\
        & = \frac{1}{\psi'(a)}\exp(-\psi(a))-\frac{1}{\psi'(b)}\exp(-\psi(b)) - \int_a^b \frac{\psi''(x)}{\psi'(x)^2}\exp(-\psi(x))\,dx\\
        & \le \frac{1}{\psi'(a)}\exp(-\psi(a))-\frac{1}{\psi'(b)}\exp(-\psi(b))
    \end{align*}
    The last inequality follows from the fact that $\psi(x)$ being convex and twice differentiable, has a positive second derivative. This implies that the integrand (of the integral in the penultimate step) and hence the integral is positive.
\end{proof}

\begin{lemma}\label{lem:lambert-bd}
    For $x\ge 0$, let $W(x)$ be the principal branch of the Lambert function defined by $W(x)e^{W(x)}=x$. Then we have
    \begin{equation}\label{eq:lambert-bd}
        \left(1-\frac{1}{e}\right)\ln(1+x)\le W(x)\le \ln(1+x)
    \end{equation}
\end{lemma}

\begin{proof}
    The upper bound on the Lambert function is easy to show. Observe that $\Psi(x)=(1+x)\ln(1+x)-x\ge 0$. This implies that we must have $(1+x)\ln(1+x)\ge x=W(x)e^{W(x)}$ which in turn implies that $\ln(1+x)\ge W(x)$. 

    As for the lower bound, it suffices to show that \[\left(1-\frac{1}{e}\right)(1+x)^{(1-1/e)}\ln(1+x)\le x\]
    Taking $T(x)=\left(1-{1}/{e}\right)(1+x)^{(1-1/e)}\ln(1+x)-x$, it is enough to show that $T'(x)\le 0$ for any $x> 0$. Now,
    \[T'(x)=\frac{\left(1-\frac{1}{e}\right)^2\ln(1+x)}{(1+x)^{1/e}}+\frac{\left(1-\frac{1}{e}\right)}{(1+x)^{1/e}}-1\]
    If we can show that \[R(x)=\frac{\left(1-\frac{1}{e}\right)\ln(1+x)+1}{(1+x)^{1/e}}\le \frac{e}{e-1}\] then we are done. Note that $R(x)$ has a unique maximum at $x=e^{e(e-2)/(e-1)}-1$. The maximum value attained by $R(x)$ is therefore ${(e-1)}/{e^{(e-2)/(e-1)}}<{e}/(e-1)$. This proves the result.
\end{proof}

\begin{lemma}\label{lem:psi-bd-small-x}
    For $x\in[0,e]$, $\Psi'(\Psi^{-1}(x))$ is bounded as follows:
    \begin{equation*}
        \frac{1}{\sqrt{2}}\ln(1+x)\le \sqrt{2}\ln(1+\sqrt{x})\le \Psi'(\Psi^{-1}(x)) \le \frac{3}{2}\ln(1+\sqrt{x})
    \end{equation*}
\end{lemma}
\begin{proof}
    We begin by observing that showing the validity of the lower bound boils down to an equivalent problem:
    \begin{align*}
        &\Psi'(\Psi^{-1}(x))\ge \sqrt{2}\ln(1+\sqrt{x})\\
         \Longleftrightarrow &\ln(1+\Psi^{-1}(x))\ge \sqrt{2}\ln(1+\sqrt{x})\\
         \Longleftrightarrow &\Psi^{-1}(x)\ge (1+\sqrt{x})^{\sqrt{2}}-1\\
         \Longleftrightarrow &\sqrt{2}(1+\sqrt{x})^{\sqrt{2}}\ln(1+\sqrt{x})-(1+\sqrt{x})^{\sqrt{2}}+1-x\le 0
    \end{align*}
    Take $\zeta(x)=\sqrt{2}(1+\sqrt{x})^{\sqrt{2}}\ln(1+\sqrt{x})-(1+\sqrt{x})^{\sqrt{2}}+1-x$. Since $\zeta(0)=0$, in order for us to show that $\zeta(x)\le 0$ for $x\in [0,e]$, it suffices to show that it is a decreasing function. As \[\zeta'(x)=\frac{(1+\sqrt{x})^{\sqrt{2}-1}}{\sqrt{x}}\ln(1+\sqrt{x})-1\] and \[\lim_{x\to 0}\frac{\ln(1+\sqrt{x})}{\sqrt{x}}=1\], it is enough to prove that \[\ln(1+\sqrt{x})\le \frac{\sqrt{x}}{(1+\sqrt{x})^{\sqrt{2}-1}}\]
    Taking $\zeta_1(x)=\ln(1+\sqrt{x})- {\sqrt{x}}/{(1+\sqrt{x})^{\sqrt{2}-1}}$, observe that 
    \begin{align*}
        \zeta_1'(x)&=\frac{(1+\sqrt{x})^{-\sqrt{2}-1}((\sqrt{2}-2)x+(\sqrt{2}-3)\sqrt{x}+(1+\sqrt{x})^{\sqrt{2}}-1)}{2\sqrt{x}}\\
        & =\frac{(1+\sqrt{x})^{-\sqrt{2}-1}((1+\sqrt{x})^{\sqrt{2}}-(1+\sqrt{x})+(\sqrt{2}-2)\sqrt{x}(1+\sqrt{x}))}{2\sqrt{x}}\\
        & =\frac{(1+\sqrt{x})^{-\sqrt{2}}((1+\sqrt{x})^{\sqrt{2}-1}+(\sqrt{2}-2)\sqrt{x}-1)}{2\sqrt{x}}
    \end{align*}
    As $\lim_{x\to 0}\zeta_1'(x)=\sqrt{2}-3/2<0$, the fact that $\zeta_1$ is decreasing follows if we can show that for $x\in [0,e]$, $\zeta_2(x)=(1+x)^{\sqrt{2}-1}+(\sqrt{2}-2)\sqrt{x}-1\le 0$. This is easy to see since $\zeta_2(0)=0$ and $\zeta_2'(x)=(\sqrt{2}-1)(1+x)^{\sqrt{2}-2}+(\sqrt{2}-2)$ is decreasing and takes the value $2\sqrt{2}-3$ at $x=0$.\\

    To prove the right side of the inequality, we observe that 
    \begin{align*}
        &\Psi'(\Psi^{-1}(x))\le \frac{3}{2}\ln(1+\sqrt{x})\\
         \Longleftrightarrow &\ln(1+\Psi^{-1}(x))\le \frac{3}{2}\ln(1+\sqrt{x})\\
         \Longleftrightarrow &\Psi^{-1}(x)\le (1+\sqrt{x})^{3/2}-1\\
         \Longleftrightarrow &\zeta_3(x)=\frac{3}{2}(1+\sqrt{x})^{3/2}\ln(1+\sqrt{x})-(1+\sqrt{x})^{3/2}+1-x\ge 0
    \end{align*}
    Then as $\zeta_3(0)=0$, it suffices to show that \[\zeta_3'(x)=\frac{9}{8}\frac{(1+\sqrt{x})^{1/2}}{\sqrt{x}}\ln(1+\sqrt{x})-1\ge 0\]
    Observe that we need only show that \[\frac{9}{8}\frac{(1+x)^{1/2}}{x}\ln(1+x)-1\ge 0\Longleftrightarrow \zeta_4(x)=\frac{9}{8}\ln(1+x)-\frac{x}{(1+x)^{1/2}}\] for $x\in [0,\sqrt{e}]$.
    Observe that \[\zeta_4'(x)=\frac{-8(1+x)+9\sqrt{1+x}+4x}{8(1+x)^{3/2}}\]
    Taking $y=\sqrt{1+x}$, note that $\zeta_4'(x)\ge 0$ if and only if $4y^2-9y+4\le 0$ which happens if and only if ${(9-\sqrt{17})}/{8}\le y\le {(9+\sqrt{17})}/{8}$. In this case, as $x\in[0,\sqrt{e}]$, therefore $y\in [1,\sqrt{1+\sqrt{e}}]$ and hence we must have $\zeta_4'(x)\ge 0$. Also, $\zeta_4(0)=0$ and hence, the result follows.\\

    Lastly, 
    \begin{align*}
        2\sqrt{x}\ge 0 & \implies 1+x+2\sqrt{x}\ge 1+x\\
        & \implies (1+\sqrt{x})^2\ge (1+x)\\
        & \implies 2\ln(1+\sqrt{x})\ge \ln(1+x)\\
        & \implies \sqrt{2}\ln(1+\sqrt{x})\ge \frac{1}{\sqrt{2}}\ln(1+x)
    \end{align*}
    which implies the result.
\end{proof}

\begin{lemma}\label{lem:psi-bd-large-x}
    For $x\in(1,\infty)$, we have 
    \[\frac{1}{\sqrt{2}}\ln(1+x)\le \Psi'(\Psi^{-1}(x))\le \frac{1}{\ln 2}\ln(1+x)\]
\end{lemma}

\begin{proof}
    Before we proceed with the proof, recall that \[\Psi^{-1}(x)=\frac{x-1}{W((x-1)/e)}-1\]
    It is therefore easy to see that \[\Psi'(\Psi^{-1}(x))=\ln(1+\Psi^{-1}(x))=1 + W((x-1)/e)\]
    Combining this with Lemma \ref{lem:lambert-bd}, we obtain
    \begin{align*}
        & 1+\left(1-\frac{1}{e}\right)\ln(1+(x-1)/e)\le \Psi'(\Psi^{-1}(x))\le 1+\ln(1+(x-1)/e)\\
         \Longleftrightarrow &\frac{1}{e}+\left(1-\frac{1}{e}\right)\ln(x+e-1)\le \Psi'(\Psi^{-1}(x))\le \ln(x+e-1)
    \end{align*}
    Having observed that ${\ln(x+e-1)}/{\ln(1+x)}$ is decreasing for $x>1$, it is immediate that ${\ln(x+e-1)}/{\ln(1+x)}\le {1}/{\ln 2} $.\\

    Employing the techniques used in the proof of Lemma \ref{lem:psi-bd-small-x}, we can see that proving that $\Psi'(\Psi^{-1}(x))\ge \ln(1+x)/\sqrt{2}$ is equivalent to showing that $$\xi(x)=\frac{1}{\sqrt{2}}(1+x)^{1/\sqrt{2}}\ln(1+x)-(1+x)^{1/\sqrt{2}}+1-x\le 0$$
    Note that since $\xi'(x)={\ln(1+x)}/(2{(1+x)^{1-\frac{1}{\sqrt{2}}}})-1$ and $\xi(0)=0$, it is enough to show that \[\xi_1(x)=\ln(1+x)-2(1+x)^{1-1/\sqrt{2}}\le 0\]
    To that end, it would be useful to observe that the function \[\xi_1'(x)=\frac{1}{1+x}-\frac{2(1-1/\sqrt{2})}{(1+x)^{1/\sqrt{2}}}\]
    has a unique root at the point $x_0=\left(1+{1}/{\sqrt{2}}\right)^{2+\sqrt{2}}-1$. Having checked the maximum value of $\xi_1$ is $\xi_1(x_0)<0$, we may conclude that $\xi_1(x)\le 0$ for $x\in (1,\infty)$. This establishes the lemma.
\end{proof}

\begin{lemma}\label{lem:union-bd-inc-bd}
    For $x\in (0,1)$ and $p\ge 1$,
    \[px-\binom{p}{2}x^2\le 1-(1-x)^p\le px\]
\end{lemma}
\begin{proof}
    Firstly, equality holds trivially for the case $p=1$. So, it suffices to show the result for $p\ge 2$. Consider the functions
    \[\nu_1(x)=1-(1-x)^p-px\] and  \[\nu_2(x)=1-(1-x)^p-px+\binom{p}{2}x^2\]
    Now, $\nu_1'(x)=p(1-x)^{p-1}-p\le 0$ Hence, $\nu_1$ is increasing and $\nu_1(0)=0$ which implies that $\nu_1(x)\le 0$ for all $x\in (0,1)$. Similarly, $\nu_2'(x)=p(1-x)^{p-1}-p+p(p-1)x$ and $\nu_2''(x)=-p(p-1)(1-x)^{p-2}+p(p-1)\ge 0$. Note that $\nu_2'(0)=0=\nu_2(0)$. This coupled with the fact that clearly, $\nu_2'(x)$ is an increasing function, implies the result.
\end{proof}

\begin{lemma}\label{lem:inc-exc}
    Consider real numbers $p_1, p_2, \cdots, p_n$ such that $0\le p_i\le 1$ for each $i=1,2,\cdots,n$. Then 
    \[\sum_{i=1}^n p_i - \sum_{i=1}^{n-1}\sum_{j=i+1}^n p_ip_j\le 1-\prod_{i=1}^n(1-p_i)\le \sum_{i=1}^n p_i\]
\end{lemma}

\begin{proof}
    We will prove the result by induction. It is easy to see that for $n=1$, equality holds throughout and for $n=2$, equality holds on for the first inequality while there is strict inequality for the second one if $p_1p_2>0$. Let the lemma hold for $n=m$ i.e.
    \[\sum_{i=1}^m p_i - \sum_{i=1}^{m-1}\sum_{j=i+1}^m p_ip_j\le 1-\prod_{i=1}^m(1-p_i)\le \sum_{i=1}^m p_i\]
    Now, for $n=m+1$,
    \begin{align*}
        1-\prod_{i=1}^{m+1}(1-p_i) &= 1- (1-p_{m+1})\prod_{i=1}^{m}(1-p_i)\\
        &=1-\prod_{i=1}^m(1-p_i)+p_{m+1}\prod_{i=1}^m (1-p_i)\\
        &\le \sum_{i=1}^m p_i+p_{m+1}\prod_{i=1}^m (1-p_i)\\
        &\le \sum_{i=1}^m p_i+p_{m+1}=\sum_{i=1}^{m+1} p_i
    \end{align*}
On the other hand, 
\begin{align*}
    1-\prod_{i=1}^{m+1}(1-p_i) &=1- (1-p_{m+1})\prod_{i=1}^{m}(1-p_i)\\
    &= 1-\prod_{i=1}^m(1-p_i)+p_{m+1}\prod_{i=1}^m (1-p_i)\\
    & \ge \sum_{i=1}^m p_i-\sum_{i=1}^{m-1} \sum_{j=i+1}^m p_ip_j+p_{m+1}\prod_{i=1}^m (1-p_i)\\
    & \ge \sum_{i=1}^m p_i-\sum_{i=1}^{m-1} \sum_{j=i+1}^m p_ip_j+p_{m+1}\left(1-\sum_{i=1}^m p_i\right)\\
    & =\sum_{i=1}^m p_i-\sum_{i=1}^{m-1} \sum_{j=i+1}^m p_ip_j+p_{m+1}-p_{m+1}\sum_{i=1}^m p_i\\
    & =\sum_{i=1}^{m+1} p_i-\sum_{i=1}^{m} \sum_{j=i+1}^{m+1}p_ip_j
\end{align*}
Hence, by induction, the lemma holds.
\end{proof}

\begin{lemma}\label{lem:benn-psi-inv-bd}
    Let $\Psi(x)=(1+x)\ln(1+x)-x$. Then $g(x)\le \Psi^{-1}(x)\le 2g(x)$ where 
    \[ g(x)=
    \begin{cases}
        \sqrt{2x}& \text{, if }x\in [0,e]\\
        \frac{x-1}{\ln(1+(x-1)/e)}-1 & \text{, if }x>1
    \end{cases}
    \]
\end{lemma}
\begin{proof}This result is motivated by Lemma 6.3 of \cite{MR3707425}. We first consider the case where $x\in[0,e]$. Proving the lemma in this case is equivalent to proving that for $0\le x\le e$
\[\Psi(\sqrt{2x})\le x\le \Psi(2\sqrt{2x})\]
Note that the inequality reduces to an equality when $x=0$. So, it is enough for us to consider the case $0<x\le e$. To that end, let 
\[\alpha(x)=\Psi(\sqrt{2x})-x\]
Then 
\[\alpha'(x)=\frac{\ln(1+\sqrt{2x})}{\sqrt{2x}}-1\]
We know that $\ln(1+y)/y$ is decreasing in $y$ and $\lim_{y\to 0}\ln(1+y)/y=1$. Hence it follows that $\alpha'(x)$ is negative for all $x>0$. Consequently, $\alpha(x)$ is decreasing in $x$ and attains the value 0 at $x=0$. Hence, $\alpha(x)$ must be negative for all $x>0$ which in turn implies that $\Psi(\sqrt{2x})\le x$ for all $x>0$.\\

As for the other part i.e. $\Psi(2\sqrt{2x})\ge x$, we consider the function \[\beta(x)=\Psi(2\sqrt{2x})-x\]
Observe that \[\beta'(x)=\frac{4}{2\sqrt{2x}}\ln(1+2\sqrt{2x})-1\] Once again, $\beta'$ is decreasing and attains its minimum value on $(0,e]$ at $x=e$. But \[\beta'(e) = \frac{\sqrt{2}}{\sqrt{e}}\ln(1+2\sqrt{2e})-1>0\]
Hence $\beta$ is increasing on $(0,e]$ and attains the value 0 at $x=0$. Hence it is established that on $(0,e]$, $\Psi(2\sqrt{2x})\ge x$.\\

    To perform the analysis on the case $x>1$, we first find the exact form of the inverse of $\Psi$ in terms of the Lambert function $W$. 
    \begin{align*}
        & \Psi(\Psi^{-1}(x))=(1+\Psi^{-1}(x))\ln(1+\Psi^{-1}(x))-\Psi^{-1}(x)\\
        & \implies x-1 = (1+\Psi^{-1}(x))(\ln(1+\Psi^{-1}(x))-1)\\
        & \implies\frac{x-1}{e} = \frac{(1+\Psi^{-1}(x))}{e}\ln \left(\frac{1+\Psi^{-1}(x)}{e}\right)\\
        & \implies W\left(\frac{x-1}{e}\right)\exp\left(W\left(\frac{x-1}{e}\right)\right) = \frac{1+\Psi^{-1}(x)}{e}\ln \left(\frac{1+\Psi^{-1}(x)}{e}\right)\\
        & \implies \frac{1+\Psi^{-1}(x)}{e}=\exp\left(W\left(\frac{x-1}{e}\right)\right)\\
        & \implies \Psi^{-1}(x) = \frac{x-1}{W(\frac{x-1}{e})}-1
    \end{align*}

    Using Lemma \ref{lem:lambert-bd}, we have for $x>1$
    \begin{equation}\label{psi-inv-bd}
        \frac{x-1}{\ln(1+(\frac{x-1}{e}))}-1 \le \Psi^{-1}(x)\le \frac{e}{e-1}\frac{x-1}{\ln(1+(\frac{x-1}{e}))}-1 
    \end{equation}

    Consider the functions $\ell(x)$ and $u(x)$ where  
    \begin{equation*}
        \ell(x) = \frac{x-1}{\ln(1+(\frac{x-1}{e}))}-1\text{ and }u(x)=\frac{e}{e-1}\frac{x-1}{\ln(1+(\frac{x-1}{e}))}-1
    \end{equation*}
    Then 
    \begin{align*}
        &\frac{u(x)}{\ell(x)}\le 2\\
        \implies &u(x)\le 2 \ell(x)\\
         \implies &2\left(\frac{x-1}{\ln(1+(\frac{x-1}{e}))}-1\right)\ge \frac{e}{e-1}\frac{x-1}{\ln(1+(\frac{x-1}{e}))}-1\\
        \implies &2\left(1-\frac{1}{e}\right)(x-1)-\left(1-\frac{1}{e}\right)\ln\left(1+\frac{x-1}{e}\right)\ge (x-1)-\left(1-\frac{1}{e}\right)\ln\left(1+\frac{x-1}{e}\right)\\
         \implies &\left(1-\frac{2}{e}\right)(x-1)\ge \ln\left(\frac{x-1}{e}+1\right)\\
    \end{align*}
    Hence it suffices to show that $f(x)=\left(1-\frac{2}{e}\right)(x-1)-\ln\left(1+{(x-1)}/{e}\right)\ge 0$ for all $x>1$. This follows easily from the fact that $f(1)=0$ and $f'(x)={(e-2)}/{(e-1)}-{1}/{(x+e-1)}>0$ for all $x>1$.
\end{proof}

\begin{lemma}\label{lem:psi-der-lower}
    For $0\le x\le e$
    \[\Psi'(\Psi^{-1}(x))\ge \ln(1+\sqrt{x})\]
\end{lemma}
\begin{proof}
    By Lemma \ref{lem:benn-psi-inv-bd}, we observe that for $0\le x\le e$
    \[\Psi^{-1}(x)\ge \sqrt{2x}\ge \sqrt{x}\]
    It immediately follows that 
    \[\Psi'(\Psi^{-1}(x))=\ln(1+\Psi^{-1}(x))\ge \ln(1+\sqrt{x})\] and the lemma is established.
\end{proof}

\begin{lemma}\label{lem:log-lower-bd}
    For $0\le x\le e$,
    \[\ln(1+x)\ge \frac{x}{3}\]
\end{lemma}
\begin{proof}
    $\ln(1+x)/x$ is a decreasing function and attains its minimum value on $[0,e]$ at $x=e$. Hence
    \[\frac{\ln(1+x)}{x}\ge \frac{\ln(1+e)}{e}\ge \frac{\ln e}{e}=\frac{1}{e}> \frac{1}{3}\].
\end{proof}

\begin{lemma}\label{lem:aux-log}
    For $x\ge 1/2$, 
    \[\ln 2x\le \frac{5}{4}\ln\left(\frac{2x}{\sqrt{\ln 2x}}\right)\le 4\ln\left(\frac{2x}{\sqrt{\ln 2x}}\right)\]
\end{lemma}
\begin{proof}
    To show that the lemma holds, it suffices to prove that \[\alpha(x)=\frac{5}{4}\ln\left(\frac{2x}{\sqrt{\ln 2x}}\right)--\ln 2x = \frac{1}{4}\ln 2x - \frac{5}{8}\ln \ln 2x\ge 0\]
    Observe that \[\alpha'(x)=\frac{2\ln 2x-5}{8x\ln 2x}\]
    It is evident that $\alpha(x)$ is increasing for $x>\exp(5/2)/2$, decreasing for $x<\exp(5/2)/2$ with $x_0=\exp(5/2)/2$ being the unique minimum. However, at $x_0=\exp(5/2)/2$, \[\alpha(x_0)=\frac{5}{8}-\frac{5}{8}\ln 2>0\]. Hence $\alpha(x)$ must be non-negative and the lemma is established.
\end{proof}

\begin{lemma}\label{lem:bin-bd-norm}
    Let $B^2\ln (2p)\le n\sigma^2$ and consider the random variable $W$ such that \[W\sim Bin\left(n, \frac{\sigma^2}{\sigma^2+B^2}\right)\]
    Then for $t = \sigma\sqrt{\ln(2p/\sqrt{\ln (2p)})}/\sqrt{en}$, we have \[\mathbb{P}\left(W\ge \frac{n\sigma^2}{\sigma^2+B^2}+\frac{nBt}{\sigma^2+B^2}\right)\ge c\Phi\left(-\frac{\sqrt{2n}t}{\sigma}\right)\]
    where $c<1$ is a universal constant and $\Phi(.)$ is the normal distribution function.
\end{lemma}
\begin{proof}
    The first thing to be noted is that \[\mathbb{P}\left(W\ge \frac{n\sigma^2}{\sigma^2+B^2}+\frac{nBt}{\sigma^2+B^2}\right)=\mathbb{P}\left(W\ge \left\lceil\frac{n\sigma^2}{\sigma^2+B^2}+\frac{nBt}{\sigma^2+B^2}\right\rceil\right)\] where $\lceil.\rceil$ is the ceiling function. Consequently, 
    \[\left\lceil\frac{n\sigma^2}{\sigma^2+B^2}+\frac{nBt}{\sigma^2+B^2}\right\rceil = \frac{n\sigma^2}{\sigma^2+B^2}+\frac{nBt^*}{\sigma^2+B^2}\]
    where $t\le t^*<t + (\sigma^2+B^2)/(nB)$. By \cite{MR3196787}, we have 
    \[
    \mathbb{P}\left(W\ge \left\lceil\frac{n\sigma^2}{\sigma^2+B^2}+\frac{nBt}{\sigma^2+B^2}\right\rceil\right)\ge \Phi\left(-\sqrt{2nH_{p_{\sigma,B}}\left(\frac{k}{n},p_{\sigma,B}\right)}\right),
    \]
    where $H_p(a,p) = a\ln(a/p)+(1-a)\ln((1-a)/(1-p))$, $p_{\sigma,B} = \sigma^2/(\sigma^2+B^2)$ and $k = (n\sigma^2+nBt^*)/(\sigma^2+B^2)$. Moreover, $H_p(a,p)$ is increasing in $a$ for $a\ge p$. By Lemma 6.3 of \cite{csiszar2006context}, 
    \[
    H_p(a,p)\le \frac{(a-p)^2}{p(1-p)},
    \]
    whereby it suffices to show that 
    \[\Phi\left(-\sqrt{2n}\left(\frac{t}{\sigma}+\frac{\sigma^2+B^2}{n\sigma B}\right)\right)\ge c\Phi\left(-\sqrt{2n}\frac{t}{\sigma}\right)\] for some $c<1$. Lemma \ref{lem:aux-log} combined with an application of Lemma \ref{lem:mills-general}, gives us 
    \begin{equation}\label{almost-mills}
        \frac{1}{\sqrt{2\pi}}\frac{e^{-x^2/2}}{x+1}\le \Phi(-x)\le \frac{1}{\sqrt{2\pi}}\frac{e^{-x^2/2}}{x}
    \end{equation}
    By \eqref{almost-mills}, 
    \begin{equation}\label{normal-dist}
        \Phi\left(-\sqrt{2n}\left(\frac{t}{\sigma}+\frac{\sigma^2+B^2}{n\sigma B}\right)\right)\ge \frac{1}{\sqrt{2\pi}}\frac{1}{1+\sqrt{2n}\left(\frac{t}{\sigma}+\frac{\sigma^2+B^2}{n\sigma B}\right)}\exp\left(-n\left(\frac{t}{\sigma}+\frac{\sigma^2+B^2}{n\sigma B}\right)^2\right)
    \end{equation}
    Since $B^2\ln (2p)\le en\sigma^2$, we can write \[\sqrt{\frac{\ln (2p)}{ne}}\le \frac{\sigma}{B}\le 1\]
    As a consequence, 
    \begin{align}
        n\left(\frac{t}{\sigma}+\frac{\sigma^2+B^2}{n\sigma B}\right)^2 & = \frac{nt^2}{\sigma^2}+\frac{2t}{\sigma}\frac{\sigma^2+B^2}{B\sigma}+\frac{1}{n}\left(\frac{\sigma^2+B^2}{\sigma B}\right)^2 \nonumber \\
        & = \frac{nt^2}{\sigma^2}+\frac{2t}{\sigma}\left(\frac{\sigma}{B}+\frac{B}{\sigma}\right)+\frac{1}{n}\left(\frac{\sigma}{B}+\frac{B}{\sigma}\right)^2 \nonumber \\
        & \le \frac{nt^2}{\sigma^2}+\frac{2\sqrt{\ln(2p/\sqrt{\ln (2p)})}}{\sqrt{ne}}\left(1+\frac{B}{\sigma}\right)+\frac{1}{n}\left(1+\frac{B}{\sigma}\right)^2 \label{lem-b6-eq-1}\\
        & \le \frac{nt^2}{\sigma^2}+\frac{2\sqrt{2\ln (2p)}}{\sqrt{ne}}\left(1+\frac{\sqrt{ne}}{\sqrt{\ln (2p)}}\right)+\frac{1}{n}\left(1+\frac{\sqrt{ne}}{\sqrt{\ln (2p)}}\right)^2\label{lem-b6-eq-2}\\
        & \le \frac{nt^2}{\sigma^2}+\frac{2\sqrt{2\ln (2p)}}{\sqrt{ne}}+2+\left(\frac{1}{\sqrt{n}}+\frac{\sqrt{e}}{\sqrt{\ln (2p)}}\right)^2\nonumber\\
        & \le \frac{nt^2}{\sigma^2}+5+\left(1+\frac{\sqrt{e}}{\sqrt{\ln 2}}\right)^2 \label{lem-b6-eq-3}
    \end{align}
    \eqref{lem-b6-eq-2} follows by observing the simple fact that for all $p\ge 1$, $\ln(2p/\sqrt{\ln (2p)})\le 2\ln (2p)$ while \eqref{lem-b6-eq-3} follows by observing that $\sqrt{\ln (2p)/ne}\le 1$ and $n\ge 1, p\ge 1$. Taking \[\ln \tau = -5-\left(1+\frac{\sqrt{e}}{\sqrt{\ln 2}}\right)^2\] we observe that \[\exp\left(-n\left(\frac{t}{\sigma}+\frac{\sigma^2+B^2}{n\sigma B}\right)^2\right)\ge \tau \exp\left(-\frac{nt^2}{\sigma^2}\right)\]
    For $t<(\sigma^2+B^2)/(nB)$, we must have \[\sqrt{2n}\left(\frac{\sigma^2+B^2}{n\sigma B}\right)\le \sqrt{\frac{2}{n}}\frac{\sigma}{B}+\sqrt{\frac{2}{n}}\frac{B}{\sigma}\le \sqrt{2}+\sqrt{\frac{2}{n}}\frac{\sqrt{ne}}{\sqrt{\ln (2p)}}\le \sqrt{2}+\frac{\sqrt{2e}}{\sqrt{\ln 2}}=\alpha\]
    Then 
    \[1 + \sqrt{2n}\left(\frac{t}{\sigma}+\frac{\sigma^2+B^2}{n\sigma B}\right)\le 1+\alpha+\frac{t\sqrt{2n}}{\sigma}\le \frac{t\sqrt{2n}}{\sigma}\left(1+{f(1+\alpha)}\right)\]
    where $\sqrt{e}/f = \min_{p\ge 1}\sqrt{2\ln(2p/\sqrt{\ln (2p)})}$. As a result, 
    \[\Phi\left(-\sqrt{2n}\left(\frac{t}{\sigma}+\frac{\sigma^2+B^2}{n\sigma B}\right)\right)\ge \frac{\tau}{\left(1+f(1+\alpha)\right)}\frac{\sigma}{t\sqrt{2n}}\exp\left(-\frac{nt^2}{\sigma^2}\right)\ge \frac{\tau}{\left(1+f(1+\alpha)\right)}\Phi\left(-\frac{t\sqrt{2n}}{\sigma}\right)\]
    For $t\ge (\sigma^2+B^2)/(nB)$, 
    \[1+\sqrt{2n}\left(\frac{t}{\sigma}+\frac{\sigma^2+B^2}{n\sigma B}\right)\le \frac{t\sqrt{2n}}{\sigma}\left(2+\frac{f}{\sqrt{2e}}\right)=\frac{\beta t\sqrt{2n}}{\sigma}\]
    Hence in this case, 
    \[\Phi\left(-\sqrt{2n}\left(\frac{t}{\sigma}+\frac{\sigma^2+B^2}{n\sigma B}\right)\right)\ge \frac{\tau}{\beta}\Phi\left(-\frac{t\sqrt{2n}}{\sigma}\right)\]
    Taking $\gamma = \max\left\{1+f(1+\alpha),\beta\right\}$,  we have the desired result 
    \[\Phi\left(-\sqrt{2n}\left(\frac{t}{\sigma}+\frac{\sigma^2+B^2}{n\sigma B}\right)\right)\ge \frac{\tau}{\gamma}\Phi\left(-\frac{t\sqrt{2n}}{\sigma}\right)\]
\end{proof}

\begin{lemma}\label{lem:bin-th-gen}
    For $x\in [0,1)$ and $p\ge 1$,
    \[(1-x)^{1/p}-\left(1-\frac{x}{p(1-x)}\right)\ge 0\]
\end{lemma}
\begin{proof}
    Taking $\mu(x)=(1-x)^{1/p}-\left(1-\frac{x}{p(1-x)}\right)$, observe that 
    \[\mu'(x)=\frac{1-(1-x)^{1+1/p}}{p(1-x)^2}\ge 0\]
    This combined with the fact that $\mu(0)=0$ implies the result.
\end{proof}

\begin{lemma}\label{lem:aux-log-2}
    For $x>0$, $y>0$ and $xy>1$, we must have 
    \[\frac{(xy-1)/e}{\ln(1+(xy-1)/e)}\le \frac{(1+x)y}{\ln((1+x)y)}\]
\end{lemma}
\begin{proof}
    Note that it is enough to show that
    \[\frac{(xy-1)/e}{\ln(1+(xy-1)/e)}\le \frac{(1+x)y-1}{\ln((1+x)y)}\]
    Now, since both the quantities $(xy-1)/e$ and $(1+x)y-1$ are positive, our result would follow immediately by the fact that $x/\ln(1+x)$ is increasing for $X>0$ if we can show that 
    \[(xy-1)/e\le (1+x)y-1\]
    Note that we can write the following if and only if statements
    \begin{align}
        &(xy-1)/e\le (1+x)y-1 \nonumber \\
        & \Longleftrightarrow xy-1 \le exy+ey-e \nonumber \\
        & \Longleftrightarrow exy-xy+ey-e+1\ge 0 \nonumber \\
        & \Longleftrightarrow (e-1)(xy-1)+ey\ge 0
    \end{align}
    the last of which is true. Hence, the result follows.
\end{proof}

\begin{lemma}\label{lem:benn-type}
    Let $Y_1, \ldots, Y_n$ be independent, mean zero random variables such that \[|Y_i|\le K,~~\frac{1}{n}\sum_{i=1}^n \E[|Y_i|^q]\le M^q, ~~\frac{1}{n}\sum_{i=1}^n\E[Y_i^2]\le \sigma^2.\] Then
\[
\frac{1}{n} \sum_{i=1}^n \mathbb{E}[|Y_i|^s] \leq
\begin{cases}
\left( \sigma^2 / M^2 \right)^{\frac{q}{q-2}} (M^q / \sigma^2)^{s/(q-2)}, & \text{if } 2 \leq s \leq q, \\
K^{s-q} M^q, & \text{if } s \geq q.
\end{cases}
\]
If $M^q = K^{q-2} \sigma^2$, then $K^{q-2} / M^{q-2} = M^2 / \sigma^2$ and hence $\left( \sigma^2 / M^2 \right)^{\frac{q}{q-2}} = (M/K)^q$. This implies that if $M^q = K^{q-2} \sigma^2$, then for all $s \geq 2$, the above bound is $K^{s-q} M^q$. This moment bound is exactly same as used in Bennett's inequality.
\end{lemma}
\begin{proof}
    Because $|Y_i| \leq K$, we get that for $s \geq q$,
\[
\frac{1}{n} \sum_{i=1}^n \mathbb{E}[|Y_i|^s] \leq K^{s-q} \frac{1}{n} \sum_{i=1}^n \mathbb{E}[|Z_i|^q] \leq K^{s-q} M^q.
\]
Using the inequality $(q-2)|Y_i|^s \leq (s-2) a^{s-q} |Y_i|^q + (q-s) a^{s-2} |Y_i|^2$ for $2 \leq s \leq q$ and any positive $a$, we get using Jensen's inequality,
\[
\frac{1}{n} \sum_{i=1}^n \mathbb{E}[|Y_i|^s] \leq \frac{s-2}{q-2} a^{s-q} \frac{1}{n} \sum_{i=1}^n \mathbb{E}[|Y_i|^q] + \frac{q-s}{q-2} a^{s-2} \frac{1}{n} \sum_{i=1}^n \mathbb{E}[|Y_i|^2]
\]
\[
\leq \frac{s-2}{q-2} a^{s-q} M^q + \frac{q-s}{q-2} a^{s-2} \sigma^2.
\]

Minimizing over $a > 0$, we choose $a = (M^q / \sigma^2)^{1/(q-2)}$ and this yields
\[
\frac{1}{n} \sum_{i=1}^n \mathbb{E}[|Z_i|^s] \leq M^{q(s-2)/(q-2)} \sigma^{2(q-s)/(q-2)}
\]
\[
= \left( \frac{\sigma^2}{M^2} \right)^{\frac{q}{q-2}} \left( \frac{M^q}{\sigma^2} \right)^{s/(q-2)},
\]
for $2 \leq s \leq q$. This completes the proof.
\end{proof}

\begin{lemma}\label{lem:c-1}
    For $x>\sqrt{e}$,
    \[\frac{1}{6\ln x}\le \frac{1}{3\ln(1+x)}\le \frac{\Psi(x)}{(1+x)\ln^2(1+x)}\le \frac{1}{\ln(1+x)}\le \frac{1}{\ln x}\]
\end{lemma}
\begin{proof}
    The upper bounds are a bit easier to establish, so let us prove them first. Recall that $\Psi(x)=(1+x)\ln(1+x)-x$ and hence
    \begin{equation}\label{eq:c-1-1}
        \frac{(1+x)\ln(1+x)-x}{(1+x)\ln^2(1+x)}\le \frac{1}{\ln(1+x)}-\frac{x}{(1+x)\ln^2(1+x)}\le \frac{1}{\ln(1+x)}\le \frac{1}{\ln x}
    \end{equation}
    Now, we focus on establishing the lower bounds. In order to show that \[\frac{\Psi(x)}{(1+x)\ln^2(1+x)}\ge \frac{1}{3\ln(1+x)}\]
    This is equivalent to showing that \[\xi(x)=\frac{2}{3}\Psi(x)-\frac{x}{3}\ge 0\] for all $x\ge \sqrt{e}$. This is fairly easy to see once we have taken note of the fact that $\xi'(x)=\frac{2}{3}\ln(1+x)-\frac{1}{3}$ is increasing and $\xi'(\sqrt{e})>0$ and $\xi(\sqrt{e})>0$. Once, this is established, note that for $x>\sqrt{e}>(\sqrt{5}+1)/2$, the function $x^2-x-1$ is positive and increasing. Therefore for $x>\sqrt{e}$, we have 
    \[x^2>1+x\implies 2\ln x>\ln(1+x)\implies 6\ln x>3\ln(1+x)\implies \frac{1}{6\ln x}\le \frac{1}{3\ln(1+x)}\] and we are done.
\end{proof}

\begin{lemma}\label{lem:c-2}
    Consider the equation \[\frac{x^q}{q\ln x}=c\] where $x$ is known to be larger than $\sqrt{e}$ and $q\ge 1$. Then $\left(c\ln c\ln c\ln c\ldots\right)^{1/q}$ is a solution to the equation. Further this solution is $\mathcal{O}((c\ln c)^{1/q})$.
\end{lemma}
\begin{proof}
    Consider first the equation \[\frac{y}{\ln y}=r\]
    where $r>e$. Define the sequence \[f_n=r\ln f_{n-1}\] and $f_1=r$. It is easy to see that this is an increasing sequence in $n$. Further $f_n\le r^2$ for all $n$, by induction, since $f_{n-1}\le r^2$ implies that $f_n=r\ln f_{n-1}=r\ln r^2=2r\ln r\le r^2$. As $f_n$ is a potive, increasing sequence that is bounded above, the limit $f_{\infty}=\lim_{n\to \infty} f_n=r\ln r\ln r\ln r\ldots$ exists finitely. \\
    
    Consider, now, the original equation. There, $x^q$ plays the role of $y$. Moreover, the left hand side is increasing in $x$ and if we can show that $(\sqrt{e})^q/q\ln\sqrt{e}>e$ for all $q\ge 1$, then we have established the first part of the lemma. Observe that for the function $g(x)=x/2-\ln\left(x/2\right)$, minima occurs at the point $x=2$ and the resultant minimum value is 1. Therefore $2e^{x/2}/x$ takes a minimum value of $e^1=e$. So, we can say that the equation \[\frac{x^q}{\ln x}=c\] has $\left(c\ln c\ln c\ln c\ldots\right)^{1/q}$ as a solution.\\

    As for the second part of the result, observe that since the sequence $f_n$ is increasing, therefore $r\ln r\le f_{\infty}$ and since each term in the sequence $f_n<r^2$, therefore $f_{\infty}\le r^2$. Consequently, $f_{\infty} = r\ln f_{\infty}=r\ln r^2\le 2r\ln r$. Therefore
    \[(c\ln c)^{1/q}\le (c\ln c\ln c\ln c\ldots)^{1/q}\le (2c\ln c)^{1/q}.\]
\end{proof}

\begin{lemma}\label{lem:c-3}
    Consider the function \[f(x)=-xW\left(-\frac{1}{x}\right)=\exp\left(-W\left(-\frac{1}{x}\right)\right)\] for $x\ge e$. Then $1\le f(x)\le e$.
\end{lemma}

\begin{proof}
    The way we go about proving this result is showing that the function $f$ is a decreasing function. Observe that since the Lambert function is strictly increasing $\forall$ $x\ge -1/e$, we have \[f'(x)=-\exp\left(-W\left(-\frac{1}{x}\right)\right)W'\left(-\frac{1}{x}\right)\frac{1}{x^2}<0.\]
    Moreover, $f(e)=e$ and \[\lim_{x\to \infty}f(x)=\lim_{y\to 0^-}\frac{W(y)}{y}=\lim_{y\to 0^-}e^{-W(y)}=1\]
    and we are done.
\end{proof}

\begin{lemma}\label{lem:calculation-of-T-part1}
        Suppose $a, b > 0$. Then
        \[
        \inf_{t > 0} \frac{a(e^{bt} - 1 - bt) + x}{t} = ab\Psi^{-1}(x/a).
        \]
    \end{lemma}
    \begin{proof}
        The infimum is attained at $t^*$ satisfying
        \[
        t\left[abe^{bt} - ab\right] - ae^{bt} + a + abt = x
        \]
        This is equivalent to
        \begin{equation}\label{eq:main-equation-part-1}
        bte^{bt} - e^{bt} + 1 = x/a\quad\Leftrightarrow\quad e^{bt}(bt - 1) + 1 = x/a.
        \end{equation}
        Taking $y = e^{bt}-1$, this equation is same as $(y+1)\ln(y+1) - y = x/a$, or equivalently, $\Psi(y) = x/a$. Hence, 
        \[
        t^* = \frac{1}{b}\ln\left(1 + \Psi^{-1}(x/a)\right).
        \]
        Observe that from~\eqref{eq:main-equation-part-1}
        \[
        \frac{a(e^{bt^*} - 1 - bt^*) + x}{t^*} = \frac{a(e^{bt^*} - 1 - bt^*) + abt^*e^{bt^*} - ae^{bt^*} + a}{t^*} = ab(e^{tb^*} - 1) = ab\Psi^{-1}(x/a). 
        \]
    \end{proof}
    \begin{lemma}\label{lem:calculation-of-T-part2}
        Suppose $a, b > 0$ and $m \ge 2$. Then
        \[
        \inf_{t > 0}\,\frac{at^{m}e^{tb} + x}{t} \le \frac{2bx}{mW(bx^{1/m}/(2ma^{1/m}))}.
        \]
    \end{lemma}
    \begin{proof}
        The infimum is attained at $t^*$ satisfying
        \[
        t(amt^{m-1}e^{tb} + abt^me^{tb}) - (at^me^{tb} + x) = 0.
        \]
        This is equivalent to
        \begin{equation}\label{eq:part-2-main-eq}
        a(m-1)t^me^{tb} + abt^{m+1}e^{tb} = x
        \end{equation}
        Consider $t(k, y)$ solving the equation $t^ke^{tb} = y$. It is clear that
        \[
        t(k, y)e^{t(k,y)b/k} = y^{1/k}\quad\Leftrightarrow\quad 
        t(k,y) = \frac{k}{b}W\left(by^{1/k}/k\right).
        \]
        From~\eqref{eq:part-2-main-eq}, we consider two values of $t$ obtained as solutions to the equations $a(m-1)t^{m}e^{tb} = x$ and $abt^{m+1}e^{tb} = x$, and use
        \[
        \inf_{t > 0}\,\frac{at^{m}e^{tb} + x}{t} \le x\min\left\{\frac{(m/(m-1))}{t(m, x/(a(m-1))},\,\frac{1 + bt(m+1, x/(ab))}{bt^2(m+1, x/(ab))}\right\}.
        \]
        For $m > 2$, $m/(m-1) \le 2$. Also, 
        \[
        t(m, x/(a(m-1))) = \frac{m}{b}W\left(\frac{b(x/a)^{1/m}}{m(m-1)^{1/m}}\right) ~\ge~ \frac{m}{b}W\left(\frac{bx^{1/m}}{2ma^{1/m}}\right).
        \]
        Similarly, 
        \[
        t(m+1,x/(ab)) = \frac{m+1}{b}W\left(\frac{bx^{1/(m+1)}}{(m+1)(ab)^{1/(m+1)}}\right) \ge \frac{m}{b}W\left(\frac{b^{m/(m+1)}x^{1/(m+1)}}{2ma^{1/(m+1)}}\right).
        \]
    \end{proof}

\end{document}